\documentclass[11pt,intlimits]{amsart}

\usepackage{amsmath}
\usepackage{amssymb}
\usepackage{amsbsy}
\usepackage{amsfonts}
\usepackage{amscd}
\usepackage{amsthm}
\usepackage[all]{xy}
\usepackage{fullpage}
\usepackage{color}
\usepackage{ifthen}

\usepackage[backref]{hyperref}
\usepackage[alphabetic,backrefs,lite,nobysame]{amsrefs}

\theoremstyle{plain}
\newtheorem{theorem}[equation]{Theorem}

\newtheorem{lemma}[equation]{Lemma}

\newtheorem{prop}[equation]{Proposition}
\newtheorem{cor}[equation]{Corollary}

\newtheorem{conjecture}[equation]{Conjecture}

\theoremstyle{definition}
\newtheorem{defn}[equation]{Definition}

\theoremstyle{remark}
\newtheorem{remark}[equation]{Remark}

\numberwithin{equation}{subsection}

\newcommand\Q{{c}}		
\newcommand\s{{s}}		
\newcommand\Qp{{c_0}}	

\newcommand\EFq{{E}}

\newcommand\SnAQ{{S_{n,\Q}(A)}}

\newcommand\VM{{\Lambda_\rho}}

\newcommand\bs\bigskip
\newcommand\np\newpage

\newcommand\etale{{\'etale}}
\newcommand\loccit{{\it loc.~cit.}}

\newboolean{usecolor}
\setboolean{usecolor}{true}

\ifx\color@rgb\@undefined\else
	\definecolor{violet}{rgb}{0.25,0,0.75}
	\definecolor{green}{rgb}{0,0.7,0}
\fi

\newcommand{\defi}[1]{{\sl #1}}

\newcommand\bbA{\mathbb{A}}
\newcommand\bbB{\mathbb{B}}
\newcommand\bbC{\mathbb{C}}
\newcommand\bbE{\mathbb{E}}
\newcommand\bbF{\mathbb{F}}
\newcommand\bbG{\mathbb{G}}
\newcommand\bbQ{\mathbb{Q}}
\newcommand\bbP{\mathbb{P}}

\newcommand\bbT{\mathbb{T}}
\newcommand\bbZ{\mathbb{Z}}

\newcommand\CC{{\mathcal{C}}}
\newcommand\EE{{\mathcal{E}}}
\newcommand\FF{\mathcal{F}}
\newcommand\GG{\mathcal{G}}
\newcommand\HH{{\mathcal{H}}}
\newcommand\KK{\mathcal{K}}
\newcommand\LL{\mathcal{L}}
\newcommand\MM{\mathcal{M}}
\newcommand\PP{\mathcal{P}}
\newcommand\QQ{\mathcal{Q}}
\newcommand\RR{\mathcal{R}}
\newcommand\TT{\mathcal{T}}

\renewcommand\AA{\mathcal{A}}
\renewcommand\SS{{\mathcal{S}}}

\newcommand\Fq{\bbF_q}
\newcommand\Fpi{\bbF_\pi}
\newcommand\Fv{\bbF_v}
\newcommand\Finf{\bbF_\infty}

\newcommand\Qell{{\bbQ_\ell}}
\newcommand\Zell{{\bbZ_\ell}}

\newcommand\Kbar{\bar{K}}

\newcommand\Fqbar{{\bar\bbF_q}}

\renewcommand\k{\Fqbar}

\newcommand\bbQbar{{\bar{\mathbb{Q}}}}

\newcommand\Qbar{\bar\bbQ}
\newcommand\Qellbar{{\bar\bbQ_\ell}}
\newcommand\El{\Qellbar}

\newcommand\Ksep{K^{\mathrm{sep}}}
\newcommand\KS{K_S}
\newcommand\KSt{K_S^{\tame}}

\newcommand\kPt{{K}}
\newcommand\kPu{{K'}}

\newcommand\bbBt{\bbB^\times}
\newcommand\Qlbartimes{{\bar{\mathbb{Q}}^{\scriptstyle\times}_\ell}}

\newcommand\GammaOf[1]{{\Gamma(#1)}}

\newcommand\BQ{\GammaOf{\Q}}
\newcommand\Bt{\GammaOf{t}}
\newcommand\Bu{\GammaOf{u}}
\newcommand\BQFq{{(\bbF_q[t]/\Q\,\bbF_q[t])^\times}}

\renewcommand\l{\lambda}

\renewcommand\r{r}
\newcommand\R{{R}}

\newcommand\gl{\mathfrak{gl}}

\newcommand\Aut{\operatorname{Aut}}
\newcommand\End{{\operatorname{End}}}
\newcommand\Frob{{\operatorname{Frob}}}
\newcommand\Gal{{\operatorname{Gal}}}
\newcommand\GL{{\operatorname{GL}}}
\newcommand\Hom{{\operatorname{Hom}}}
\newcommand\Ind{{\operatorname{Ind}}}
\newcommand\Isom{\operatorname{Isom}}
\newcommand\Res{\operatorname{Res}}
\newcommand\SL{{\operatorname{SL}}}
\newcommand\Sp{{\operatorname{Sp}}}
\newcommand\Sym{{\operatorname{Sym}}}
\newcommand\Tr{{\operatorname{Tr}}}
\newcommand\Unip{{\operatorname{Unip}}}
\newcommand\Var{{\operatorname{Var}}}

\renewcommand\O{{\operatorname{O}}}

\newcommand\arith{{\operatorname{arith}}}
\newcommand\diag{{\operatorname{diag}}}
\newcommand\dr{{\operatorname{drop}}}
\newcommand\geom{{\operatorname{geom}}}
\newcommand\id{\operatorname{id}}
\newcommand\len{{\operatorname{len}}}
\newcommand\norm{\operatorname{Norm}}
\newcommand\rank{{\operatorname{rank}}}
\newcommand\std{{\operatorname{std}}}
\newcommand\swan{{\operatorname{Swan}}}
\newcommand\tame{{\operatorname{tame}}}
\newcommand\unip{{\operatorname{unip}}}

\renewcommand\div{{\operatorname{div}}}

\newcommand\Alg{\mathbf{Alg}}
\newcommand\Rep{\mathbf{Rep}}
\newcommand\Set{\mathbf{Set}}

\renewcommand\Vec{\mathbf{Vec}}

\newcommand\ulAut{\underline{\operatorname{Aut}}}

\newcommand\ulIsom{\underline{\operatorname{Isom}}}

\newcommand\Dbc[1]{{D^b_c(#1,\El)}}
\newcommand\Perv[1]{{\mathrm{Perv}(#1,\El)}}
\newcommand\Negl[1]{{\mathrm{Negl}(#1,\El)}}
\newcommand\Tann[1]{\mathrm{Tann}(#1,\Qellbar)}

\newcommand\Dual[1]{D#1}

\newcommand\ssm{\smallsetminus}
\newcommand\seq{\subseteq}
\newcommand\sub{\subset}

\newcommand\into\hookrightarrow
\newcommand\onto\twoheadrightarrow
\newcommand\longto\longrightarrow

\newcommand\gp[1]{\langle#1\rangle}

\newcommand\midstar{*_{\mathrm{mid}}}

\newcommand\wm[1]{{\mu_{#1}}}

\newcommand\LambdaLeg{\Lambda_{\mathrm{Leg}}}
\newcommand\Xleg{X_{\mathrm{Leg}}}

\newcommand\one{\mathbf{1}}
\newcommand\zero{\mathbf{0}}

\newcommand\OneSpace{1}
\newcommand\ZeroSpace{0}

\newcommand\GK{{G_K}}
\newcommand\GKC{{G_{K,\CC}}}
\newcommand\GKS{{G_{K,\SS}}}
\newcommand\GKSt{{G_{K,\SS}^{\tame}}}
\newcommand\GKSbar{{\bar{G}_{K,\SS}}}
\newcommand\GKR{{G_{K,\RR}}}
\newcommand\Gv{{G_v}}

\newcommand\piOne[1]{\pi_1(#1)}
\newcommand\piOneTame[1]{\pi_1^t(#1)}

\newcommand\piOneT{\piOne{T}}
\newcommand\piOneU{\piOne{U}}
\newcommand\piOneTbar{\piOne{\Tbar}}
\newcommand\piOneUbar{\piOne{\Ubar}}
\newcommand\piOneUpBar{\piOne{\Ubar'}}

\newcommand\CCp{{\CC'}}

\newcommand\LC{{L_\CC}}
\newcommand\MC{{M_\CC}}
\newcommand\PC[1]{{P_{\CC,#1}}}
\newcommand\rC{{r_\CC}}

\newcommand\Gm{\bbG_m}
\newcommand\GmBar{\bar\bbG_m}

\newcommand\Tbar{\bar{T}}
\newcommand\Ubar{\bar{U}}
\newcommand\Zbar{\bar{Z}}

\newcommand\Ponet{\bbP^1_t}
\newcommand\Poneu{\bbP^1_u}

\newcommand\PoneBar{{\bar\bbP^1}}
\newcommand\PonetBar{{\bar\bbP^1_t}}
\newcommand\PoneuBar{{\bar\bbP^1_u}}

\newcommand\Aonet{{\bbA^1_t}}
\newcommand\Aoneu{{\bbA^1_u}}

\newcommand\AonetBar{{\bar\bbA^1_t}}
\newcommand\AoneuBar{{\bar\bbA^1_u}}

\newcommand\Udee[1]{{\Aonet[1/#1]}}

\newcommand\tbar{{\bar{t}}}
\newcommand\ubar{{\bar{u}}}

\newcommand\etabar{{\bar\eta}}

\newcommand\PhiOf[1]{\Phi(#1)}
\newcommand\PhiEOf[2]{{\Phi_{#1}(#2)}}

\newcommand\PhiQ{\PhiOf{\Q}}

\newcommand\PhiU{{\PhiOf{u}}}

\newcommand\PhiUNu{{\PhiU^\nu}}

\newcommand\PhiOfDistinct[1]{{\PhiOf{#1}_{\rhomod{distinct}}}}

\newcommand\PhiEOfDistinct[2]{{\Phi_{#1}(#2)_{\rhomod{distinct}}}}

\newcommand\PhiQp{\PhiOf{\Qp}}
\newcommand\PhiQpDistinct{{\PhiQp_{\rhomod{distinct}}}}
\newcommand\PhiQpGood[1]{{\PhiQp_{#1\,\rhomod{good}}}}

\newcommand\PhiQAwful[1]{{\PhiQ_{#1\,\rhomod{heavy}}}}
\newcommand\PhiQBad[1]{{\PhiQ_{#1\,\rhomod{bad}}}}
\newcommand\PhiQBig[1]{{\PhiQ_{#1\,\rhomod{big}}}}
\newcommand\PhiQDistinct{{\PhiQ_{\rhomod{distinct}}}}
\newcommand\PhiQGood[1]{{\PhiQ_{#1\,\rhomod{good}}}}
\newcommand\PhiQGoodBig[1]{{\PhiQ_{#1\,\rhomod{good}\,\cap\,\rho\,\rhomod{big}}}}
\newcommand\PhiQMixed[1]{{\PhiQ_{#1\,\rhomod{mixed}}}}

\newcommand\PhiUNuGood[1]{{\PhiU_{#1\,\rhomod{good}}^\nu}}
\newcommand\PhiUNuBad[1]{{\PhiU_{#1\,\rhomod{bad}}^\nu}}

\newcommand\rhomod[1]{\mathrm{#1}}

\newcommand\trhochi{\theta_{\rho,\dc}}
\newcommand\trhochip{\theta_{\rho,\dc'}}

\newcommand\ThetaRhoq{{\Theta_{\rho,q}}}

\newcommand\Garith[2]{{\GG_\arith(#2,#1\PhiUNu)}}
\newcommand\Ggeom[2]{{\GG_\geom(#2,#1\PhiUNu)}}

\newcommand\dc{{\varphi}}
\newcommand\dcc{{\varphi_\CC}}
\newcommand\dcr{{\varphi_\RR}}

\newcommand\chinot{{\mathbf{1}}}
\newcommand\chiz{{\dc_!}}
\newcommand\chibarz{{\bar\dc_!}}

\newcommand\chione{{\dc_1}}
\newcommand\chitwo{{\dc_2}}
\newcommand\chionez{{\dc_{1!}}}
\newcommand\chitwoz{{\dc_{2!}}}
\newcommand\chionebarz{{\bar\dc_{1!}}}
\newcommand\chitwobarz{{\bar\dc_{2!}}}

\newcommand\rhol{{\rho}}
\newcommand\rholv{{\rho_v}}

\newcommand\rhochi{{\rho\otimes\dc}}

\newcommand\Vl{{V}}
\newcommand\Wl{{V_v}}
\newcommand\Wll{{V^{I(v)}}}

\newcommand\phivl[1]{{L(#1,\rho_v)}}
\newcommand\LCStar{{L^*_\CC}}

\newcommand\arvm[1]{{a_{#1,m}}}

\newcommand\bn[1]{{b_{#1,n}}}
\newcommand\bnstar[1]{{b^*_{#1,n}}}

\newcommand\dropCee[1]{{\dr_\CC}(#1)}
\newcommand\degL{{r_\emptyset}}

\renewcommand\H[2]{\HH^{#1}(#2)}

\newcommand\ME[1]{{\mathrm{ME}(#1)}}

\newenvironment{enum}%
	{%
	 \begin{enumerate}\setlength{\itemsep}{0.075in}}%
	{\end{enumerate}%
	}

\begin{document}


\begin{abstract}
We compute the variances of sums in arithmetic progressions of arithmetic functions associated with certain $L$-functions of degree two and higher in $\Fq[t]$, in the limit as $q\to\infty$. This is achieved by establishing appropriate equidistribution results for the associated Frobenius conjugacy classes.  The variances are thus related to matrix integrals, which may be evaluated.  Our results differ significantly from those that hold in the case of degree-one $L$-functions (i.e. situations considered previously using this approach).  They correspond to expressions found recently in the number field setting assuming a generalization of the pair-correlation conjecture.  Our calculations apply, for example, to elliptic curves defined over $\Fq[t]$. 
\end{abstract}


\title[Variance of sums of arithmetic functions]%
{Variance of sums in arithmetic progressions of arithmetic functions associated with higher degree $L$-functions in $\mathbb{F}_q[t]$}

\author{Chris Hall%
\address{Department of Mathematics, The University of Western Ontario, London, ON, Canada, N6A 5B7}}

\author{Jonathan P.~Keating%
\address{School of Mathematics, University of Bristol, Bristol BS8 1TW, UK}}

\author{Edva Roditty-Gershon%
\address{School of Mathematics, University of Bristol, Bristol BS8 1TW, UK}}

\thanks{ We are pleased to acknowledge support under EPSRC Programme Grant EP/K034383/1
LMF: \textit{L}-Functions and Modular Forms.  JPK is also grateful for support through a Royal Society Wolfson Research Merit Award and a Royal Society Leverhulme Senior Research Fellowship.   We thank Nick Katz, MManuel Kowalski, and Zeev Rudnick for discussion and helpful comments.}


\allowdisplaybreaks
\maketitle


\section{Introduction}\label{sec:introduction}


\subsection{Analytic motivation}

Let $\Lambda(n)$ denote the von Mangoldt function, defined by
\begin{equation*}
	\Lambda(n)
	=
	\begin{cases}
		\log p & \mbox{if }n=p^k\mbox{ for some prime }p\mbox{ and integer }k\ge 1, \\
		0 & \mbox{otherwise.}
	\end{cases}
\end{equation*}
The prime number theorem implies that  
\begin{equation*}
	\sum_{n \leq x} \Lambda(n)= x+o(x),
\end{equation*}
as $x\to\infty$, determining the average of $\Lambda(n)$ over long intervals.  In many problems one needs to understand sums over shorter intervals and in arithmetic progressions.  This is significantly more difficult, because the fluctuations between different short intervals/arithmetic progressions can be large, and in many important cases we do not have rigorous results.  

One may seek to characterize the fluctuations in these sums via their variances.  These variances are the subject of several long-standing conjectures.  For example, in the case of short intervals Goldston and Montgomery \cite{GM} have made the following conjecture
\begin{conjecture}[Variance of primes in short intervals]\label{GMcon}
For any fixed $\varepsilon>0$,
\begin{equation*}
	\int_{1}^{X} \Big( \sum_{X\leq n \leq x+h} \Lambda(n)- h\Big)^{2} dx
	\sim
	hX\big(\log X-\log h\big)
\end{equation*}
uniformly for $1\leq h\leq X^{1-\varepsilon}$.
\end{conjecture}

It is natural to try to compute the variance in Conjecture \ref{GMcon} using the Hardy-Littlewood Conjecture
\begin{equation}\label{HL}
	\sum_{n\le X}\Lambda(n)\Lambda(n+k)\sim \mathfrak{S}(k)X
\end{equation}
as $X\rightarrow\infty$, where $\mathfrak{S}(k)$ is the singular series 
\begin{equation*}
	\mathfrak{S}(k) =
	\begin{cases} 
		2\prod_{p>2}\left(1-\frac{1}{(p-1)^2}\right)
		 \prod_{\substack{p>2\\ p|k}}\frac{p-1}{p-2} 
		  & \mbox{\quad if }k \mbox{ is even, } \\
		0 & \mbox{\quad if }k \mbox{ is odd. }
	\end{cases}
\end{equation*}
Montgomery and Soundararajan \cite{MS} proved that \eqref{HL}, together with an assumption concerning the implicit error term, implies a more precise asymptotic for the variance in Conjecture \ref{GMcon} when $\log X\leq h\leq X^{1/2}$, namely that it is equal to
\begin{equation*}\label{Zeta variance h LOT}
	hX\big(\log X-\log h - \gamma_0-\log 2\pi\big)
	+
	O_\varepsilon\Big(h^{15/16}X(\log X)^{17/16}+h^2X^{1/2+\varepsilon}\Big),
\end{equation*}
where $\gamma_0$ is the Euler-Mascheroni constant.

An alternative approach to computing this variance follows from
\begin{equation*}
	\frac{\zeta^\prime(s)}{\zeta(s)}=-\sum_{n=1}^\infty\frac{\Lambda (n)}{n^s},
\end{equation*}
which links statistical properties of $\Lambda(n)$ to those of the zeros of the Riemann zeta-function $\zeta(s)$.  Taking this line, Goldston and Montgomery \cite{GM} proved that Conjecture \ref{GMcon} is equivalent to the following conjecture, due to Montgomery \cite{M}, concerning the pair correlation of the non-trivial zeros $\frac12 +i\gamma$ of the zeta-function:
\begin{conjecture}[Pair Correlation Conjecture]\label{SPC}
Let 
\[
	\mathcal{F}(X,T)
	=
	\sum_{0<\gamma,\gamma'\leq T}X^{i(\gamma-\gamma')}w(\gamma-\gamma'),
\]
where $w(u)=\frac{4}{4+u^2}$.  Then for any fixed $A\geq1$ we have, assuming the Riemann Hypothesis,
\[
	\mathcal{F}(X,T)\sim\frac{T\log T}{2\pi}
\]
uniformly for $T\leq X\leq T^{A}$.
\end{conjecture}
\noindent
See also \cite{C} and \cite{LPZ}, where lower order terms are considered in the equivalence.

There is a similar theory in the case of sums in arithmetic progressions.  The Prime Number Theorem for arithmetic progression states that for a fixed modulus $\Q$,
\begin{equation}\label{PNT for arith prog}
  \sum_{\substack{n\leq X\\ n=A\bmod \Q}} \Lambda(n) \sim \frac{X}{\phi(\Q)},\quad \mbox{ as }X\to \infty
   \;,
\end{equation}
where $\phi(\Q)$ is the Euler totient function, giving the number of reduced residues modulo $\Q$.  The variance of sums over different arithmetic progressions is then defined by 
\begin{equation}
  G(X,\Q)=\sum_{\substack{A\bmod \Q\\ \gcd(A,\Q)=1}} \left|\sum_{\substack{n\leq X\\ n=A\bmod \Q}} \Lambda(n)-\frac X{\phi(\Q)}
  \right|^2.
\end{equation}
Asymptotic formulae are known when $G(X, \Q)$ is summed over a long range of values of $\Q$ (c.f.~\cite{Montgomery}, \cite{HooleyI} and \cite{HooleyII}), but much less is known concerning $G(X, \Q)$ itself.  In the latter case, Hooley has made the following conjecture \cite{HooleyICM}.
\begin{conjecture}[Variance of primes in arithmetic progressions]\label{Hooleycon}
\begin{equation*}
	G(X, \Q)\sim X\log \Q.
\end{equation*}
\end{conjecture}
\noindent
Hooley was not specific about the size of $\Q$ relative to $X$ for which this asymptotic should hold.  Friedlander and Goldston \cite{FG} have shown that in the range $\Q>X^{1+o(1)}$,
\begin{equation}\label{FG uninteresting}
  G(X,\Q)
  \sim
  X\log X - X - \frac{X^2}{\phi(\Q)} + O\left(\frac X{(\log X)^A}\right) + O((\log \Q)^3) \;.
\end{equation}
This is a relatively straightforward range because it contains at most one prime.  They conjecture that Hooley's asymptotic holds if $X^{1/2+\epsilon}<\Q<X$ and further conjecture that if
$X^{1/2+\epsilon}<\Q<X^{1-\epsilon}$ then
\begin{equation}\label{FG conj}
	G(X,\Q)
	\sim
	X\log \Q - X\cdot\left(\gamma_0 +\log 2\pi + \sum_{p\mid \Q} \frac{\log p}{p-1}\right) \;.
\end{equation}
They show that both Conjecture~\ref{Hooleycon} and \eqref{FG conj} hold assuming the Hardy-Littlewood conjecture with small remainders.  For $\Q<X^{1/2}$ relatively little seems to be known.

Conjectures \ref{GMcon} and \ref{Hooleycon} remains open, but their analogues in the function field setting have been proved in the limit of large field size \cite{KR}. 

Let $\Fq$ be a finite field of $q$ elements and $\Fq[t]$ the ring of polynomials with coefficients in $\Fq$.  Let $\MM\sub\Fq[t]$ be the subset of monic polynomials and $\MM_n\sub\MM$ be the subset of polynomials of degree $n$.  Let $\PP\sub\MM$ be the subset of irreducible polynomials and $\PP_n=\PP\cap\MM_n$.  The norm of a non-zero polynomial $f\in\Fq[t]$ is defined to be $|f|=q^{\deg f}$. 

The von Mangoldt function is the function on $\MM$ defined as
$$
	\Lambda(f)
	=
	\begin{cases}
		d & \mbox{if }f=\pi^m\mbox{ with }\pi\in\PP_d \\
		0 & \mbox{otherwise}
	\end{cases}
$$
The Prime Polynomial Theorem in this context is the identity
\begin{equation}\label{Explicit formula}
	\sum_{f\in \mathcal M_n}\Lambda(f) = q^n \;.
\end{equation}
The analogue of Conjecture \ref{GMcon} is the following result, proved in \cite{KR}: for $h\le n-5$, 
\begin{equation}\label{KRint}
	\frac{1}{q^n}\sum_{A\in \mathcal M_n}
	\left| \sum_{|f-A|\le q^h}\Lambda (f)-q^{h+1}\right|^2
	\sim
	q^{h+1}(n-h-2)
\end{equation}
as $q\rightarrow \infty$; note that $|\{f:|f-A|\le q^h\}|=q^{h+1}$.

In the same vein, the function-field analogue of Conjecture \ref{Hooleycon} was also established in \cite{KR}: fix $n\ge 2$, then, given a sequence of finite fields $\Fq$ and square-free polynomials $\Q\in \Fq[t]$ with $2\le\deg(\Q)\le n+1$, one has
\begin{equation}\label{KRap}
	\sum_{\substack{A\bmod \Q\\ \gcd(A,\Q)=1}} 
	\left|
		\sum_{\substack{f\in \mathcal M_n\\f=A \bmod \Q}} 
		\Lambda(f)-\frac{q^n}{\Phi(\Q)}
	\right|^2
	\sim
	q^n({\rm deg} \Q-1)
\end{equation}
as $q \rightarrow \infty$.

The asymptotic formulae (\ref{KRint}) and (\ref{KRap}) were established in \cite{KR} by expressing the variances as sums over families of $L$-functions.  These $L$-functions can be expressed as the characteristic polynomials of matrices representing Frobenius conjugacy classes.  In the limit as $q\rightarrow \infty$, these matrices become equidistributed in one of the classical compact groups and the sums become matrix integrals of a kind familiar in Random Matrix Theory.  Evaluating these integrals leads to the expressions above.

This approach to computing variances has subsequently been applied to other arithmetic functions defined over function fields, including the M\"obius function  \cite{KRII}, the square of the M\"obius function (i.e., the characteristic function of square-free polynomials) \cite{KRII}, square-full polynomials \cite{R-G}, and the generalized divisor functions \cite{KRRR}.  For overviews see \cite{Rud}, \cite{KR-G}, and \cite{Rod}.  The arithmetic functions considered so far have all been associated with degree-one $L$-functions (or simple functions of these).  Our main aim in this paper is to extend the theory to arithmetic functions associated with $L$-functions of degree-two and higher.  For example, our results apply to $L$-functions associated with elliptic curves defined over $\Fq[t]$. This will require us to establish the appropriate equidistribution results for such $L$-functions.  We achieve this using the machinery developed by Katz \cite{Katz:CE}.  

The main reason for moving to higher-degree $L$-functions is the recent discovery in the number-field setting that one gets qualitatively new behaviour when the degree exceeds one \cite{BKS}.

We summarize briefly now the results in \cite{BKS}.  Let $\mathcal{S}$ denote the Selberg class $L$-functions.  For $F\in\mathcal{S}$ primitive, write
\[
	F(s) = \sum_{n=1}^{\infty}\frac{a_F(n)}{n^s}.
\]
Then $F(s)$ has an Euler product
\begin{equation}\label{Euler}
	F(s)
	=
	\prod_{p}\textrm{exp}\bigg(\sum_{l=1}^{\infty}\frac{b_F(p^l)}{p^{ls}}\bigg)
\end{equation}
and satisfies the functional equation
\[
	\Phi(s) = \varepsilon_F\overline{\Phi}(1-s),
\]
where $\overline{\Phi}(s)=\overline{\Phi(\overline{s})}$ and
\[
	\Phi(s) = \Q^s\bigg(\prod_{j=1}^{r}\Gamma(\lambda_j s+\mu_j)\bigg) F(s),
\]
for some $\Q>0$, $\lambda_j>0$, $\textrm{Re}(\mu_j)\geq0$ and $|\varepsilon_F|=1$.  

There are two important invariants of $F(s)$: the degree $d_F$ and the conductor $\mathfrak{q}_F$, given by
\[
	d_F=2\sum_{j=1}^{r}\lambda_j
	,\quad
	\mathfrak{q}_F=(2\pi)^{d_F}\Q^2\prod_{j=1}^{r}\lambda_j^{2\lambda_j}.
\]
respectively.  Another is $m_F$, the order of the pole at $s=1$, which equals $1$ for the Riemann zeta function and is expected to be 0 otherwise.

Let $\Lambda_F$ be the arithmetic function defined by
\[
	\frac{F'(s)}{F(s)}
	=
	-\sum_{n=1}^{\infty}\frac{\Lambda_F(n)}{n^s},
\]
and let $\psi_F$ be the function defined by
\begin{equation*}
	\psi_{F}(x) := \sum_{n \leq x} \Lambda_F(n).
\end{equation*}
The former will be the main focus of our attention.

A generalized prime number theorem of the form
\begin{equation*}
	\sum_{n \leq x} \Lambda_F(n) = m_{F} x+o(x)
\end{equation*}
is expected to hold.  In analogy with the case of the Riemann zeta function, it is natural to consider the variance
\begin{equation*}
	\tilde{V}_F(X, h) :=  \int_{1}^{X}\Big |\psi_F(x+h)-\psi_F(x) - m_Fh\Big|^{2} dx.
\end{equation*}
For example, when $F$ represents an $L$-function associated with an elliptic curve, $\tilde{V}_F(X, h)$ is the variance of sums over short intervals involving the Fourier coefficients of the associated modular form evaluated at primes and prime powers; and in the case of Ramanujan's $L$-function, it represents the corresponding variance for sums involving the Ramanujan $\tau$-function.

For most $F\in\mathcal{S}$ it is expected that
\begin{equation*}
	\sum_{n\le X}\Lambda_F(n)\Lambda_F(n+h) = o(X).
\end{equation*}
This might lead one to expect that $\tilde{V}_F(X, h)$ typically exhibits significantly different asymptotic behaviour than in the case when $F$ is the Riemann zeta-function because in that case \eqref{HL} plays a central role in our understanding of the variance.  However, all principal $L$-functions are believed to look essentially the same from the perspective of the statistical distribution of their zeros; that is, it is conjectured that the zeros of all primitive $L$-functions have a limiting distribution which coincides with that of random unitary matrices, as in Montgomery's conjecture (\ref{SPC}).  It was proved in \cite{BKS}, assuming the Generalized Riemann Hypothesis (GRH), that an extension of the pair correlation conjecture for the zeros that includes lower or terms (and which itself follows from the ratio conjecture of \cite{CFZ} along the lines of \cite{CS}) is equivalent to the formulae \eqref{2.2} and \eqref{2.3} below for $\tilde{V}_F(X, h)$ which generalize the Montgomery-Soundararajan formula (\ref{Zeta variance h LOT}).   

If $0<B_1<B_2\leq B_3<1/d_F$, then
\begin{eqnarray}\label{2.2}
	\tilde{V}_F(X,h)
	& = &
	h X\Big(d_F \log\frac{X}{h}+\log\mathfrak{q}_F-(\gamma_0+\log 2\pi)d_F\Big)\nonumber\\
	&   &
	\qquad\qquad+O_\varepsilon\big(hX^{1+\varepsilon}(h/X)^{c/3}\big)+O_\varepsilon\Big(hX^{1+\varepsilon}\big(hX^{-(1-B_1)}\big)^{1/3(1-B_1)}\Big)
\end{eqnarray}
uniformly for $X^{1-B_3}\ll h\ll X^{1-B_2}$, for some $c>0$.

Otherwise, if $1/d_F<B_1<B_2\leq B_3<1$,
\begin{eqnarray}\label{2.3}
	\tilde{V}_F(X,h)
	& = &
	\frac{1}{6}h X\Big(6\log X-\big(3+8\log 2\big)\Big) \\
	&   &
	\qquad\qquad+O_\varepsilon\big(hX^{1+\varepsilon}(h/X)^{c/3}\big)+O_\varepsilon\Big(hX^{1+\varepsilon}\big(hX^{-(1-B_1)}\big)^{1/3(1-B_1)}\Big)
\end{eqnarray}
uniformly for $X^{1-B_3}\ll h\ll X^{1-B_2}$, for some $c>0$.

If $d_F=1$ there is only one regime of behaviour, governed by (\ref{2.2}).  When $\mathfrak{q}_F=1$, this coincides exactly with (\ref{Zeta variance h LOT}); and when  $\mathfrak{q}_F\neq 1$, it generalizes (\ref{Zeta variance h LOT}) in a straightforward way.

If $d_F>1$ there are two ranges of behaviour, depending on the size of $h$.  In the first range, $\tilde{V}_F(X,h)/h$ is proportional to $\log h$; in the regime it is independent of $h$ at leading order.  It is this behaviour that we seek to understand better in the case of function fields.  In that case we are able to establish unconditional theorems which illustrate the qualitatively new form of the variance when the degree two or higher.


\subsection{Function-field analogue}

Our results are quite general and to state them requires a good deal of notation and terminology to be developed.  For this reason we postpone presenting them until later sections, when the necessary theory has been developed.  For reference, our main results are Theorem~\ref{thm:variance-estimate} (see \S\ref{sec:sums-in-arithmetic-progressions}) and Theorem~\ref{thm:application} (see \S\ref{sec:explicit-abelian-varieties}).  The former provides the variance estimates we need and the latter provides an application of these estimates to $L$-functions of abelian varieties.  Two key ingredients used to prove these theorems are Theorem~\ref{thm:big-monodromy-implies-equidistribution} (see \S\ref{sec:equidistribution}) and Theorem~\ref{thm:is-equidistributed} (see \S\ref{sec:big-monodromy}) which provide requisite equidistribution and  big-monodromy results respectively.

To illustrate our results we state now a special case of one of them.

Suppose $q$ is an odd prime power, and let $E/\Fq(t)$ be the Legendre curve, that is, the elliptic curve with affine model
$$
	y^2 = x(x-1)(x-t).
$$
Over the ring $\Fq[t]$, this curve has bad reduction at $t=0,1$ and good reduction everywhere else, so it has conductor $\s=t(t-1)$.  It also has additive reduction at $\infty$, so the $L$-function is given by an Euler product
$$
	L(T,E/\Fq(t))
	=
	\prod_{\pi\in\mathcal{P}} L(T^{\deg(\pi)},E/\mathbb{F}_\pi)^{-1} 
$$
where $\mathcal{P}\subset\Fq[t]$ is the subset of monic irreducibles and $\mathbb{F}_\pi$ is the residue field $\Fq[t]/\pi\Fq[t]$.

Each Euler factor of $L(T,E/\Fq(t))$ is the reciprocal of a polynomial in $\mathbb{Q}[T]$ and satisfies
$$
	T\frac{d}{dT}
	\log L(T,E/\mathbb{F}_\pi)^{-1}
	=
	\sum_{m=1}^\infty a_{\pi,m}T^m
	\in\mathbb{Z}[[T]].
$$
Moreover, if we define $\Lambda_{\mathrm{Leg}}$ to be the function on the subset $\mathcal{M}$ of monic polynomials given by
$$
	\LambdaLeg(f)
	=
	\begin{cases}
		d\cdot a_{\pi,m} & \mbox{if }f=\pi^m\mbox{ with }\pi\in\mathcal{P}\mbox{ and }\deg(\pi)=d \\
		0 & \mbox{otherwise},
	\end{cases}
$$
then the $L$-function satisfies
$$
	T\frac{d}{dT}
	\log(L(T,E/\Fq(t)))
	=
	\sum_{n=1}^\infty
	\left(
	\sum_{f\in\mathcal{M}_n}
	\LambdaLeg(f)
	\right)
	T^n.
$$

Let $\Q\in\Fq[t]$ be monic and square free.  For each $n\geq 1$ and each $A$ in $\Gamma(\Q)=(\Fq[t]/\Q\Fq[t])^\times$, consider the sum
$$
	S_{n,\Q}(A)
	\ :=
	\sum_{\substack{f\in\mathcal{M}_n\\f\equiv A\bmod\Q}}
	\LambdaLeg(f).
$$
Let $A$ vary uniformly over $\Gamma(\Q)$, and consider the moments
$$
	\mathbb{E}_A[S_{n,\Q}(A)]
	=
	\frac{1}{|\Gamma(\Q)|}
	\sum_{A\in\Gamma(\Q)}
	S_{n,\Q}(A),
	\quad
	\mathrm{Var}_A[S_{n,\Q}(A)]
	=
	\frac{1}{|\Gamma(\Q)|}
	\sum_{A\in\Gamma(\Q)}
	|S_{n,\Q}(A)-\mathbb{E}_A[S_{n,\Q}(A)]|^2.
$$
These moments (and the quantity $|\Gamma(\Q)|$) depend on $q$, so one can ask how they behave when we replace $\Fq$ by a finite extension, that is, let $q\to\infty$.  Using the theory we develop in this paper one can prove the following theorem.

\begin{theorem}\label{thm:intro-theorem}
\label{intro theorem}
If $\gcd(\Q,\s)=t$ and if $\deg(\Q)$ is sufficiently large, then
$$
	|\Gamma(\Q)|\cdot\mathbb{E}_A[S_{n,\Q}(A)]
	=
	\sum_{f\in\mathcal{M}_n}\LambdaLeg(f),\quad
	\lim_{q\to\infty}\frac{|\Gamma(\Q)|}{q^{2n}}\cdot\mathrm{Var}_A[S_{n,\Q}(A)] = \min\{n,2\deg(\Q)-1\}.
$$
\end{theorem}

\noindent
See Theorem~\ref{thm:application}.  This should be compared to (\ref{KRap}).  For definiteness, we could replace ``sufficiently large'' by $\deg(\Q)>900$, but we do not believe this bound to be optimal.  We also do not believe the hypothesis on $\gcd(\Q,\s)$ is necessary (cf.~Remark~\ref{rmk:unipotence-hypothesis}).

The fact that the expression for the variance depends on $2\deg(\Q)$ is a direct consequence of the fact that the associated $L$-functions have degree two.  (For an $L$-function of degree $r$, one will get a leading term of $r\deg(\Q)$ instead.)  This then leads to there being two ranges of behaviour.

The analogues of our main results in the number field setting are formulae for the variance of $\Lambda_F$ when summed over arithmetic progressions (a similar case to when these sums are considered in short intervals, as in \eqref{2.2} and \eqref{2.3}). For example if we take a rational elliptic curve and write the number of points over the field of $p$ elements as 
$$
N_{p}=p+1-a_{p}
$$
and the number of points over an extension field of degree $m$ as
$$
N_{p^{m}}=p^{m}+1-a_{p^{m}}
$$ 
then our function field theorems are analogous to considering the fluctuations of the sum of $a_{p^{m}}$, weighted by the logarithm of $p$, over residue classes of $p^{m} \bmod c$.


\subsection{Underlying equidistribution theorem}

The key ingredients we use to prove Theorem~\ref{thm:intro-theorem} and its generalizations are the Mellin transform and Deligne's equidistribution theorem.  More precisely, we start with a lisse sheaf $\FF$ on a dense open $T\seq\Aonet[1/\s]$ and twist it by variable Dirichlet characters $\dc$ with square-free conductor $\Q$ to obtain a family of lisse sheaves $\FF_\dc$ on $T[1/\Q]$; this family is a Mellin transform of $\FF$.

One can associate a monodromy $\GG_\arith$ group to this family generated by Frobenius conjugacy classes $\Frob_{\EFq,\dc}$ for variable Dirichlet characters $\dc$ over  finite extensions $\EFq/\Fq$.  A priori $\GG_\arith$ is reductive and defined over $\Qellbar$, but Deligne's Riemann hypothesis allows us to associate the classes $\Frob_{\EFq,\dc}$ for `good' $\dc$ to well-defined conjugacy classes in a compact form of the `same' reductive group over $\bbC$.  Deligne's equidistribution theorem implies these classes are equidistributed.

For our applications, we need equidistribution in a unitary group $U_\R(\bbC)$, and thus we need $\GG_\arith$ to be as big as possible, namely $\GL_{\R,\Qellbar}$.  We were only able to prove this is the case under the hypotheses that $\deg(\Q)\gg 1$ and that $\FF$ has a unipotent block of exact multiplicity one about $t=\gcd(\Q,\s)=0$.

On one hand, while we do expect that one may encounter exceptions when $\deg(\Q)$ is small, we do not believe our lower bound on $\deg(\Q)$ is sharp.  On the other hand, the hypothesis on the monodromy about the unique prime dividing $\gcd(\Q,\s)$ was made in order to ensure we could exhibit elements of $\GG_\arith$ whose existence helped ensure the group was big.  We conjecture one still has big monodromy under the weaker hypothesis that $\gcd(\Q,\s)=1$.


\subsection{Overview}

The structure of this paper is as follows.

We start in \S\ref{sec:framework} by establishing notation and relatively basic facts that we need throughout the rest of the paper.

In \S\ref{sec:l-functions} we define two $L$-functions that one can attach to a Galois representation $\rho$: the complete $L$-function $L(T,\rho)$ and a partial $L$-function $\LC(T,\rho)$.  The former may be defined in terms of an Euler product over all places of the function field $\Fq(t)$, and for the latter we exclude the Euler factors indexed by a finite set $\CC$ of places in $\Fq(t)$.  If the excluded Euler factors are in fact trivial, then the two $L$-functions will coincide, but otherwise they will not.  Either way, after imposing requisite hypotheses on the representation $\rho$, we apply the theory of $L$-functions and also Deligne's theorem to deduce information about their degrees and zeros.

In \S\ref{sec:twisted-l-functions} we consider twists of the representation $\rho$ by Dirichlet characters $\dc$ of square-free conductor $\Q$.  The material in this section is mostly a recasting of the results in \S\ref{sec:l-functions} in a manner which is convenient for us.  The main objects of interest at the complete $L$-function $L(T,\rhochi)$ and the partial $L$-function $\LC(T,\rhochi)$.

In \S\ref{sec:equidistribution} we recall the notion of a good character $\dc$ for $\rho$: it is a character such that $L(T,\rhochi)$ and $\LC(T,\rhochi)$ are both polynomials and equal to each other.  This is precisely the property we need to deduce that they are `pure', that is, that their zeros are Weil numbers, and to produce a unitarized $L$-function $\LCStar(T,\rhochi)$.  This allows us to associate to each good character $\dc$ a conjugacy class $\theta_{\rho,\dc}$ in a unitary group $U_\R(\bbC)$ for $R=\deg(\LC(T,\rho))$.  We define what it means for the resulting multiset of conjugacy classes $\Theta_{\rho,q}$ to be equidistributed in $U_\R(\bbC)$ as $q\to\infty$.  SSentially it says that for any representation $\Lambda\colon U_\R(\bbC)\to\GL_n(\bbC)$, the average of $\Lambda(\Tr(\theta_{\rho,\dc}))$ over the good $\dc$ tends to the value of a matrix integral $\int_{U_\R(\bbC)}\Tr(\Lambda(\theta))d\theta$.  We then prove a theorem which asserts that one achieves equidistribution when the Mellin transform of $\rho$ has big monodromy.

In \S\ref{sec:sums-in-arithmetic-progressions} we introduce the arithmetic functions of interest to us.  More precisely, we define a generalization $\Lambda_\rho$ of the von Mangoldt function and consider sums $\SnAQ$ of its values in an arithmetic progression modulo $\Q$.  For each $n$, we consider the expected value and variance of these sums as $A$ varies uniformly over $\BQ$.  We show how to evaluate the limit of both quantities as $q\to\infty$ under the hypothesis that the Mellin transform of $\rho$ has big monodromy.  As mentioned above, we use this hypothesis to deduce that the conjugacy classes $\Theta_{\rho,q}$ are equidistributed and then to evaluate the variance in terms of an easy-to-evaluate matrix integral.

In \S\ref{sec:big-monodromy} we prove a theorem which asserts that the Mellin transform of $\rho$ has big monodromy provided $\rho$ satisfies certain hypotheses.  The material in this section rests heavily on the monumental works of Katz, most notably the monograph \cite{Katz:CE}.  In order to prove our result, we were forced to impose the condition that the (square-free) conductor $\s$ of $\rho$ and the twisting conductor $\Q$ satisfy $\deg(\gcd(\Q,\s))=1$.  We also imposed conditions on the local monodromy of $\rho$ at the zero of $\deg(\Q,\s)$.  We used both of these hypotheses to deduce that the relevant monodromy groups contained an element so special that the group was forced to be big (e.g., for the specific example considered in Theorem~\ref{thm:intro-theorem} one obtains pseudoreflections).  While the specific result we proved is new, it borrows heavily from the rich set of tools developed by Katz, and one familiar with his work will easily recognize the intellectual debt we owe him.

In \S\ref{sec:explicit-abelian-varieties} we bring everything together and show how Galois representations arising from (Tate modules of) certain abelian varieties satisfy the requisite properties to apply the theorems of the earlier sections.  More precisely, we consider Jacobians of (elliptic and) hyperelliptic curves of arbitrary genus, the Legendre curve being one such example.  Because we chose to work with hyperelliptic curves we were forced to assume $q$ is odd.  Nonetheless, we expect one can find other suitable examples in characteristic two.

There are two appendices to the paper containing material we needed for the results in Section~\ref{sec:big-monodromy}.  In the first appendix we prove the group-theoretic result which asserts that a reductive subgroup of $\GL_\R$ with the sort of special element alluded to above is big.  In the second appendix we recall much of the abstract formalism required to define the monodromy groups which we want to show are big.  While none of this material is new, it elaborates on some of the facts which we felt were not always easy to give a direct reference for in \cite{Katz:CE}.  In particular, our work should not be regarded as a substitute for Katz's original monograph, but we hope some readers will find it an acceptible complement to his masterful presentation. 


\section{Framework}\label{sec:framework}


\subsection{Notation}

Let $q$ be the power of an odd prime $p$, $\Fq$ be the finite field with $q$ elements, and $K$ be the global field $\Fq(t)$.  Let $\PP$ be the places of $K$ and $\PP_d\sub\PP$ be the finite subset of places of degree $d$.  For each $v\in\PP$, let $\Fv$ be its residue field and $d_v=[\Fv:\Fq]$ be its degree.  If $v$ is a finite place, then it corresponds to a monic irreducible $\pi\in\Fq[t]$, and $\Fpi$ is the quotient ring $\Fq[t]/\pi$.  On the other hand, the residue field of the unique infinite place $v=\infty$ can be regarded as the quotient ring $\Fq[u]/u$ by taking $u=1/t$.

Let \defi{$\MM\sub\Fq[t]$} be the subset of monic polynomials and \defi{$\MM_d\sub\MM$} be the subset of polynomials of degree $d$.  Let \defi{$\AA_d\seq\MM_d$} be the subset of irreducible polynomials and \defi{$v\colon\AA_d\to\PP_d$} be the map which identifies an irreducible $\pi$ with its corresponding finite place $v(\pi)$.

Let $\Ksep$ be a separable closure of $K$ and $\Fqbar\sub\Ksep$ be the algebraic closure of $\Fq\sub K$.  Let $\GK=\Gal(\Ksep/K)$ and $G_{\Fq}=\Gal(\Fqbar/\Fq)$, and let $\bar{G}_K\seq G_K$ be the stabilizer of $\Fqbar$ so that there is an exact sequence
$$
	1
	\longto \bar{G}_K
	\longto G_K
	\longto G_{\Fq}
	\longto 1
$$
of profinite groups.  Given a quotient $\GK\onto Q$ of profinite groups, we write $\bar{Q}\seq Q$ for the image of $\bar{G}_K$ and call it the \defi{geometric subgroup}.

For each $v\in\PP$, we fix a decomposition group $D(v)\seq \GK$, that is, a representative of its conjugacy class; equivalently we fix a place of $\Ksep$ over $v$.  Let $I(v)\seq D(v)$ be the inertia subgroup and $P(v)\seq I(v)$ be the wild inertia subgroup (i.e., the $p$-Sylow subgroup).  The quotient \defi{$\Gv=D(v)/I(v)$} is the absolute Galois group of $\Fv$, and we write $\Frob_v\in\Gv$  for the Frobenius element $\Frob_q^{d_v}$.

For each subset $S\sub\PP$, let $\KS\seq\Ksep$ be the maximal subextension unramified \emph{away} from $S$ and $\KSt\seq\KS$ be the maximal subextension \emph{tamely} ramfied over $S$.  Both extensions are Galois over $K$, so we write \defi{$\GKS$} and \defi{$\GKSt$} for their respective Galois groups.  There is a commutative diagram
$$
	\xymatrix{
		\GK\ar[rr]\ar[dr] & & \GKS\ar[dl] \\
		& \GKSt &
	}
$$
of quotients.

If $v\not\in S$, then the inertia subgroup $I(v)$ is contained in the kernel of the horizontal map.  In particular, every element of the coset $\Frob_v I(v)$ maps to the same element of $\GKS$ which we denote $\Frob_v\in\GKS$.  Moreover, the kernel of the horizontal map is generated by the conjugates of the $I(v)$ for $v\not\in S$, and it and the conjugates of the $P(v)$ for $v\in S$ generate the rest of the kernel of the other map from $\GK$.

Given a number field $E$, we write $\bbZ_E$ for the ring of integers.  Given a maximal prime $\l\sub\bbZ_E$, we write $\ell\in\bbZ$ for the rational prime it divides and $E_\l$ for the $\l$-adic completion of $E$.  We also write $\bar{E}_\l$ for an algebraic closure of $E_\l$, e.g., $\Qellbar$ is an algebraic closure of $\Qell$.

Given a smooth geometrically connected curve $U$ over $\Fq$, we write $\Ubar$ for the base change curve $U\times_{\Fq}\Fqbar$.  We fix (but do not name) a geometric generic point of $U$ and write $\piOne{U}$ and $\piOne{\Ubar}$ for the arithmetic and geometric \'etale fundamental groups of $U$ respectively.  Moreover, if $T$ is a second smooth geometrically connected curve over $\Fq$ and if $T\to U$ is a finite \'etale cover, then we implicitly suppose the geometric generic point of $T$ maps to that of $U$ and write $\piOne{T}\to\piOne{U}$ for the induced inclusion of fundamental groups.

Given a sheaf $\FF$ on $U$, we suppose that $\FF$ is constructible, and unless stated otherwise we suppose it has coefficients in $\Qellbar$.  We also write $H^i(\Ubar,\FF)$ and $H^i_c(\Ubar,\FF)$ for the \'etale cohomology groups of $\FF$.  For each integer $n$, we write $\FF(n)$ for the Tate twisted sheaf $\FF\otimes_{\Qellbar}\Qellbar(n)$ and recall that
$$
	\det(1-T\,\Frob_q\mid H^i(\Ubar,\FF(n)))
	=
	\det(1-q^nT\,\Frob_q\mid H^i(\Ubar,\FF)).
$$
A similar identity holds for cohomology with compact supports (cf.~\cite[Proof of 6.1.13]{SGA4.5}).  In particular, we have identities
$$
	\dim(H^i(\Ubar,\FF(n)))
	=
	\dim(H^i(\Ubar,\FF)),
	\quad
	\dim(H^i_c(\Ubar,\FF(n)))
	=
	\dim(H^i_c(\Ubar,\FF))
$$
for every $i$ and $n$.

The sheaf $\FF$ is lisse (or locally constant) on $U$ if and only it corresponds to a continuous representation $\piOne{U}\to\GL(V)$ from the \'etale fundamental group to a finite-dimensional $\Qellbar$ vector space $V$ (cf.~\cite[II.3.16.d]{Milne}).  In that case one has identifications
\begin{equation}\label{eqn:invariants-and-coinvariants}
	H^0(\Ubar,\FF)
	=
	V^{\piOne{\Ubar}}
	\mbox{ and }
	H^2_c(\Ubar,\FF(2))
	=
	V_{\piOne{\Ubar}}
\end{equation}
with the subspace of $\piOne{\Ubar}$-invariants and quotient space of $\piOne{\Ubar}$-coinvariants (see \cite[Exp.~6, 1.18.d]{SGA4.5}).


\section{$L$-functions}\label{sec:l-functions}

Let $\ell$ be a prime distinct from $p$ and $\Vl$ be a finite-dimensional $\Qellbar$-vector space.  Let $\Q\in\Fq[t]$ be monic and square free, $\CC\sub\PP$ be the subset consisting of $\infty$ and $v(\pi)$ for every prime factor $\pi$ of $\Q$, and $\SS\sub\PP$ be a finite subset of places.  Suppose $\rho$ is a homomorphism
$$
	\rho\colon \GKS\to\GL(\Vl)
$$
which is continuous with respect to the profinite topologies and which has \defi{trivial geometric invariants} (i.e., the subspace of $\bar{G}_{K,\SS}$-invariants of $\Vl$ vanishes).

In this section, we define, for each $v\in\PP$, the Euler factor $L(T,\rho_v)\in\Qellbar[T]$ of a local representation $\rho_v\colon G_v\to\GL(\Vl_v)$, as well as $L$-functions
$$
	L(T,\rho)
	=
	\prod_{v\in\PP}
	L(T^{d_v},\rho_v)^{-1},
	\quad
	\LC(T,\rho)
	=
	\prod_{v\not\in\CC}
	L(T^{d_v},\rho_v)^{-1}
$$
and cohomological factors
$$
	\PC{i}(T) = \det(1-T\,\Frob_q\mid H^i_c(\AonetBar[1/\Q], \ME{\rho}))
	\mbox{ for }i=1,2
$$
(see \S\ref{subsec:l-functions-of-rho}).  We also define numerical invariants of $\rho$, including $\dr(\rho)$, $\dropCee{\rho}$, and $\swan(\rho)$, and we show that $\deg(\LC(T,\rho))$ and $\deg(L(T,\rho))$ equal
\begin{equation}\label{eq:def-rC}
	\rC(\rho)
	=
	\dr(\rho)
	-
	\dropCee{\rho}
	+
	\swan(\rho)
	+
	(\deg(\Q)-1)\cdot\dim(\Vl)
\end{equation}
and
\begin{equation}\label{e:def-degL}
	\degL(\rho)
	=
	\dr(\rho)
	+
	\swan(\rho)
	-
	2\cdot\dim(\Vl)
\end{equation}
respectively (see \S\ref{subsec:numerical-invariants-of-rho}).  Finally, we define what it means for $\rho$ to be punctually $\iota$-pure of weight $w$ and use Deligne's Riemann hypothesis to derive some consequential properties of the $L$-functions (see \S\ref{subsec:purity} and \S\ref{subsec:semisimplicity-and-irreducibility}).  Using these definitions we then given the main result of this section, Theorem~\ref{thm:archimedean-bound}, in \S\ref{sec:proof-of-archimedean-bound}.


\subsection{Galois modules versus sheaves}

While most of this paper uses the language of global fields, it is useful to adopt a geometric language.  Certain readers will find the latter language more to their taste, and we acknowledge that many of our results may have a more appealing formulation in the language of geometry (and sheaves).  However, we felt the language of Galois representations over global (function) fields was accessible to a broader audience, so we tried to do `as much as possible' in that language.


\subsection{Middle extensions}

Let $U\seq X\seq\Ponet$ be dense Zariski open subsets and $j\colon U\to X$ be the inclusion, and let $\FF$ be a sheaf on $X$.  Suppose everything is defined over $\Fq$ so that the fiber $\FF_{\etabar}$ of $\FF$ over the geometric point $\etabar=\mathrm{Spec}(\Kbar)$ is a $G_K$-module.  If the restriction $j^*\FF$ is lisse on $U$, then the fiber $\FF_\etabar$ is even a module over the \'etale fundamental group $\piOneU$.  Conversely, for every continuous homomorphism $\piOneU\to\GL(\Vl)$, there is a lisse $\Qellbar$-sheaf on $U$ whose fiber over $\etabar$ is the $\piOneU$-module $\Vl$.

Given a sheaf $\GG$ sheaf on $U$ (e.g., $j^*\FF$), there are two functorial extensions of $\GG$ to a sheaf on all of $X$ we wish to consider, the extension by zero $j_!\GG$ and the direct image $j_*\GG$.  (One can also consider hybrid versions such as $j''_!j'_*\GG$ for inclusions $j'\colon U\to U'$ and $j''\colon U''\to X$, but we do not need such versions.)  As $\FF$ and $\GG$ vary we have
$$
	\Hom_{X}(j_!\GG,\FF) = \Hom_U(\GG,j^*\FF)
	\quad\mbox{and}\quad
	\Hom_{X}(\FF,j_*\GG) = \Hom_U(j^*\FF,\GG),
$$
that is, the functors $j_!,j_*$ are adjoints of $j^*$ (cf.~\cite[II.3.14.a]{Milne}).  In particular, the adjoints of the identity $j^*\FF\to j^*\FF$ are maps of the form $j_!j^*\FF\to\FF$ and $\FF\to j_*j^*\FF$ which we call \defi{adjunction maps}.  We say that $\FF$ is \defi{supported on $U$} iff the first map is an isomorphism, and $\FF$ is a \defi{middle extension} iff the second map is an isomorphism for \emph{every} $j$.

\begin{lemma}\label{lem:lisse-vs-middle}\ 

\begin{enum}
\item\label{part:single-U} If $j^*\FF$ is lisse and $\FF\to j_*j^*\FF$ is an isomorphism, then $\FF$ is a middle extension.
\item\label{part:lisse-to-middle} If $\GG$ is lisse, then $j_*\GG$ is a middle extension.
\end{enum}
\end{lemma}

\begin{proof}
Let $U'\seq X$ be a dense Zariski open and $U''=U\cap U'$.  Consider the commutative diagram
$$
	\xymatrix{
	U''\ar[r]^{i'}\ar[d]_{i} & U'\ar[d]^{j'} \\
	U\ar[r]_j & X \\
	}
$$
of inclusions and the corresponding commutative diagram
\begin{equation}\label{eq:adj-com}
	\begin{array}{c}
	\xymatrix{
	\FF\ar[d]\ar[r] & j_*j^*\FF\ar[d] \\
	j'_*j^{\prime *}\FF\ar[r] & (ij)_*(ij)^*\FF = (i'j')_*(i'j')^*\FF
	}
	\end{array}
\end{equation}
of adjunction maps.

Suppose $\GG$ is lisse.  On one hand, this implies the map $\GG\to i_*i^*\GG$ is an isomorphism, so the right map of \eqref{eq:adj-com} is an isomorphism when $j^*\FF$ is lisse.  In particular, if the top map of \eqref{eq:adj-com} is also an isomorphism, then the left map must also be an isomorphism, for \emph{every} $j$, hence \eqref{part:single-U} holds.  On the other hand, the direct image map $j_*\GG\to j_*i_*i^*\GG$ is also an isomorphism.  It even coincides with the adjunction map $j_*\GG\to j'_*j^{\prime *}j_*\GG$ via the functorial identities $j_*i_*i^*\GG = j'_*i'_*i^*\GG = j'_*j^{\prime *}j_*\GG$, so \eqref{part:lisse-to-middle} holds.	\end{proof}

The following proposition shows that there is a canonical middle extension sheaf on $\Ponet$ we can associate to $\rho$.  We denote it and its restriction to $X$ by $\ME{\rho}$.

\begin{prop}\label{prop:assoc-me}
There is a middle extension $\FF$ with $\FF_\etabar=\Vl$ as $G_K$-modules, and it is unique up to isomorphism.
\end{prop}

\begin{proof}
There are quotients $G_K\onto\piOneU$ and $G_K\onto\GKS$, so $\FF_\etabar$ and $\Vl$ are $G_K$-modules.  Moreoever, if $U'\seq U$ is a sufficiently small dense Zariski open, then there exist a quotient $\piOneUpBar\onto\GKS$ and a unique lisse sheaf $\GG$ on $U'$ with $\GG_\etabar=\Vl$ as $\piOneUpBar$-modules.  Its direct image $\ME{\rho}$ on $X$ is a middle extension by Lemma~\ref{lem:lisse-vs-middle}.\ref{part:lisse-to-middle}, and $\ME{\rho}_\etabar=\GG_\etabar=\Vl$ as $\GK$-modules by construction.

Let $\FF$ be any middle extension with $\FF_\etabar=\Vl$ as $\GK$-modules; we must show it isomorphic to $\ME{\rho}$.  Up to shrinking $U$, we may suppose that $\ME{\rho}$ and $\FF$ are lisse on $U$ and thus $\ME{\rho}_\etabar,\FF_\etabar$ are $\piOneU$-modules.  Then the canonical bijection $\ME{\rho}_\etabar\to\FF_\etabar$ extends uniquely to an isomorphism $j^*\ME{\rho}\to j^*\FF$ of lisse sheaves.  Moreover, the direct image $j_*j^*\ME{\rho}\to j_*j^*\FF$ and adjunction maps $\ME{\rho}\to j_*j^*\ME{\rho}$ and $\FF\to j_*j^*\FF$ are all isomorphisms, so there exists an isomorphism $\ME{\rho}\to\FF$ as claimed.
\end{proof}

\begin{cor}\label{cor:assoc-me}
Let $\SS'\sub\PP$ be a finite subset containing $\SS$ and $\rho'\colon G_{K,\SS'}\to\GL(\Vl)$ be the composition of $\rho$ with the natural quotient $G_{K,\SS'}\onto\GKS$.  Then $\ME{\rho}$ and $\ME{\rho'}$ are isomorphic.
\end{cor}

\begin{proof}
The quotient $\GK\to\GKS$ factors as $\GK\onto G_{K,\SS'}\onto\GKS$, and $\ME{\rho'}_\etabar=\Vl=\ME{\rho}$ as $\GK$-modules.  Since $\ME{\rho},\ME{\rho'}$ are both middle extensions, Proposition~\ref{prop:assoc-me} implies they are isomorphic.
\end{proof}


\subsection{Euler characteristics}

Let $\GG$ be a sheaf on $U$.  Then there is an exact sequence
$$
	0
	\longto j_!\GG
	\longto j_*\GG
	\longto \SS_\GG
	\longto 0
$$

\smallskip\noindent
where $\SS_\GG$ is a skyscraper sheaf supported on \defi{$Z=\Ponet\ssm U$}, and the corresponding long exact sequence of (\'etale) cohomology (over $\Fqbar$) can be written
\begin{equation}\label{eq:ff-ex-seq}
	\cdots
	\to H^n(\Zbar,\SS_\GG)
	\to H^{n+1}_c(\Ubar,\GG)
	\to H^{n+1}(\PonetBar,j_*\GG)
	\to \cdots
\end{equation}
where $n\in\bbZ$.

\begin{lemma}
There exist exact sequences
\begin{equation}\label{eq:ff-ex-seq-head}
	0
	\to H^0_c(\Ubar,\GG)
	\to H^0(\PonetBar,j_*\GG)
	\to H^0(\Zbar,\SS_\GG)
	\to H^1_c(\Ubar,\GG)
	\to H^1(\PonetBar,j_*\GG)
	\to 0
\end{equation}
and
\begin{equation}\label{eq:ff-ex-seq-piece}
	0
	\longto H^2_c(\Ubar,\GG)
	\longto H^2(\PonetBar,j_*\GG)
	\longto 0
\end{equation}
and all other cohomology groups in \eqref{eq:ff-ex-seq} vanish.
\end{lemma}
  
\begin{proof}
The first term of \eqref{eq:ff-ex-seq} vanishes unless $n=0$ since $\dim(Z)=0$, and the other two terms vanish for $n+1\neq 0,1,2$ since $U$ and $\Ponet$ are curves.  Therefore \eqref{eq:ff-ex-seq} breaks into the pieces \eqref{eq:ff-ex-seq-head} and \eqref{eq:ff-ex-seq-piece}, and all other terms vanish.
\end{proof}

If $U=\Ponet$, then the middle term of \eqref{eq:ff-ex-seq-head} vanishes, and otherwise the first term vanishes since any curve $U\subsetneq\Ponet$ is affine.  Either way, the Euler characteristics
$$
	\chi(\PonetBar,j_*\GG)
	=
	\sum_{n=0}^2 (-1)^n\dim(H^n(\PonetBar,j_*\GG)),
	\quad
	\chi_c(\Ubar,j_*\GG)
	=
	\sum_{n=0}^2 (-1)^n\dim(H^n_c(\Ubar,j_*\GG)),
$$
and $\chi(\Zbar,\SS_\GG)=\dim(H^0(\Zbar,\SS_\GG))$
satisfy
\begin{equation}
	\chi(\PonetBar,j_*\GG)
	-
	\chi_c(\Ubar,\GG)
	=
	\chi(\Zbar,\SS_\GG)
	=
	\sum_{z\in Z} \deg(z)\cdot\dim(\GG^{I(z)}_\etabar).
\end{equation}


\subsection{$L$-functions of $\rho$}\label{subsec:l-functions-of-rho}

The decomposition group $D(v)$ stabilizes the subspace $\Wl=\Wll$, and $I(v)$ acts trivially on it, so there is a representation
$
	\rholv\colon\Gv\to\GL(\Wl).
$
We identify the subspace $V_v\seq V$ and the representation $\rholv$ with a  geometric fiber of $\ME{\rho}$ (cf.~\cite[3.1.16]{Milne}).  The \defi{Euler factor} of $\rho$ at $v$ is given by
$$
	\phivl{T} = \det\left(1 - T\rholv(\Frob_v)\mid\Wl\right)\in\Qellbar[T],
$$
and its degree equals the dimension of $V_v$.

The \defi{partial} and \defi{complete $L$-functions} of $\rho$ are the formal power series in $\Qellbar[T]$ with respective Euler products
\begin{equation}\label{eq:inc-l-fcn}
	\LC(T,\rho)
	=
	\prod_{v\not\in\CC}\phivl{T^{d_v}}^{-1}
	\quad\mbox{and}\quad
	L(T,\rho)
	=
	\prod_{v\in\PP}\phivl{T^{d_v}}^{-1}.
\end{equation}
If $U=\Aonet[1/\Q]$, then they equal the $L$-functions of the sheaves $j_!j^*\ME{\rho}$ and $\ME{\rho}$, and the ratio
$$
	\MC(T,\rho) = L(T,\rho)/\LC(T,\rho)
	=
	\prod_{v\in\CC} L(T^{d_v},\rho_v)^{-1}
$$
is the $L$-function of the restriction of $\ME{\rho}$ to $Z$ and hence is the reciprocal of a polynomial.

The \'etale cohomology groups of these sheaves are finite-dimensional $\Qellbar$-vector spaces, and $\Frob_q$ acts $\Qellbar$-linearly on them.  In particular, we have characteristic polynomials
$$
	\PC{n}(T) = \det(1-T\,\Frob_q\mid H^n_c(\AonetBar[1/\Q], \ME{\rho}))
$$
which are trivial for $i\neq 1,2$ since $\Udee{\Q}$ is an affine curve, and they satisfy
\begin{equation}\label{eq:L-fun:frac:partial}
	\LC(T,\rho)
	=
	{\PC{1}(T,\rho)}
	/
	{\PC{2}(T,\rho)}.
\end{equation}
Similarly, the characteristic polynomials
\begin{equation}\label{eq:def-P_n(T)}
	P_n(T,\rho) = \det(1-T\,\Frob_q\mid H^n(\PonetBar,\ME{\rho})).
\end{equation}
are trivial for $i\neq 0,1,2$ since $\Ponet$ is a complete curve, and they otherwise satisfy
\begin{equation}\label{eq:L-fun:frac}
	L(T,\rho)
	=
	\frac{P_1(T,\rho)}{P_0(T,\rho)P_2(T,\rho)}.
\end{equation}
Moreover, the degrees are related to the respective Euler characteristics via the identities
$$
	\deg(\LC(T,\rho))=-\chi_c(\Ubar,\ME{\rho})
	\mbox{\ \ and\ \ }
	\deg(L(T,\rho))=-\chi(\PonetBar,\ME{\rho}).
$$


\subsection{Numerical invariants of $\rho$}\label{subsec:numerical-invariants-of-rho}

Let 
$
	\rank_v(\rho) = \deg(L(T,\rho_v))
$
and
$
	\dr_v(\rho)=\dim(\Vl)-\rank_v(\rho),
$
and let $\swan_v(\rho)$ be the Swan conductor of $\Vl$ as an $\Qellbar[I(v)]$-module (see \cite[1.6]{Katz:GKM}).  We call these and
$$
	\dropCee{\rho}
	=
	\sum_{v\in\CC} d_v\cdot\dr_v(\rho).
$$
the \defi{local invariants} of $\rho$ and
$$
	\rank(\rho)
	=
	\dim(\Vl),
	\quad
	\dr(\rho)
	=
	\sum_{v\in\PP} d_v\cdot\dr_v(\rho),
	\quad
	\swan(\rho) = \sum_{v\in\PP} d_v\cdot\swan_v(\rho)
$$
are the \defi{global invariants}.  The latter remain unchanged if we replace $\Fq$ by a finite extension.

\begin{prop}\label{prop:chi-Ponet}
$$
	\chi(\PonetBar,\ME{\rho})
	=
	2\cdot\rank(\rho)
	-
	(
	\dr(\rho)
	+
	\swan(\rho)
	)
$$
\end{prop}

\begin{proof}
Suppose $\ME{\rho}$ is lisse on $U$ since $\ME{\rho}$ is a middle extension.  On one hand, the Euler-Poincare formula, as proved by Raynaud \cite[Th.~1]{Raynaud}, asserts
\begin{align*}
	\chi_c(\Ubar,\ME{\rho})
	& =
	\rank(\rho)
	\cdot
	(2-\deg(Z))
	-
	\swan(\rho).
\end{align*}
On the other hand, a short calculation shows
$$
	\chi(\Zbar,\ME{\rho})
	=
	\deg(Z)\cdot\rank(\rho)
	-
	\dr(\rho)
$$
since $\ME{\rho}$ is also a middle extension, and thus
$$
	\chi(\PonetBar,\ME{\rho})
	=
	\chi_c(\Ubar,\ME{\rho})
	+
	\chi(\Zbar,\ME{\rho})
	=
	2\cdot\rank(\rho) - \dr(\rho) - \swan(\rho)
$$
as claimed.
\end{proof}

\begin{cor}\label{cor:chi_c}
If $\ME{\rho}$ is supported on $\Aonet[1/\Q]$, then $\chi_c(\AonetBar[1/\Q],\ME{\rho})=\chi(\PonetBar,\ME{\rho})$, and
\begin{equation}\label{eq:chi-c-ME}
	\chi_c(\AonetBar[1/\Q],\ME{\rho})
	=
	(1-\deg(\Q))\cdot\rank(\rho)
	-
	(
	\dr(\rho)
	-
	\dropCee{\rho}
	+
	\swan(\rho)
	)
\end{equation}
in general.
\end{cor}

\begin{proof}
If $\ME{\rho}$ is supported on $\Aonet[1/\Q]$, then $\dropCee{\rho}=\deg(\CC)\cdot\rank(\rho)$ and $\deg(\CC)=1+\deg(\Q)$, so it suffices to show \eqref{eq:chi-c-ME} holds in general.  There is a canonical bijection $Z=\CC$ when $U=\Aonet[1/\Q]$, so the desired identity follows easily from the identities
$$
	\chi_c(\AonetBar[1/\Q],\ME{\rho})
	=
	\chi(\PonetBar,\ME{\rho}) - \chi(\Zbar,\ME{\rho})
$$
and
$$
	\chi(\Zbar,\ME{\rho})
	=
	\deg(\CC)\cdot\rank(\rho) - \dropCee{\rho}
$$
and from the identity in Proposition~\ref{prop:chi-Ponet}.
\end{proof}


\subsection{Purity}\label{subsec:purity}

Let $\iota\colon\Qbar\to\bbC$ and $\Qbar\to\Qellbar$ be field embeddings.  A non-zero polynomial $\psi\in\Qellbar[T]$ is \defi{$\iota$-pure of $q$-weight $w$} iff every zero $\alpha\in\Qellbar$ is a $q$-Weil number of weight $w$, that is, lies in $\Qbar$ and satisfies
$$
	|\iota(\alpha)|^2=(1/q)^w.
$$
It is \defi{pure of $q$-weight $w$} iff it is $\iota$-pure of $q$-weight $w$ for every $\iota$, and it is \defi{($\iota$-)mixed of $q$-weights $\leq w$} iff it is a product of ($\iota$-)pure polynomials each of $q$-weights $\leq w$.  Our terminology is unconventional in that we incorporate $q$, however, we need to make $q$ explicit since we have not said where $\psi$ comes from.

\begin{lemma}\label{lem:trace-bound}
If $M$ is an invertible $d\times d$ matrix with coefficients in $\Qellbar$ and if $\det(1-M\,T)$ is mixed of $q$-weights $\leq w$, then $\Tr(M)\in\Qbar$ and $|\iota(\Tr(M))|^2\leq dq^{w}$ for every field embedding $\iota\colon\Qbar\to\bbC$.
\end{lemma}

\begin{proof}
If $M$ is invertible and $\psi(T)=\det(1-M\,T)$ is mixed, there exist $\beta_1,\ldots,\beta_d\in\bar{E}^\times$ such that
$$
	\psi(T)
	= \prod_{i=1}^d (1 - \beta_i T)
	= 1 - \Tr(M)\cdot T + \cdots + (-1)^d\cdot\det(M)\cdot T^d
$$
and such that $\Tr(M)=\beta_1+\cdots+\beta_m$ also lies in $\bar{E}$.  Therefore, if $\iota\colon\bar{E}\to\bbC$ is a field embedding, then
$$
	|\Tr(M)|^2
	=
	\left|
	\sum_{i=1}^d \beta_i
	\right|^2
	\leq
	\sum_{i=1}^d |\beta_i|^2
	= dq^{w}
$$
as claimed.
\end{proof}

The representation $\rho$ is \defi{punctually ($\iota$-)pure of weight $w$} iff $\phivl{T}$ is ($\iota$-)pure of $q^{d_v}$-weight $w$ for all $v\in\PP\ssm S$.  Equivalently, we want $\phivl{T^{d_v}}$ to be pure of $q$-weight $w$ for all $v\not\in S$.  The modifier punctually should remind the reader the definition is local.

\begin{theorem}\label{thm:deligne}
If $\rho$ is punctually $\iota$-pure of weight $w$, then the cohomological factors $P_{n,\CC}(T,\rho)$ are $\iota$-mixed of $q$-weights $\leq w+n$ and the factors $P_n(T,\rho)$ are $\iota$-pure of $q$-weight $w+n$.
\end{theorem}

\begin{proof}
See Theorems 1 and 2 of \cite{Deligne:WeilII} for the respective assertions about $P_{n,\CC}(T,\rho)$ and $P_n(T,\rho)$.
\end{proof}

Let $\FF$ be a middle-extension sheaf on $\Ponet$.  We say that $\FF$ is \defi{punctually ($\iota$-)pure of weight $w$} iff for some dense Zariski open subset $U\seq\Ponet$ on which $\FF$ is lisse, the corresponding representation of $\piOne{U}$ is punctually ($\iota$-)pure of weight $w$.

\begin{lemma}\label{lem:weight-over-Z}
Let $j\colon U\to\Ponet$ be the inclusion of a dense Zariski open subset and $Z=\Ponet\ssm U$.  If $\FF$ is lisse on $U$ and punctually $\iota$-pure of weight $w$, then $\det(1-T\Frob_q\mid H^0(\Zbar,j_*\FF))$ is $\iota$-mixed of $q$-weights $\leq w$.
\end{lemma}

\begin{proof}
See \cite[1.8.1]{Deligne:WeilII}.
\end{proof}


\subsection{Semisimplicity and irreducibility}\label{subsec:semisimplicity-and-irreducibility}

Consider an exact sequence of $\GKS$-modules
\begin{equation}\label{eq:V-es}
	0
	\longto \Vl_1
	\longto \Vl
	\longto \Vl_2
	\longto 0,
\end{equation}
and let $\rho_i\colon\GKS\to\GL(\Vl_i)$ be the corresponding structure homomorphism for $i\in\{1,2\}$.  A priori, \eqref{eq:V-es} does not split, but we say $\rho$ is \defi{arithmetically semisimple} iff the sequence splits for \emph{every} $\GKS$-invariant subspace $\Vl_1\seq \Vl$.  By Clifford's theorem, the condition implies that $\rho$ is \defi{geometrically semisimple} since $\GKSbar$ is normal in $\GKS$ (cf.~\cite[49.2]{CurtisReiner}), that is, every $\GKSbar$-invariant subspace of $\Vl$ has a $\GKSbar$-invariant complement, but the converse need not be true.

We say that $\rho$ is \defi{geometrically simple} iff $\rho$ is irreducible and geometrically semisimple.  It is equivalent to assuming $\ME{\rho}$ is \defi{geometrically irreducible}, that is, there are no non-zero proper subsheaves over $\Fqbar$.

\begin{prop}\label{prop:pure-impliessemisimple}
If $\rho$ is punctually $\iota$-pure, then it is geometrically semisimple, and in particular, the subspace of $\Vl$ of $\GKSbar$-invariants is trivial if and only if the quotient space of $\Vl$ of $\GKSbar$-coinvariants is trivial.
\end{prop}

\begin{proof}
One can rephrase semisimplicity for $\rho$ in terms of semisimplicity for $\ME{\rho}$ (cf.~\cite[5.1.7]{BBD}).  It follows that both are geometrically semisimple of $\rho$ is $\iota$-pure (see \cite[5.3.8]{BBD}).  In particular, $\rho$ has trivial $\GKS$-invariants if and only if it has trivial $\GKS$-coinvariants, hence $H^0(\PonetBar,\ME{\rho})$ vanishes if and only if $H^2(\PonetBar,\ME{\rho})$ does.
\end{proof}

\begin{cor}\label{cor:poly-L-function}
If $\rho$ is punctually $\iota$-pure, then the following are equivalent:
\begin{enum}
\item\label{enum:proper-rational} $L(T,\rho)$ is in $\Qbar(T)$ but not $\Qbar[T]$;
\item\label{enum:both-spaces} $V^{\GKSbar}$ and $V_{\GKSbar}$ vanish;
\item\label{enum:both-polys} $P_0(T,\rho)$ and $P_2(T,\rho)$ are non-trivial polynomials in $\Qbar[T]$;
\item\label{enum:one-space} $V_{\GKSbar}$ vanishes;
\item\label{enum:one-poly} $P_2(T,\rho)$ is a non-trivial polynomial in $\Qbar[T]$.
\end{enum}
\end{cor}

\begin{proof}
On one hand, Theorem~\ref{thm:deligne} implies that the cohomological factors $P_n(T,\rho)$ are relatively prime, so \eqref{enum:proper-rational} and \eqref{enum:both-polys} are equivalent.  Moreover, \eqref{enum:both-spaces} and \eqref{enum:both-polys} (resp.~\eqref{enum:one-space} and \eqref{enum:one-poly}) are equivalent by \eqref{eqn:invariants-and-coinvariants} and \eqref{eq:def-P_n(T)}.  On the other hand, Proposition~\ref{prop:pure-impliessemisimple} implies that $P_0(T,\rhochi)$ is trivial if and only if $P_2(T,\rhochi)$ is trivial, so \eqref{enum:both-polys} and \eqref{enum:one-poly} are equivalent.
\end{proof}

\begin{cor}\label{cor:pure-trivial-invs}
If $\rho$ is punctually $\iota$-pure and has trivial geometric invariants, then $H^i(\PonetBar,\ME{\rho})$ and $H^i_c(\Ubar,\ME{\rho})$ vanish for $i\neq 1$, and there is an exact sequence
\begin{equation}\label{eq:ME(rho)-exact-sequence}
	0
	\longto H^0(\Zbar,\ME{\rho})
	\longto H^1_c(\Ubar,\ME{\rho})
	\longto H^1(\PonetBar,\ME{\rho})
	\longto 0.	
\end{equation}
Therefore $L(T,\rho)=P_1(T,\rho)$ and $\LC(T,\rho)=P_{1,\CC}(T,\rho)$.
\end{cor}

\begin{proof}
Suppose $\rho$ is punctually $\iota$-pure and has trivial geometric invariants so that Proposition~\ref{prop:pure-impliessemisimple} implies $\rho$ has trivial geometric coinvariants.  We claim $H^i(\PonetBar,\ME{\rho})$ vanishes for $i\neq 1$.  The Corollary then follows by observing that \eqref{eq:ff-ex-seq-head} simplifies to \eqref{eq:ME(rho)-exact-sequence} and that $H^2_c(\Ubar,\ME{\rho})$ vanishes by \eqref{eq:ff-ex-seq-piece}.

The claim is independent of $U$, so up to shrinking $U$, we suppose $j^*\ME{\rho}$ is lisse.  Then
$$
	H^0(\PonetBar,\ME{\rho})
	=
	H^0(\Ubar,\ME{\rho})
	\mbox{ and }
	H^2(\PonetBar,\ME{\rho})
	=
	H^2_c(\Ubar,\ME{\rho})
$$
are the subspace of $\piOneUbar$-invariants and (a Tate twist of the) quotient space of $\piOneUbar$-coinvariants respectively of $\Vl$ by \eqref{eqn:invariants-and-coinvariants}.  The claim is also independent of $\SS$, so up to replacing $\SS$ by a finite superset in $\PP$, we suppose $\rho$ factors through a natural quotient $\GKSbar\onto\piOneUbar$.  Then the cohomology spaces in question are the $\GKSbar$-invariants and $\GKSbar$-coinvariants of $\Vl$, which are trivial by hypothesis, so $H^i(\PonetBar,\ME{\rho})$ vanishes for $i\neq 1$ as claimed.
\end{proof}

\begin{cor}\label{cor:good-vs}
The following are equivalent:
\begin{enum}
\item\label{li:a:3} $\MC(T,\rho)=1$, that is, $\ME{\rho}$ is supported on $\Aonet[1/\Q]$;
\item\label{li:a:4} $\LC(T,\rho)$ is a polynomial which is $\iota$-pure of $q$-weight $w+1$.
\end{enum}
\end{cor}

\noindent
Note, $\MC(T,\rho)$ is the $L$-function of the restriction of $\ME{\rho}$ to $Z$, so the former is trivial if and only if the latter is.  

\begin{proof}
If \eqref{li:a:3} holds, then the subspace of $I(\infty)$-invariants of $\Vl$ is trivial, so a fortiori, the subspace of $\GKSbar$-invariants is trivial.  Therefore Corollary~\ref{cor:pure-trivial-invs} implies $\LC(T,\rho)$ equals $L(T,\rho)=P_1(T,\rho)$ and hence Theorem~\ref{thm:deligne} implies \eqref{li:a:4} holds.  

If \eqref{li:a:4} holds, then $P_{2,\CC}(T,\rho)$ divides $P_{1,\CC}(T,\rho)$ by \eqref{eq:L-fun:frac:partial}.  Theorem~\ref{thm:deligne} implies $P_{2,\CC}(T,\rho)=P_2(T,\rho)$ is $\iota$-pure of $q$-weight $w+2$, so it is coprime to $P_{1,\CC}(T,\rho)$ and hence trivial.  Therefore $H^2(\PonetBar,\ME{\rho})$ vanishes, and hence $H^0(\PonetBar,\ME{\rho})$ also vanishes since $\rho$ is geometrically semisimple.  That is, $\rho$ has trivial geometric invariants.  Moreover, $1/\MC(T,\rho)$ is a polynomial which is $\iota$-mixed of $q$-weights $\leq w$ by Lemma~\ref{lem:weight-over-Z} while $L(T,\rho)$ is a polynomial which is $\iota$-pure of $q$-weight $w$, so Corollary~\ref{cor:pure-trivial-invs} implies \eqref{li:a:3} holds.
\end{proof}


\subsection{Main Theorem}\label{sec:proof-of-archimedean-bound}

The following theorem is the main result of Section~\ref{sec:l-functions}.  The essential ingredient it uses is Deligne's Riemann hypothesis.

\newcommand\thmA{
Suppose $\rho$ is punctually $\iota$-pure of weight $w$.  Then $\deg(\LC(T,\rho))=\rC(\rho)$ and $\deg(L(T,\rho))=\degL(\rho)$.  Moreover, $\rho$ has trivial geometric invariants if and only if $L(T,\rho)$ is a polynomial if and only if the cohomological factor $P_{2,\CC}(T,\rho)$ is trivial.  If these conditions hold, then $\LC(T,\rho)$ is the polynomial $P_{1,\CC}(T)$ and is $\iota$-mixed of $q$-weights $\leq w+1$, and $L(T,\rho)$ is its largest $\iota$-pure factor of $q$-weight $w+1$.
}

\begin{theorem}\label{thm:archimedean-bound}
\thmA
\end{theorem}

%
%

\begin{proof}
Suppose $\rho$ is punctually $\iota$-pure.  Corollary~\ref{cor:poly-L-function} implies that it has trivial geometric invariants if and only if $L(T,\rho)$ is a polynomial if and only if $P_{2,\CC}(T,\rho)$ is trivial, so suppose these equivalent conditions hold.  On one hand, Corollary~\ref{cor:pure-trivial-invs} implies $L(T,\rho)=P_1(T,\rho)$ and $\LC(T,\rho)=P_{1,\CC}(T,\rho)$, so both are polynomials are claimed.  Moreover, Proposition~\ref{prop:chi-Ponet} implies
$$
	\deg(L(T,\rho))
	=
	\dr(\rho)
	+
	\swan(\rho)
	-
	2\cdot\dim(\Vl)
	=
	\degL(\rho)
$$
and Corollary~\ref{cor:chi_c} implies
 $$
	\deg(\LC(T,\rho))
	=
	-\chi_c(\Aonet[1/\Q],\ME{\rho})
	=
	\rC(\rho)
$$
as claimed.  On the other hand, Theorem~\ref{thm:deligne} implies $L(T,\rho)$ is pure of $q$-weight $w+1$ and $\LC(T,\rho)$ is mixed of $q$-weights $\leq w+1$ since $\rho$ is punctually pure of weight $w$.  Moreover, Lemma~\ref{lem:weight-over-Z} implies that $\LC(T,\rho)/L(T,\rho)=1/\MC(T,\rho)$ is a polynomial which is $\iota$-mixed of $q$-weights $\leq w$, so $L(T,\rho)$ is the largest $\iota$-pure factor of $\LC(T,\rho)$ of $q$-weight $w+1$ as claimed.
\end{proof}


\section{Twisted $L$-functions}\label{sec:twisted-l-functions}

Recall we have a finite-dimensional $\El$-vector space $V$ and a (continuous) representation
$$
	\rho\colon \GKS\to\GL(V).
$$
We fix a field embedding $\iota\colon\Qbar\to\bbC$ and suppose $\rho$ is punctually $\iota$-pure of weight $w$ so that we can apply the results of the previous section.

Let $\s,\Q\in\Fq[t]$ be monic and square free, and suppose $\SS\sub\PP$ is the finite subset consisting of $\infty$ and $v(\pi)$ for every prime factor $\pi$ of $\s$.  Let $\CC\sub\PP$ be defined similarly and \defi{$\RR=\SS\cup\CC$}.

Let \defi{$\BQ$} be the finite group $\BQFq$ and \defi{$\PhiQ$} be the dual group of all Dirichlet characters
$$
	\dc\colon
	\BQ\to\Qlbartimes
$$
of conductor dividing $\Q$.  For each $\dc$, we define a twisted representation
$$
	\rhochi\colon \GKR\to\GL(V_\dc)
$$
where $V_\dc=V$ as $\El$-vector spaces (see \S\ref{subsec:tensor-products}).  We show that $\rhochi$ is also punctually $\iota$-pure of weight $w$ and that the corresponding $L$-functions
$$
	L(T,\rhochi)
	=
	\prod_{v\in\PP} L(T^{d_v},(\rhochi)_v)^{-1},
	\quad
	\LC(T,\rhochi)
	=
	\prod_{v\not\in\CC} L(T^{d_v},(\rhochi)_v)^{-1}
$$
are $\iota$-mixed (see \S\ref{subsec:tensor-products} and \S\ref{subsec:induced-representations}).

\newcommand\thmB{
Suppose $\rho$ is $\iota$-pure of weight $w$ and $\dc\in\PhiQ$.  Then
$$
	\deg(\LC(T,\rhochi))
	=
	\rC(\rho)
	=
	\deg(L(T,\rho)) + (\deg(\Q)+1)\dim(V) - \dropCee{\rho}.
$$
Moreover, $\rhochi$ has trivial geometric invariants if and only if $L(T,\rhochi)$ is a polynomial if and only if $\PC{2}(T,\rhochi)$ is trivial.  If these conditions hold, then $\LC(T,\rhochi)$ is the polynomial $\PC{1}(T)$ and is $\iota$-mixed of $q$-weights $\leq w+1$, and $L(T,\rhochi)$ is its largest $\iota$-pure factor of $q$-weight $w+1$.
}

\begin{theorem}\label{thmB}
\thmB
\end{theorem}

\noindent
The proof is in \S\ref{sec:proof-of-thmB}.


\subsection{Dirichlet characters}

By definition, each $\dc\in\PhiQ$ is a homomorphism $\BQ\to\Qlbartimes$.  There is also a quotient $G_{K,\CC}\onto\BQ$ from abelian class field theory, and we write
$$
	\dcc\colon 
	G_{K,\CC}\to\GL_1(\El)
$$
for the composition of these maps and the canonical isomorphism $\Qlbartimes\to\GL_1(\El)$.  The corresponding middle-extension sheaf \defi{$\ME{\dc}$} is a so-called Kummer sheaf.  It is tamely ramified over $\CC$ since the hypothesis that $\Q$ is square free implies that $\BQ$ has order prime to $p$ and thus $\dcc(P(t))$ is trivial for every $t\in\CC$.

There is a natural quotient $\GKR\onto\GKC$ since $\CC\seq\RR$, and we write \defi{$\dcr$} for the composition of this quotient and $\dcc$.


\subsection{Tensor products}\label{subsec:tensor-products}

The \defi{tensor product} of $\rho$ and $\dc$ is the representation
$$
	\rho\otimes\dc \colon \GKR\to\GL(V_\dc)
$$

\smallskip\noindent
given by $(\rhochi)(g)=\rho(g)\dcc(g)$ where $V_\dc=V$ as $\El$-vector spaces.  The corresponding Euler factors are given by
$$
	L(T,(\rhochi)_v)
	=
	\det(1 - T\,(\rhochi)_v(\Frob_v)\mid V^{I(v)}_\dc),
$$
and in particular,
\begin{equation}\label{eq:twisted-euler}
	L(T,(\rhochi)_v)
	=
	L(\dcc(\Frob_v)T,\rho_v)
\end{equation}
for $v\not\in\CC$.

\begin{lemma}\label{lem:rhochi-pure}\ 

\begin{enum}
\item\label{lem:item:rhochi-pure--simplicity} If $\rho$ is geometrically simple, then so is $\rhochi$.
\item\label{lem:item:rhochi-pure--purity} If $\rho$ is punctually $\iota$-pure of weight $w$, then so is $\rhochi$.
\end{enum}
\end{lemma}

\begin{proof}
If $W_\dc\seq V_\dc$ be a $\bar{G}_{K,\RR}$-invariant subspace, then $W=W_\dc\otimes\bar\dc$ is a $\bar{G}_{K,\RR}$-invariant subspace.  Moreover, if $\rho$ is geometrically simple, then $W$ equals $\ZeroSpace$ or $V$, hence $W_\dc$ equals $\ZeroSpace$ or $V_\dc$.  Thus \eqref{lem:item:rhochi-pure--simplicity} holds.

Observe that $\zeta=\dcc(\Frob_v)$ is a root of unity since $\BQ$ has finite order, hence $\zeta\in\Qbar$ and $|\iota(\zeta)|^2=1$.  If $v\not\in\CC$ and if $\alpha\in\Qbar$ is a zero of $L(T,(\rhochi)_v)$, then \eqref{eq:twisted-euler} implies that $\alpha/\zeta$ is a zero of $L(T,\rho_v)$.  In particular, $|\alpha|^2=|\alpha/\zeta|^2=(1/q^{d_v})^w$, hence $L(T^{d_v},(\rhochi)_v)$ is $\iota$-pure of $q$-weight $w$ for almost all $v$.  Thus \eqref{lem:item:rhochi-pure--purity} holds.
\end{proof}

Therefore we can apply Theorem~\ref{thm:archimedean-bound} to $\rhochi$.

\begin{lemma}
$
	\dr(\rhochi)
	-
	\dr(\rho)
	=
	\dropCee{\rhochi}
	-
	\dropCee{\rho}
$
and
$
	\swan(\rhochi)
	=
	\swan(\rho).
$
\end{lemma}

\begin{proof}
If $v\in\PP$, then $\swan_v(\rhochi)=\swan_v(\rho)$ since tensoring with tamely ramified character (e.g., $\dc$) does not change the local Swan conductor.  Moreover, if $v\not\in\CC$, then $V$ and $V_\dc$ are isomorphic as $I(v)$-modules, and thus $L(T,\rho_v)$ and $L(T,(\rhochi)_v)$ have the same degree, that is, $\dr_v(\rhochi)=\dr_v(\rho)$.
\end{proof}

\begin{cor}\label{cor:rC-independent-of-chi}
$
	\rC(\rhochi)
	=
	\rC(\rho).
$
\end{cor}

\begin{proof}
Combine the lemma and \eqref{eq:def-rC} to deduce
\begin{align*}
	\rC(\rhochi)
	& =
	\dr(\rhochi)
	-
	\dropCee{\rhochi}
	+
	\swan(\rhochi)
	+
	(\deg(\Q)-1)\cdot\dim(V) \\
	& =
	\dr(\rho)
	-
	\dropCee{\rho}
	+
	\swan(\rho)
	+
	(\deg(\Q)-1)\cdot\dim(V)
	\ = \rC(\rho)
\end{align*}
as claimed.
\end{proof}


\subsection{Induced representations}\label{subsec:induced-representations}

Let $L=\Fq(u)$ be the subfield of $K$ corresponding to the finite cover $\Q\colon\Ponet\to\Poneu$, and let $\SS$ be a finite set of places in $L$ including those lying below $\RR$ and those which ramify in $L/K$.  Then for each $\dc\in\PhiQ$, we have an induced representation
$$
	\Ind(\rhochi)\colon G_{L,\SS}\to\GL(\Ind(V_\dc))
$$
where $\Ind(V_\dc)$ is a vector space of dimension $n\cdot\dim(V_\dc)$.

\begin{lemma}\label{lem:Ind-rhochi-pure}
If $\rho$ is punctually $\iota$-pure of weight $w$, then so is $\Ind(\rhochi)$.
\end{lemma}

\begin{proof}
\newcommand\w{{\bar{v}}}
Let $\w$ be a place in $L$ not lying and $\SS$, and let $v|\w$ denote any place in $K$ lying over $\w$.  Then
$$
	L(T^{\deg(\w)},\Ind(\rhochi)_\w)
	=
	\prod_{v\mid \w}
	L(T^{\deg(v)},(\rhochi)_v).
$$
In particular, Lemma~\ref{lem:rhochi-pure}.\ref{lem:item:rhochi-pure--purity} implies the factors on the right are $\iota$-pure of $q$-weight $w$, so the left side is also $\iota$-pure of $q$-weight $w$.
\end{proof}


\subsection{Proof of Theorem~\ref{thmB}}\label{sec:proof-of-thmB}

If $\rho$ is punctually $\iota$-pure of $q$-weight $w$, then so is $\rhochi$ by Lemma~\ref{lem:rhochi-pure}.\ref{lem:item:rhochi-pure--purity}.  Hence $\rhochi$ also has trivial geometric invariants, then we can apply Theorem~\ref{thm:archimedean-bound}.  In particular, we deduce that $\LC(T,\rhochi)$ and $L(T,\rhochi)$ are polynomials of respective degrees $\rC(\rhochi)$ and $\degL(\rhochi)$, that $\LC(T,\rhochi)$ is $\iota$-mixed of $q$-weights $\leq w+1$, and that $L(T,\rhochi)$ is the largest factor of $\LC(T,\rhochi)$ which is $\iota$-pure of $q$-weight $w+1$.  Finally, we observe that
\begin{eqnarray*}
	\rC(\rhochi)
	& \overset{\mathrm{Cor.~}\ref{cor:rC-independent-of-chi}}{=} &
	\rC(\rho) \\
	& \overset{\eqref{eq:def-rC}}{=} &
	\dr(\rho)
	-
	\dropCee{\rho}
	+
	\swan(\rho)
	+
	(\deg(\Q)-1)\cdot\dim(\Vl) \\
	& \overset{\mathrm{Th.~}\ref{thm:archimedean-bound}}{=} &
	\deg(L(T,\rho))
	+
	(\deg(\Q)+1)\cdot\dim(\Vl)
	-
	\dropCee{\rho}
\end{eqnarray*}
as claimed.


\section{Statement of Equidistribution}\label{sec:equidistribution}

Recall we have an $\El$-vector space $V$ of finite dimension \defi{$\r$} and a (continuous) representation
$$
	\rho\colon \GKS\to\GL(V)
$$
which is punctually pure of weight $w$.  We also have monic square free $\s,\Q\in\Fq[t]$ and corresponding finite subsets $\SS,\CC\sub\PP$ of supporting places.

In this section, we consider the partial $L$-functions $\LC(T,\rhochi)$ as $\chi$ varies over $\PhiQ$ and regard them as a proxy for coefficients in a Mellin transform of $\rho$.  One can easily show that there are hardly any characters $\dc\in\PhiQ$ such that $\rhochi$ has \emph{non-}trivial geometric invariants, and otherwise, having trivial geometric invariants implies $\LC(T,\rhochi)$ is a polynomial in $\bbQbar[T]$ of degree $\R=\rC(\rho)$ by Theorem~\ref{thmB}.  Moreover, the subset
\begin{equation}\label{eq:phi-good}
	\PhiQGood\rho
	=
	\left\{\,
	\dc\in\PhiQ
	:
	\LC(T,\rhochi)=L(T,\rhochi)\in\bbQbar[T]
	\,\right\}	
\end{equation}

\smallskip\noindent
is `big' (see Corollary~\ref{cor:bad-bound-for-PhiQ}) and consists of all $\dc$ for which
\begin{equation}\label{eqn:LCStar}
	\LCStar(T,\rhochi)
	=
	\LC(T/(\sqrt{q})^{1+w},\rhochi)
\end{equation}
is pure of $q$-weight zero by Theorem~\ref{thmB}.   In particular, for each $\dc\in\PhiQGood\rho$, $\LCStar(T,\rhochi)$ is the characteristic polynomial of a unitary element of $\GL_\R(\bbC)$, so there is a unique conjugacy class $\trhochi$ of $U_\R(\bbC)\seq\GL_\R(\bbC)$ whose elements have the same characteristic polynomial.  We would to know whether or not they are equidistributed.

We say the multiset
$
	\ThetaRhoq
	=
	\{\,\trhochi:\dc\in\PhiQGood\rho\,\}
$
of conjugacy classes becomes \defi{equidistributed in $U_\R(\bbC)$ as $q\to\infty$} iff, for every continuous central function $f\colon U_\R(\bbC)\to\bbC$, one has
\begin{equation}\label{eqn:trace-limit}
	\lim_{q\to\infty}
	\frac{1}{|\PhiQGood\rho|}
	\sum_{\dc\in\PhiQGood\rho}f(\trhochi)
	=
	\int_{U_\R(\bbC)} f(\theta)d\theta
\end{equation}
where $d\theta$ is the unique Haar probability measure on $U_\R(\bbC)$.  Equivalently, by the Peter-Weyl theorem, one has equidistribution if and only if for every irreducible finite-dimensional representation $\Lambda\colon U_\R(\bbC)\to\GL_{\dim(\Lambda)}(\bbC)$ and for $f=\Tr\circ\Lambda$, the identity in \eqref{eqn:trace-limit} holds.

In principle, one could try to exhibit equidistribution for all of $\ThetaRhoq$ at once.  Instead we follow Katz and (try to) prove simultaneous and uniform equidistribution for certain one-parameter families of characters.  More precisely, we partition $\PhiQ$ into cosets $\dc\PhiUNu$ of a subgroup $\PhiUNu$ (defined in \S\ref{sec:one-parameter-families}) and (try to) prove equidistribution for characters in
\begin{equation}
	\dc\PhiUNuGood{\rho}=\dc\PhiUNu\cap\PhiQGood\rho.
\end{equation}
Doing so for a single coset is equivalent to showing that an associated monodromy group we denote $\Ggeom{\dc}{\rhol}$ equals $\GL_{\R,\Qellbar}$.  See \S\ref{sec:one-parameter-families}, \S\ref{sec:properties-preserved-by-Q_*}, and \S\ref{subsec:tannakian-monodromy-groups}.

The monodromy group is an algebraic subgroup of $\GL_{\R,\Qellbar}$.  We say the former is \defi{big} iff it equals the latter, and we write
\begin{equation}\label{eqn:big-defn}
	\PhiQBig\rho
	=
	\{\,
	\dc\in\PhiQ
	:
	\Ggeom{\dc}{\rho}\mbox{ is big}
	\,\}
\end{equation}
for the subset of big characters.  We say that the \defi{Mellin transform} of $\rho$ has \defi{big monodromy} in $\GL_{\R,\Qellbar}$ iff
\begin{equation}\label{eqn:big-monodromy-bis}
	|\PhiQBig\rho|
	\sim
	|\PhiQGood\rho|
	\mbox{ as }
	q\to\infty,
\end{equation}
or equivalently (cf.~Corollary~\ref{cor:bad-bound-for-PhiQ}),
\begin{equation}\label{eqn:big-monodromy}
	|\PhiQBig\rho|
	\sim
	|\PhiQ|
	\mbox{ as }
	q\to\infty.
\end{equation}

\begin{theorem}\label{thm:big-monodromy-implies-equidistribution}
Suppose $\rho$ is punctually $\iota$-pure and $\dc$ is in $\PhiQBig\rho$.
Let $\Lambda\colon U_\R(\bbC)\to\GL_{\dim(\Lambda)}(\bbC)$ be a finite-dimensional representation.  If $q$ is sufficiently large, then
\begin{equation}\label{eqn:sum-to-integral}
	\frac{1}{|\dc\PhiUNuGood{\rho}|}
	\sum_{\dc'\in\dc\PhiUNuGood{\rho}}
	\Tr\,\Lambda(\trhochip)
	=
	\int_{U_\R(\bbC)}\Tr\,\Lambda(\theta)\,d\theta
	+
	o(1)
	\mbox{ as }
	q\to\infty,
\end{equation}
and the implicit constant depends only on $\r=\dim(V)$ and $\dim(\Lambda)$.  In particular, if the Mellin transform of $\rho$ has big monodromy, then $\ThetaRhoq$ is equidistributed in $U_\R(\bbC)$.
\end{theorem}

\noindent
The proof is in \S\ref{sec:proof-of-thmD}.

\begin{remark}
Observe that the $q$-weight $w$ of $\rho$ plays no role in the statement of the theorem.  This is because we factored out the weight in the normalization \eqref{eqn:LCStar}.  Another way to achieve the same renormalization is to replace $\rho$ by an appropriate Tate twist so that $w=-1$ and 
$
	\LCStar(T,\rhochi)
	=
	\LC(T,\rhochi).
$
\end{remark}


\subsection{Reduction to $\bbG_m$}

Let $\Ponet$ and $\Poneu$ denote the projective $t$-line and $u$-line respectively, and let $\kPu=\Fq(u)$.  The function-field embedding $\kPu\to K$ generated by $u\mapsto\Q$ corresponds to a finite morphism $\Q\colon \Ponet\to\Poneu$.  The morphism has generic degree $n=\deg(\Q)$ and is generically etale since $\Q$ is square free of degree $n$, and it fits in a commutative diagram
$$
	\xymatrix{
	\Aonet[1/\Q]\ar[r]\ar[d]_\Q & \Ponet\ar[d]^\Q & \div(\Q)\ar[l]\ar[d]^\Q \\
	\bbG_m\ar[r] & \Poneu & \{0,\infty\}\ar[l]
	}
$$
where the outer vertical maps are finite morphisms.  There are canonical identifications of $\div(\Q)$ with $\CC$ and $\{0,\infty\}$ with a set $\CCp$ composed of two places of the function field $\Fq(u)$.  

For any sheaf $\FF$ on the $t$-line, one can define the direct image sheaf $\Q_*\FF$.

On one hand, the geometric generic fiber of $\FF=\Q_*\ME{\rho}$ is the induced representation
$$
	\Ind(\rho) \colon G_{K'}\to\GL(\Ind(V))
$$
where $\Ind(V)$ is a vector space of dimension $n\cdot\dim(V)$ (cf.~\cite[II.3.1.e]{Milne}).  
Moreover, if $\ubar$ is a geometric closed point of $\Poneu$, that is, a closed point of $\Ponet\times_{\Fq}\Fqbar$, and if $\Q^{-1}(\ubar)=\{\tbar_1,\ldots,\tbar_m\}\sub\Ponet\times_{\Fq}\Fqbar$, then the various geometric fibers satisfy
\begin{equation}\label{eq:direct-image-fiber}
	(\Q_*\FF)_{\ubar}
	=
	H^0(\ubar,\Q_*\FF)
	=
	\bigoplus_{i=1}^m H^0(\tbar_i,\FF)
	=
	\bigoplus_{i=1}^m \FF_{\tbar_i}
\end{equation}
as $\Qellbar$-vector spaces (cf.~\cite[II.3.5.c]{Milne}).  In particular, if $\FF$ is supported on $\Aonet[1/\Q]$, then $\Q_*\FF$ is supported on $\bbG_m$.

On the other hand, the functorial properties of $\Q_*$ yield canonical isomorphisms
\begin{equation}\label{eq:cohomology-comparison}
	H^n(\PonetBar,\FF)
	=
	H^n(\PonetBar,\Q_*\FF)
	\mbox{\ \ and\ \ }
	H^n_c(\AonetBar[1/\Q],\FF)
	=
	H^n_c(\bar\bbG_m,\Q_*\FF)
\end{equation}
for each $n$.  For example, $\Q_*$ is exact since $\Q$ is a finite map, so the first identity in \eqref{eq:cohomology-comparison} is a consequence of the (trivial) Leray spectral sequence (cf.~\cite[II.3.6 and III.1.18]{Milne}).  In particular, the identities \eqref{eq:L-fun:frac:partial}, \eqref{eq:L-fun:frac}, and \eqref{eq:cohomology-comparison} jointly imply that
\begin{equation}
	L(T,\ME{\rhochi})
	=
	L(T,\Q_*\ME{\rhochi})
	\mbox{\ and\ }
	\LC(T,\ME{\rhochi})
	=
	L_{\CCp}(T,\Q_*\ME{\rhochi})
\end{equation}
for $\dc\in\PhiQ$.


\subsection{One-parameter families}\label{sec:one-parameter-families}

Recall $\Q\in\Fq[t]\sub\kPt$ is monic and square free and $\kPu\to\kPt$ is the function-field embedding which sends $u$ to $\Q$.  The norm map $\kPt\to\kPu$ is multiplicative and sends $t$ to $(-1)^{n}u$ for $n=\deg(\Q)$.  It also induces homomorphisms
$$
	\nu
	\colon
	\BQ
	\to
	\Bu
	\mbox{\ \ and\ \ }
	\nu^*
	\colon
	\PhiU
	\to
	\PhiQ
$$
where $\Bu=(\Fq[u]/u\Fq[u])^\times$ and $\PhiU$ is its dual.  In particular, $\nu$ is surjective, so its dual $\nu^*$ is injective, and we can identify $\PhiU$ with its image $\PhiUNu$.  Moreover, as the following lemma shows, twisting by elements of the coset $\dc\PhiUNu$ is the `same' as twisting by elements of $\PhiU$.

\begin{lemma}\label{lem:direct-image-of-middle-extension}
Let $\dc\in\PhiQ$ and $\alpha\in\PhiU$.

\medskip
\begin{enum}
\item\label{lem:ind-is-me} $\Q_*\ME{\rhochi}$ is isomorphic to $\ME{\Ind(\rhochi)}$.
\item\label{lem:ind-projection} $\Q_*\ME{\rhochi\alpha^\nu}$ is isomorphic to $\ME{\Ind(\rhochi)\otimes\alpha}$.
\end{enum}
\end{lemma}

\begin{proof}
By \cite[3.3.1]{Katz:TLFM}, $\Q_*\ME{\rhochi}$ is a middle extension, and since it is generically equal to the middle extension sheaf $\ME{\Ind(\rhochi)}$, Proposition~\ref{prop:assoc-me} implies part \eqref{lem:ind-is-me} holds.

Up to replacing $\rho$ by $\rhochi$, we suppose without loss of generality that $\dc=\chinot$.  Let $T\seq\Ponet$ be a dense Zariski open subset and $U=\Q(T)$.  Suppose that $U\seq\Gm$ so that $\Q^*\ME{\alpha}$ is lisse on $T$, that the restriction $\Q\colon T\to U$ is \etale, and that $\ME{\rho}$ is lisse on $T$.  Let $i\colon T\to\Ponet$ and $j\colon U\to\Poneu$ be the  inclusions.  We have
$$
	\ME{\rho\otimes\alpha^\nu}
	\simeq
	i_*i^*(\ME{\rho\otimes\alpha^\nu})
	\simeq
	i_*i^*(\ME{\rho}\otimes\ME{\alpha^\nu})
	\simeq
	i_*i^*(\ME{\rho}\otimes\Q^*\ME{\alpha})
$$
since each of the sheaves is a middle extensions and lisse on $T$.  Therefore the projection formula implies
$$
	\Q_*\ME{\rho\otimes\alpha^\nu}
	\simeq
	\Q_*(i_*i^*(\ME{\rho}\otimes\Q^*\ME{\alpha}))
	\simeq
	j_*j^*(\Q_*\ME{\rho}\otimes\ME{\alpha})
$$
since each of the sheaves is lisse on $U$ and a middle extension on $\Poneu$ (by part \eqref{lem:ind-is-me}) and since $\Q\colon T\to U$ is \etale.  Finally,
\begin{align*}
	j_*j^*(\Q_*\ME{\rho}\otimes\ME{\alpha})
	\simeq
	j_*j^*(\ME{\Ind(\rho)}\otimes\ME{\alpha})
	\simeq
	\ME{\Ind(\rho)\otimes\alpha}
\end{align*}
and thus part \eqref{lem:ind-projection} holds.
\end{proof}


\subsection{Properties preserved by $\Q_*$}\label{sec:properties-preserved-by-Q_*}

We say a character $\dc\in\PhiQ$ is \defi{good for $\rho$} or simply \defi{good} iff it lies in the subset $\PhiQGood\rho$ defined in \eqref{eq:phi-good}.  When $\Q=t$ and thus $\Aonet[1/\Q]=\bbG_m$, then Lemma~\ref{lem:direct-image-of-middle-extension} and the following lemma together show that our notion of good coincides with that of Katz's (cf.~\cite[Chapter 3]{Katz:CE}):

\begin{lemma}
If $\dc\in\PhiQ$ and $\alpha\in\PhiU$, then the following are equivalent:

\smallskip
\begin{enum}
\item\label{li:good-1} $\dc\alpha^\nu$ is good for $\rho$;
\item\label{li:good-2} $\ME{\rhochi\alpha^\nu}$ is supported on $\Aonet[1/\Q]$;
\item\label{li:good-3} $\ME{\Ind(\rhochi)\otimes\alpha}$ is supported on $\bbG_m$;
\item\label{li:good-4} $\alpha\in\PhiU$ is good (\`a la Katz) for $\Q_*\ME{\rhochi}$.
\end{enum}
\end{lemma}

\begin{proof}
Corollary~\ref{cor:good-vs} implies the first conditions \eqref{li:good-1} and \eqref{li:good-2} are equivalent.  Conditions \eqref{li:good-2} and \eqref{li:good-3} are equivalent by the identity in \eqref{eq:direct-image-fiber} for $\ubar\in\CC'$.  Finally, taking $\Q=t$ and applying the equivalence of \eqref{li:good-1} and \eqref{li:good-2} yields the equivalence of \eqref{li:good-3} and \eqref{li:good-4}.
\end{proof}

Let $\PhiQBad\rho$ be the complement $\PhiQ\ssm\PhiQGood\rho$ and $\dc\PhiUNuBad{\rho}=\PhiQBad\rho\cap\dc\PhiUNu$.

\begin{cor}\label{cor:bad-bound}
$|\dc\PhiUNuBad{\rho}|\leq(1+\deg(\Q))\cdot\rank(\rho)$.
\end{cor}

\begin{proof}
If $\dc\in\PhiQBad\rho$, then $\dc$ it coincides with some tame character of $\rho$ at some $v\in\CC$, and there are at most $(1+\deg(\Q))\cdot\rank(\rho)$ such characters.  Compare \cite[pp.~12--13]{Katz:CE}.
\end{proof}

\begin{cor}\label{cor:bad-bound-for-PhiQ}
$|\PhiQGood\rho|\sim|\PhiQ|$ as $q\to\infty$.
\end{cor}

\begin{proof}
Observe that Corollary~\ref{cor:bad-bound} implies
$$
	|\PhiQ| - |\PhiQGood\rho|
	=
	|\PhiQBad\rho|
	=
	\sum_{\dc\PhiUNu}
	|\PhiUNuBad\rho|
	\leq
	O(|\PhiQ|/|\PhiUNu|)
	=
	o(|\PhiQ|)
$$
as $q\to\infty$.
\end{proof}


\subsection{Tannakian monodromy groups}\label{subsec:tannakian-monodromy-groups}

Suppose $\Q=t$ and thus $\CCp=\CC=\{0,\infty\}$ and $\PhiU=\PhiQ$.      Suppose moreover that $\rho$ is geometrically simple and $\dim(V)>1$ so that no geometric subquotient of $\ME{\rho}$ is a Kummer sheaf.

Let $j\colon\bbG_m\to\Poneu$ be the inclusion, let $j_0\colon\bbG_m\to\Aoneu$ be the inclusion map, and for each $\alpha\in\PhiU$, let
$$
	\omega_\alpha(\ME{\rho})
	=
	H^1_c(\AoneuBar,j_{0*}j^*\ME{\rho\otimes\alpha}).
$$
It is a $G_{\Fq}$-module, that is, $\Frob_q$ acts functorially, and it corresponds to a well-defined conjugacy class of elements $\Frob_{\Fq,\alpha}\sub\GL(\omega(\ME{\rho}))$ where $\omega(\ME{\rho})=\omega_{\mathbf{1}}(\ME{\rho})$ and $\mathbf{1}\in\PhiU$ is the trivial character.  Moreover, if $\alpha$ is good, then
$$
	\omega_\alpha(\ME{\rho})
	=
	H^1_c(\bar\bbG_m,\ME{\rho\otimes\alpha}),
$$
and in particular
$$
	L_\CC(T,\rho\otimes\alpha)
	=
	\det(1-\Frob_\alpha T\mid \omega(\ME{\rho})).
$$

\smallskip
In a way we will not make precise here, the $\Frob_\alpha$ `generate' $\ell$-adic reductive subgroups
$$
	\Ggeom{}{\rho}\seq\Garith{}{\rho}\seq\GL_{\R,\Qellbar}
$$
which are well-defined up to conjugacy.  They are fundamental groups of certain Tannakian categories, and we call them the \defi{Tannakian monodromy groups of $\rho$}.  See Appendix~\ref{sec:tannakian-appendix} for details.  We say the Mellin transform of $\rhol$ has \defi{big Tannakian monodromy} iff $\Ggeom{}{\rho}=\GL_{\R,\Qellbar}$.

For general $\Q$ and $\dc\in\PhiQ$, we write
$$
	\Ggeom\dc\rho \seq \Garith\dc\rho \seq \GL_{\R,\Qellbar}
$$
for the Tannakian monodromy groups of $\Ind(\rhochi)$, and we say that the Mellin transform of $\rhochi$ has \defi{big Tannakian monodromy} iff $\Ggeom\dc\rho=\GL_{\R,\Qellbar}$.  Now the action of $\Frob_q$ on $\omega_\alpha(\ME{\rhochi})$ corresponds to a well-defined conjugacy class $\Frob_{\Fq,\alpha}\sub\Garith{\dc}{\rho}$.


\subsection{Proof of Theorem~\ref{thm:big-monodromy-implies-equidistribution}}\label{sec:proof-of-thmD}

%
%

We may suppose without loss of generality that $\Lambda$ is irreducible since it is semisimple and $\Tr(\Lambda_1\oplus\Lambda_2)=\Tr(\Lambda_1)+\Tr(\Lambda_2)$ for any representations $\Lambda_1,\Lambda_2$.  Moreover, one can show that
$$
	\int_{U_\R(\bbC)}\Tr\,\Lambda(\theta)\,d\theta
	=
	\begin{cases}
	1 & \Lambda\mbox{ is the trivial representation} \\
	0 & \mbox{otherwise}
	\end{cases}
$$
so to prove \eqref{eqn:sum-to-integral} we must show that
\begin{equation}\label{eq:to-show}
	\frac{1}{|\dc\PhiUNuGood{\rho}|}\sum_{\dc'\in\dc\PhiUNuGood{\rho}}
	\Tr\,\Lambda(\trhochip)
	=
	\begin{cases}
	1 & \Lambda\mbox{ is the trivial representation} \\
	o(1) & \mbox{otherwise}
	\end{cases}
\end{equation}
when $q$ is large.

If $q$ is sufficiently large, then Corollary~\ref{cor:bad-bound} implies that
$$
	|\dc\PhiUNuBad{\rho}|
	\leq
	(1+\deg(\Q))\cdot\rank(\rho)
	<
	|\dc\PhiUNu|
$$
and thus $\dc\PhiUNuGood{\rho}$ is non-empty.  In particular, the left side of \eqref{eq:to-show} is defined for large $q$, and it is identically $1$ when $\Lambda$ is the trivial representation.  On the other hand, if $\Lambda$ is non-trivial and if $q$ is bigger than $(|\dc\PhiUNuBad{\rho}|+1)^2$, then \cite[7.5]{Katz:CE} implies that
\begin{equation}\label{eqn:a-uniform-bound}
	\frac{1}{|\dc\PhiUNuGood{\rho}|}
	\left|
	\sum_{\dc'\in\dc\PhiUNuGood{\rho}}
	\!\!\!\!\Tr\,\Lambda(\trhochip)
	\right|
	\leq
	(\dim(V) + \dim(\Lambda))
	\left(
	\frac{1}{\sqrt{q}}
	+
	\frac{1}{\sqrt{q}^3}
	\right).
\end{equation}
Thus \eqref{eq:to-show} holds, as claimed, and the implicit constant depends only on $r$ and $\dim(\Lambda)$.

To complete the proof of the theorem we must show that $\ThetaRhoq$ becomes equidistributed in $U_\R(\bbC)$.  We observe that
\begin{equation}\label{eqn:uniform-bound-for-Tr}
	|\Tr\,\Lambda(\trhochip)|\leq\dim(\Lambda)
	\mbox{ for }
	\dc'\in\dc\PhiUNuGood{\rho}
\end{equation}
Therefore
\begin{equation*}\label{eqn:key-trace-estimate}
	\sum_{\dc\in\PhiQGood\rho}
	\!\!\!\!\Tr\,\Lambda(\trhochi)
	=
	\sum_{\dc\in\PhiQGoodBig\rho}
	\!\!\!\!\Tr\,\Lambda(\trhochi)
	+
	o(1)\cdot|\PhiQGood\rho\ssm\PhiQGoodBig\rho|
\end{equation*}
where
$$
	\PhiQGoodBig\rho
	=
	\PhiQGood\rho\cap\PhiQBig\rho.
$$

\smallskip\noindent
In particular, if the Mellin transform of $\rho$ has big monodromy, that is, if \eqref{eqn:big-monodromy-bis} holds, then
$$
	\frac{|\PhiQGood\rho\ssm\PhiQGoodBig\rho|}{|\PhiQGood\rho|} = o(1)
	\mbox{ for }
	q\to\infty
$$
and thus
\begin{eqnarray*}
	\frac{1}{|\PhiQGood\rho|}
	\sum_{\dc\in\PhiQGood\rho}
	\!\!\!\!\Tr\,\Lambda(\trhochi)
	& \overset{\eqref{eqn:uniform-bound-for-Tr}}= &
	\frac{1}{|\PhiQGood\rho|}
	\sum_{\dc\in\PhiQGoodBig\rho}
	\!\!\!\!\Tr\,\Lambda(\trhochi)
	+
	o(1)\cdot O(\dim(\Lambda)) \\
	& \overset{\eqref{eqn:sum-to-integral}}= &
    \int_{U_\R(\bbC)}\Tr\,\Lambda(\theta)\,d\theta + o(1)
\end{eqnarray*}
as $q\to\infty$.  Therefore $\ThetaRhoq$ becomes equidistributed in $U_\R(\bbC)$ as claimed.

\begin{remark}
An examination of the above proof will show that one does not need to suppose $q\to\infty$ by taking $q=p^m$ and letting $m\to\infty$.  Indeed, the key identities \eqref{eqn:a-uniform-bound} and \eqref{eqn:uniform-bound-for-Tr} are valid even if one takes $q=p$ and $p\to\infty$ in $\bbZ$.  This would allow one to prove `horizontal' variants of Theorem~\ref{thm:big-monodromy-implies-equidistribution}.  Because stating a correspondingly general result would be cumbersome and we do not need such results, we leave the details to an interested reader.
\end{remark}


\section{Sums in Arithmetic Progressions}\label{sec:sums-in-arithmetic-progressions}

In addition to assuming that our representation
$$
	\rho\colon \GKS\to\GL(V)
$$
is punctually $\iota$-pure of weight $w$, we suppose that $\rho$ is geometrically simple yet not an element of $\PhiQ$ and that the Mellin transform of $\rho$ has big monodromy.  The first hypothesis ensures that $\rhochi$ has trivial geometric invariants for every $\dc\in\PhiQ$ while the second allows us to apply Theorem~\ref{sec:proof-of-thmD}.

In this section, which forms the heart of our paper, we shift gears and analyze the distribution of certain traces indexed by residue classes modulo $\Q$.  More precisely, for each monic irreducible $\pi\in\MM$, the traces are coefficients of the Euler factor $\phivl{T}$ of $v=v(\pi)$, and we use them to define a function $\VM\colon\MM\to\El$ satisfying
\begin{equation*}
	T\frac{d}{dT}
	\log(L_{\{\infty\}}(T,\rho))
	=
	\sum_{n=1}^\infty
	\left(
	\sum_{f\in\MM_n}\VM(f)
	\right)
	T^n
\end{equation*}
(see \S\ref{subsec:von-mangoldt}).  In particular, for each $n\geq 1$ and $A\in\BQ$, we consider the sum
\begin{equation}\label{eqn:SnAQ}
	\SnAQ
	=
	\sum_{\substack{f\in \MM_n\\ f\equiv A\bmod\Q}}
	\VM(f),
\end{equation}
and then we consider the mean and variance of these sums given by
\begin{equation}\label{eqn:E-and-V}
	\bbE_A[\SnAQ]
	=
	\frac{1}{\phi(\Q)}\sum_{A\in\BQ}\SnAQ,
	\ 
	\Var_A[\SnAQ]
	=
	\frac{1}{\phi(\Q)}\sum_{A\in\BQ} \left|\SnAQ - \bbE_A[\SnAQ]\right|^2
\end{equation}
respectively.

Our main result has two parts.  On one hand, we can precisely evaluate $\bbE_A[\SnAQ]$ in terms of the coefficients $\bn{\rho}$ coming from the identity
$$
	T\frac{d}{dT}\LC(T,\rho)
	=
	\sum_{n=1}^\infty
	\bn{\rho}T^n
$$
satisfied by the normalized $L$-function (see \S\ref{subsec:random-sums}).  We can also give bounds for the archimedean norm of these coefficients (see \S\ref{subsec:key-estimates}).  On the other hand, we can evaluate $\Var_A[\SnAQ]$ using trace formulae (see \S\ref{subsec:random-sums}), and its leading order term is the value of a matrix integral on $U_\R(\bbC)$ by our hypotheses on $\rho$ (see \S\ref{subsec:key-estimates} and \S\ref{sec:proof-of-thmE}).  The value of this integral exhibits a dichotomy depending on whether or not $n\leq\R=\rC(\rho)$, and in particular, the interval of small $n$ grows with $\r=\dim(V)$ since $\rC$ does.

After giving some preliminary results we calculate the mean and variance in Theorem~\ref{thm:variance-estimate} of \S\ref{sec:proof-of-thmE}.  In our proof we use a classification of the elements of $\PhiQ$ in terms of a trichotomy of good, mixed, and heavy characters (see \S\ref{sec:character-trichotomy}).  As we explain, this is a refinement of Katz's dichotomy of good and bad characters.


\subsection{Trace formula}

In this section we define local and cohomological traces of $\rho$ and recall how they are related by a trace formula.  For details, see \cite[Exp.~2, \S 3]{SGA4.5}.

On one hand, the \defi{local traces} of $\rho$ are given by
$$
	\arvm{\rho,v}
	=
	\Tr\left(\rholv(\Frob_v)^m\mid\Wl\right)
	\mbox{ for }v\in\PP\mbox{ and }m\geq 1,
$$
and they satisfy
\begin{equation*}\label{eqn:euler-factor}
	T\frac{d}{dT}\log
	L(T,\rho_v)^{-1}
	=
	\sum_{m=1}^\infty \arvm{\rho,v}T^m
	\mbox{ for }v\in\PP.
\end{equation*}
Combining this identity with \eqref{eq:twisted-euler} yields the more general identity
\begin{equation}\label{eqn:twisted-euler-factor}
	T\frac{d}{dT}\log
	L(T,(\rhochi)_v)^{-1}
	=
	\sum_{m=1}^\infty \dc(\Frob_v)^m\arvm{\rho,v}T^m
	\mbox{ for }v\in\PP\ssm\CC.
\end{equation}

On the other hand, the \defi{cohomological traces} of $\rho\otimes\dc$ are given by
$$
	\bn{\rhochi}
	=
	\sum_{i=1}^2
	(-1)^i\cdot
	\Tr\left(\Frob_q\mid H^i_c(\AonetBar[1/\Q],\FF\otimes\LL_\dc)\right)
	\mbox{ for }n\geq 1,
$$
and they satisfy
\begin{equation}\label{eqn:cohomological-trace}
	T\frac{d}{dT}
	\log\LC(T,\rhochi)
	=
	\sum_{n=1}^\infty
	\bn{\rhochi}T^n.
\end{equation}
Similarly, we define the \defi{normalized cohomological traces} of $\rhochi$ by
$$
	\bnstar{\rho,\dc}
	=
	\frac{1}{q^{n(1+w)/2}}
	\bn{\rhochi}
	=
	\frac{1}{(\sqrt{q})^{1+w}}
	\sum_{i=1}^2
	(-1)^i\cdot\Tr\left(\Frob_q\mid H^i_c(\AonetBar[1/\Q],\FF\otimes\LL_\dc)\right)
$$
so that \eqref{eqn:LCStar} and \eqref{eqn:cohomological-trace} imply
$$
	T\frac{d}{dT}\log
	\LCStar(T,\rhochi)
	=
	T\frac{d}{dT}\log
	\LC(T/(\sqrt{q})^{1+w},\rhochi)
	=
	\sum_{n=1}^\infty
	\bnstar{\rho,\dc}T^n.
$$

Combining \eqref{eqn:twisted-euler-factor} and \eqref{eqn:cohomological-trace} with \eqref{eq:L-fun:frac:partial} yields the identity
$$
	T\frac{d}{dT}
	\log\LC(T,\rhochi)
	=
	\sum_{n=1}^\infty
	\left(
	\sum_{md=n}
	\sum_{v\in\PP_d}
	d\cdot\arvm{\rho,v}
	\right)
	T^n
$$
and, in particular, we obtain the Grothendieck--Lefschetz trace formula
\begin{equation}\label{eqn:trace-formula}
	\sum_{md=n}
	\sum_{v\in\PP_d}
	d\cdot\arvm{\rho,v}
	=
	\bn{\rhochi}.
\end{equation}


\subsection{Von Mangoldt function}\label{subsec:von-mangoldt}

We define the \defi{von Mangoldt function} of $\rho$ to be the map $\VM\colon\MM\to\El$ given by
\begin{equation}\label{eqn:von-mangoldt}
	\VM(f)
	=
	\begin{cases}
	d\cdot \arvm{\rho,v(\pi)} & f=\pi^m\mbox{ and }\pi\in\AA_d \\
	0 & \mbox{otherwise}.
	\end{cases}
\end{equation}
We also define the \defi{extension by zero} of $\dc\in\PhiQ$ to be the map $\chiz\colon\MM\to\El$ given by
$$
	\chiz(f) =
	\begin{cases}
		\dc(f+\Q\,\Fq[t]) & \mbox{if }
		\gcd(f,\Q)=1 \\
		0 & \mbox{otherwise}.
	\end{cases}
$$
It is multiplicative and satisfies
\begin{equation*}\label{eq:chis.vs.chiz}
	\chiz(\pi)
	=
	\begin{cases}
		\dc(\Frob_{v(\pi)}) &  \mbox{if }\pi\nmid\Q \\
		0 & \mbox{otherwise}
	\end{cases}
	\mbox{ for }\pi\in\AA.
\end{equation*}
These functions allow us to rewrite \eqref{eqn:trace-formula} as
\begin{equation*}\label{eqn:trace-formula-bis}
	T\frac{d}{dT}
	\log(\LC(T,\rhochi))
	=
	\sum_{n=1}^\infty
	\left(
	\sum_{f\in\MM_n}\chiz(f)\VM(f)
	\right)
	T^n
\end{equation*}
and, in particular, to deduce the identity
\begin{equation}\label{lem:twisted-VM-sum}
	\sum_{f\in\MM_n}\chiz(f)\VM(f)
	=
	\bn{\rhochi}
	\mbox{ for }n\geq 1.
\end{equation}
We observe that in the special case $\dc=\one$ this simplifies to
\begin{equation}\label{eqn:formula-for-bn}
	\bn{\rho}
	=
	\sum_{\substack{f\in \MM_n\\ \gcd(f,\Q)=1}}
	\VM(f).
\end{equation}


\subsection{Random arithmetic-progression sums}\label{subsec:random-sums}

Regard $A$ is a uniformly random element of $\BQ$, and consider the expected value
\begin{equation}\label{eqn:expected-value}
	\bbE_A[\SnAQ]
	=
	\frac{1}{\phi(\Q)}\sum_{A\in\BQ}\SnAQ.
\end{equation}
Observe that, for each $A_1,A_2\in\BQ$, one has
\begin{equation*}\label{eq:orth:A}
	\frac{1}{\phi(\Q)}\sum_{\dc\in\PhiQ}\chiz(A_1)\chibarz(A_2)
	=
	\begin{cases}
	1 & \mbox{if }A_1=A_2 \\
	0 & \mbox{if }A_1\neq A_2,
	\end{cases}
\end{equation*}
and thus
\begin{equation*}\label{eqn:sum-of-bs}
	\SnAQ
	=
	\frac{1}{\phi(\Q)}
	\sum_{f\in\MM_n}
	\VM(f)
	\sum_{\dc\in\PhiQ}
	\chiz(f)
	\chibarz(A)
	=
	\frac{1}{\phi(\Q)}\sum_{\dc\in\PhiQ} \bn{\rhochi}\cdot\chibarz(A)
\end{equation*}
by \eqref{lem:twisted-VM-sum}.  Therefore, if we write $\chinot\in\PhiQ$ for the trivial character, then the right side of \eqref{eqn:expected-value} equals
$$
	\frac{1}{\phi(\Q)^2}
	\sum_{\dc\in\PhiQ}
	\bn{\rhochi}
	\sum_{A\in\BQ}
	\bar\chiz(A)
	=
	\frac{1}{\phi(\Q)}
	\bn{\rho,\chinot}
$$
since, for every $\chione,\chitwo\in\PhiQ$, one has
\begin{equation}\label{eq:orth:chi}
	\frac{1}{\phi(\Q)}\sum_{A\in\BQ}\chionez(A)\chitwobarz(A)
	=
	\begin{cases}
	1 & \mbox{if }\chione=\chitwo \\
	0 & \mbox{if }\chione\neq\chitwo.
	\end{cases}
\end{equation}
In particular, we have the identity
\begin{equation}\label{eq:s-e}
	\SnAQ - \bbE_A[\SnAQ]
	=
	\frac{1}{\phi(\Q)}\sum_{\substack{\dc\in\PhiQ \\ \dc\neq\chinot}} \bn{\rhochi}\cdot\bar\dc(A).
\end{equation}

Now consider the variance
$$
	\Var_A[\SnAQ]
	=
	\frac{1}{\phi(\Q)}\sum_{A\in\BQ} \left|\SnAQ - \bbE_A[\SnAQ]\right|^2.
$$
If we apply identities \eqref{eq:orth:chi} and \eqref{eq:s-e}, then the right side equals
$$
	\frac{1}{\phi(\Q)^3}
	\sum_{A\in\BQ}
	\sum_{\substack{\chionez,\chitwoz\in\PhiQ\\\chionez,\chitwoz\neq\chinot}}
	\bn{\rho\otimes\chione}\overline{\bn{\rho\otimes\chitwo}}\cdot
	\chionebarz(A)\chitwoz(A)
	=
	\frac{1}{\phi(\Q)^2}
	\sum_{\substack{\dc\in\PhiQ\\\dc\neq\chinot}}
	|\bn{\rhochi}|^2.
$$

\medskip
In summary, the function $\SnAQ$ of the random variable $A$ satisfies
\begin{equation}\label{eqn:E-and-V-formulae}
	\bbE_A[\SnAQ]
	= \frac{1}{\phi(\Q)}\bn{\rho\otimes\chinot},
	\quad
	\Var_A[\SnAQ] = 
	\frac{1}{\phi(\Q)^2}
	\sum_{\substack{\dc\in\PhiQ\\\dc\neq\chinot}}
	|\bn{\rhochi}|^2.
\end{equation}
Observe that $\rho\otimes\chinot=\rho$ and thus $\bnstar{\rho\otimes\chinot}=\bnstar{\rho}$.


\subsection{Trichotomy of characters}\label{sec:character-trichotomy}

On one hand, a character $\dc\in\PhiQ$ is \defi{good for $\rho$} (or \defi{$\rho$-good}) if and only if the $L$-functions $\LC(T,\rhochi)$ and $L(T,\rhochi)$ are polynomials and equal in $\Qbar[T]$; see \eqref{eq:phi-good}.  In that case Theorem~\ref{thmB} implies they equal $\PC{1}(T,\rhochi)$ and are $\iota$-pure of $q$-weight $w+1$, and then $\LCStar(T,\rhochi)$ is given by
$$
	\LCStar(T,\rhochi)
	=
	\det(1-T\,\Frob_q\mid H^1_c(\AonetBar[1/\Q], \FF))
$$
where
$$
	\FF
	:=
	\ME{\rhochi}((1+w)/2)
	=
	\ME{\rhochi}\otimes\El((1+w)/2)
$$
is a so-called Tate twist of $\ME{\rhochi}$.  Moreover, $\LCStar(T,\rhochi)$ has degree $\R=\rC(\rhochi)=\rC(\rho)$ and is $\iota$-pure of $q$-weight zero.  In particular, it is the characteristic polynomial of a unique conjugacy class $\trhochi\sub U_\R(\bbC)$, and thus
\begin{equation}\label{eqn:bnstar-to-trace}
	\bnstar{\rhochi}
	=
	-\Tr\left(\Frob_q^n\mid H^1_c(\AonetBar[1/\Q],\FF)\right)
	=
	-\Tr\,\std(\trhochi^n)
\end{equation}
where $\std\colon U_\R(\bbC)\to\GL_\R(\bbC)$ is the inclusion $U_\R(\bbC)\seq\GL_R(\bbC)$.

On the other hand, there are two ways a character can fail to be good for $\rho$: either $L(T,\rhochi)$ is not a polynomial or $L(T,\rhochi)$ and $\LC(T,\rhochi)$ are polynomials but not equal to each other.  Only the first of these possibilities is problematic for us because in that case the denominator of $L(T,\rhochi)$ has zeros of excessive weight.  More precisely, if the factor $P_2(T,\rhochi)$ of the denominator of $L(T,\rhochi)$ is non-trivial, then it $\iota$-mixed of $q$-weights $\leq w+1$ but not $\iota$-mixed of $q$-weights $\leq w$ (cf.~Theorem~\ref{thmB}).  Hence we say that $\dc$ is \defi{heavy for $\rho$} (or \defi{$\rho$-heavy}) iff it lies in the subset
$$
	\PhiQAwful\rho
	=
	\{\,
	\dc\in\PhiQ
	:
	L(T,\rhochi)\not\in\Qbar[T]
	\,\}.
$$

\medskip\noindent
The following lemma can be used to classify $\dc$ which are heavy for $\rho$.

\begin{lemma}\label{lem:heavy-criterion}
Suppose $\rho$ is geometrically simple and punctually $\iota$-pure and $\dc\in\PhiQ$.  Then $\dc\in\PhiQAwful\rho$ if and only if $\rhochi$ is geometrically isomorphic to the trivial representation.
\end{lemma}

\begin{proof}
The essential point is that since $\rhochi$ is geometrically simple, the quotient space of geometric coinvariants $(V_\dc)_{\GKSbar}$ either vanishes or equals $V_\dc$.  The former occurs if and only if $\rhochi$ is geometrically isomorphic to the trivial representation, so the lemma follows from Corollary~\ref{cor:poly-L-function}.
\end{proof}

\begin{cor}\label{cor:classify-awful}
Suppose $\rho$ is geometrically simple and punctually $\iota$-pure, and let $r=\dim(V)$.  Then $\PhiQAwful\rho\seq\{\chinot\}$ if and only if one of the following hold:

\begin{enum}
\item\label{cor:item:r>1} $r>1$;
\item\label{cor:item:r=1 and trivial} $r=1$ and $\rho$ is geometrically isomorphic to the trivial representation;
\item\label{cor:item:r=1 and non-Kummer} $r=1$ and $\rho$ is not geometrically isomorphic to a Dirichlet character in $\PhiQ$.
\end{enum}

\smallskip\noindent
Moreover, $\PhiQAwful\rho=\{\chinot\}$ if and only if \eqref{cor:item:r=1 and trivial} holds.
\end{cor}

\begin{proof}
Let $\dc\in\PhiQ$.  Lemma~\ref{lem:heavy-criterion} implies that $\dc$ is heavy for $\rho$ if and only if $\rhochi$ is geometrically isomorphic to the trivial representation (and hence $r=1$).  By the contrapositive, $\dc$ is not heavy for $\rho$ if and only if $r>1$ or $\rho$ is not geometrically isomorphic to $1/\dc$.  Therefore \eqref{cor:item:r>1} or \eqref{cor:item:r=1 and non-Kummer} holds if and only if $\PhiQAwful\rho$ is empty, and \eqref{cor:item:r=1 and trivial} holds if and only if $\PhiQAwful\rho=\{\chinot\}$.
\end{proof}

We also say that $\dc$ is \defi{mixed for $\rho$} (or \defi{$\rho$-mixed}) iff it lies in the subset
$$
	\PhiQMixed\rho = \PhiQ \ssm (\PhiQGood\rho \cup \PhiQAwful\rho).
$$
Equivalently, $\dc$ is mixed for $\rho$ if and only if $\LC(T,\rhochi)$ is a polynomial which is $\iota$-mixed of $q$-weights $\leq w+1$ but not $\iota$-pure of $q$-weight $w+1$.

In summary, we classify the characters in $\PhiQ$ by a trichotomy: each is either $\rho$-good, $\rho$-mixed, or $\rho$-heavy.  This terminology refines Katz's because we divide his bad characters into mixed and heavy characters.

\begin{lemma}\label{lem:bn-bound}
Suppose $\rho$ is punctually $\iota$-pure of weight $w$ and $\dc\in\PhiQ$.  Then

\smallskip
\begin{enum}
\item If $\dc$ is heavy for $\rho$, then $|\bnstar{\rhochi}|^2=O(q^n)$, and otherwise $|\bnstar{\rhochi}|^2=O(1)$.
\item $|\PhiQMixed\rho\ssm\{\chinot\}|\sim O(|\PhiQGood\rho|/q)$ and $|\PhiQAwful\rho|=O(1)$.
\end{enum}

\smallskip\noindent
Moreover, the bounds assume $q$ tends to infinity and the implied constants depend only on $\rho$.
\end{lemma}

\begin{proof}
Regardless of whether $\dc$ is good, mixed, or heavy, we have
$$
	\bnstar{\rhochi}
	=
	-\Tr\left(\Frob_q^n\mid H^1_c(\AonetBar[1/\Q],\FF)\right)
	+\Tr\left(\Frob_q^n\mid H^2_c(\AonetBar[1/\Q],\FF)\right).
$$
One one hand, the second term on the right vanishes unless $\dc$ is heavy.  On the other hand, Theorem~\ref{thm:deligne} and Lemma~\ref{lem:trace-bound} imply
$$
	|\,\Tr\left(\Frob_q^n\mid H^i_c(\AonetBar[1/\Q],\FF)\right)|^2
	=
	O(q^{i-1})
$$
since $\FF$ is punctually pure of weight $-1$.
\end{proof}

Up to replacing $\Q$ by a proper monic divisor $\Qp$, we can apply the same trichotomy to characters in $\PhiQp$.

\begin{lemma}\label{lem:Qp-bn-bound}
Let $\Qp$ be a monic divisor of $\Q$ in $\Fq[t]$.  If $\rho$ is punctually $\iota$-pure of weight $w$ and if $\dc\in\PhiQ$, then $|\PhiQpGood\rho|\sim|\PhiQp|$ as $q\to\infty$.
\end{lemma}

\begin{proof}
Apply Lemma~\ref{lem:bn-bound} with $\Qp$ in lieu of $\Q$.
\end{proof}


\subsection{Key estimates}\label{subsec:key-estimates}

In this section we provide the exact formula and key asymptotic estimate we need to prove Theorem~\ref{thm:variance-estimate}.

\begin{prop}\label{prop:E-estimate}
Suppose $\rho$ is punctually $\iota$-pure of weight $w$ and $\PhiQAwful\rho\seq\{\chinot\}$.  Then
$$
	\phi(\Q)\cdot\bbE_A[\SnAQ] = \bn{\rho}.
$$
\end{prop}

\begin{proof}
By definition,
$$
	\phi(\Q)\cdot\bbE_A[\SnAQ]
	=
	\sum_{A\in\BQ} \SnAQ
	=
	\sum_{A\in\BQ}
	\sum_{\substack{f\in \MM_n\\ f\equiv A\bmod\Q}}
	\VM(f)
	=
	\sum_{\substack{f\in \MM_n\\ \gcd(f,\Q)=1}}
	\VM(f),
$$
and \eqref{eqn:formula-for-bn} then yields the desired identity.
\end{proof}

\begin{remark}
While we do not need the result, we point out that Proposition~\ref{prop:E-estimate} and Lemma~\ref{lem:bn-bound} imply
$$
	\frac{\phi(\Q)}{q^{n(1+w)}}\cdot|\bbE_A[\SnAQ]|^2
	=
	|\bnstar{\rho}|^2
	\sim
	O(1)
	\mbox{ for }
	q\to\infty
$$
when $\rho$ is punctually $\iota$-pure of weight $w$ and $\PhiQAwful\rho\seq\{\chinot\}$.
\end{remark}

\begin{proof}
Combine .
\end{proof}

\begin{prop}\label{prop:var-estimate}
Suppose $\rho$ is punctually $\iota$-pure of weight $w$ and $\PhiQAwful\rho\seq\{\chinot\}$.  Then
$$
	\frac{\phi(\Q)}{q^{n(1+w)}}
	\cdot
	\Var_A[\SnAQ]
	=
	\frac{1}{|\PhiQGood\rho|}
	\sum_{\dc\in\PhiQGood\rho}
	|\Tr\,\std(\trhochi^n)|^2
	+
	O(q^{-1})
	\mbox{ as }
	q\to\infty
$$
where $\std\colon U_\R(\bbC)\to\GL_\R(\bbC)$ is the representation given by the inclusion $U_\R(C)\seq\GL_\R(\bbC)$.
\end{prop}

\begin{proof}
Lemma~\ref{lem:bn-bound} implies
\begin{align*}
	\phi(\Q)^2\cdot\Var_A[\SnAQ]\ 
	& -
	\sum_{\substack{\dc\in\PhiQGood\rho\\\dc\neq\chinot}}
	|\bn{\rhochi}|^2
	\ =
	\sum_{\substack{\dc\in\PhiQMixed\rho\\\dc\neq\chinot}}
	|\bn{\rhochi}|^2
	\ +\ 
	\sum_{\substack{\dc\in\PhiQAwful\rho\\\dc\neq\chinot}}
	|\bn{\rhochi}|^2 \\[0.1in]
	& \sim\ 
	|\PhiQMixed\rho\ssm\{\chinot\}|\cdot O(q^{n(1+w)})
	\ +\ 
	|\PhiQAwful\rho\ssm\{\chinot\}|\cdot O(q^{n(2+w)}),
\end{align*}
and thus Lemma~\ref{lem:bn-bound} implies
$$
	\Var_A[\SnAQ]
	\sim
	\frac{q^{n(1+w)}}{\phi(\Q)}
	\left(
	\frac{1}{|\PhiQGood\rho|}
	\sum_{\dc\in\PhiQGood\rho}
	|\bnstar{\rhochi}|^2
	+
	O(q^{-1})
	\right)
$$
as $q\to\infty$.  The proposition now follows from \eqref{eqn:bnstar-to-trace}.
\end{proof}


\subsection{Proof of Theorem~\ref{thm:variance-estimate}}\label{sec:proof-of-thmE}

%
%

\newcommand\thmE{%
Suppose that $\rho$ is punctually $\iota$-pure of weight $w$, that $\PhiQAwful\rho\seq\{\chinot\}$ for all $q$, and that the Mellin transform of $\rho$ has big monodromy.  Then, for each $n\geq 1$ and,
$$
	\phi(\Q)\cdot\bbE_A[\SnAQ]
	=
	\bn{\rho}
	\mbox{ and }
	\lim_{q\to\infty}
	\frac{\phi(\Q)}{q^{n(1+w)}}
	\cdot
	\Var_A[\SnAQ]
	=
	\min\{n,\rC(\rho)\}.
$$
}

The following theorem is the main result of this section.

\begin{theorem}\label{thm:variance-estimate}
\thmE
\end{theorem}

\noindent
See Corollary~\ref{cor:classify-awful} for a classification of $\rho$ satisfying the condition $\PhiQAwful\rho\seq\{\chinot\}$.

\begin{proof}
The first part of the theorem is an immediate consequence of \eqref{eqn:E-and-V-formulae} since $\PhiQAwful\rho\seq\{\chinot\}$ for all $q$.  Let $\R=\rC(\rho)$.  Then Theorem~\ref{thm:big-monodromy-implies-equidistribution} implies that $\ThetaRhoq$ is equidistributed in $U_\R(\bbC)$ as $q\to\infty$ since the Mellin transform of $\rho$ has big monodromy.  Therefore Proposition~\ref{prop:E-estimate} implies that
$$
	\phi(\Q)\cdot\bbE_A[\SnAQ]
	=
	\bn{\rho},
$$
and Proposition~\ref{prop:var-estimate} and \eqref{eqn:trace-limit} imply
$$
	\frac{\phi(\Q)}{q^{n(1+w)}}
	\cdot
	\Var_A[\SnAQ]
	\ \sim
	\int_{U_\R(\bbC)}
	|\Tr\,\std(\theta^n)|^2
	\,d\theta.
$$
The second part of the theorem now follows from the identity
$$
	\int_{U_\R(\bbC)}
	|\Tr\,\std(\theta^n)|^2
	\,d\theta
	=
	\min\{n,\R\}
	=
	\min\{n,\rC(\rho)\}
$$
(see\footnote{NB: The reference \cite[Th.~2]{DS} is sometimes used, but as explained in \cite{DE}, the theorem is incorrectly stated.} \cite[Th.~1]{DE}).
\end{proof}


\section{Exhibiting Big Monodromy}\label{sec:big-monodromy}

In this section we present sufficient criteria for the Mellin transform of $\rho$ to have big monodromy and refer the interested reader to \S\ref{sec:explicit-abelian-varieties} for explicit examples of representations meeting these criteria.  Before stating the main theorem, we make some hypotheses and introduce pertinent terminology.

Throughout this section, we suppose that $\gcd(\s,\Q)=t-a$, for some $a\in\Fq$.  One could easily argue that this is less general than supposing that $\s,\Q$ are relatively prime, however, we do not presently have a way to avoid our hypothesis.  For ease of exposition, we also suppose that $a=0$ and observe that, up to performing an additive translation $t\mapsto t+a$, this represents no additional loss of generality.

For $t=0,\infty$, we regard $V_\dc$ as an $I(t)$-module and then denote it $V_\dc(t)$.  We write $V_\dc(t)^\unip$ for the maximal subspace of $V_\dc(t)$ on which $I(t)$ acts unipotently.  It is a direct summand of $V_\dc(t)$, and each simple $e$-dimensional submodule of it is isomorphic to a common module $\Unip(e)$.  We say $V_\dc(t)$ has a \defi{unique unipotent block exact multiplicity one} iff, for a unique integer $e\geq 1$, some $I(t)$-submodule is isomorphic $\Unip(e)$ but no submodule is isomorphic to $\Unip(e)\oplus\Unip(e)$.

\newcommand\thmF{
Suppose that $\gcd(\s,\Q)=t$ and that $\deg(\Q)\geq 3$.  Suppose moreover that $V(0)$ has a unique unipotent block of exact multiplicity one and that $\rho$ is geometrically simple and punctually pure.  If $r:=\dim(V)$ and $\deg(\Q)$ satisfy
$$
	\deg(\Q)
	>
	\frac{1}{r}\left(72(r^2+1)^2 - r - \deg(L(T,\rho)) + \dropCee{\rho}\right),
$$
then the Mellin transform of $\rho$ has big monodromy.
}

\begin{theorem}\label{thm:is-equidistributed}
\thmF
\end{theorem}

\noindent
We prove the theorem in \S\ref{subsec:proof-of-equidistribution-theorem}.

\begin{remark}\label{rmk:unipotence-hypothesis}
As the reader will notice, the proof of our theorem has a lot in common with Katz's proof of \cite[Th.~17.1]{Katz:CE}.  We both need the hypothesis on $\gcd(\Q,\s)$ and the structure of $V(0)^\unip$ in order to exhibit special elements of the relevant arithemtic monodromy groups.  More precisely, the hypothesis that $\gcd(\Q,\s)=t$ helps ensure that, for sufficiently many $\dc$, some induced representation $\Ind(V_\dc)$ has the property that $\Ind(V_\dc)(0)^\unip=V(0)^\unip$ (cf.~Lemma~\ref{lem:induced-unipotent}).  The hypothesis on the structure of these coincident modules then leads to the desired element (cf.~Lemma~\ref{lem:nice-element}).  We expect one can remove this hypothesis but do not know how to do so.
\end{remark}

\begin{remark}
The hypothesis $\gcd(\Q,\s)=t$ also plays a minor role in Proposition~\ref{prop:induced-simplicity}.  However, one could easily make other hypotheses (e.g.,~$\gcd(\Q,\s)=1$) and still be able to proceed (cf.~\cite[Th.~5.1]{Katz:QKR}).
\end{remark}


\subsection{Two norm maps}

This subsection recalls material from \cite[\S 2]{Katz:CE} and borrows heavily from \loccit

Let $B$ be the finite $\Fq$-algebra $\Fq[t]/\Q\,\Fq[t]$.  It is a direct product of finite extensions of $\Fq$ and hence \etale{} since $\Q$ is square free.  More generally, for each finite extension $\EFq/\Fq$, the $\Fq$-algebra
$$
	B_{\EFq} = B\otimes_{\Fq}\EFq
$$
is \etale{} and has the structure of a free $B$-module of rank $d=[\EFq:\Fq]$.

Let $\bbB$ be the functor on variable $\Fq$-algebras $R$ defined by
$$
	\bbB(R)
	= R[t]/\Q R[t].
$$
It is the functor $R\mapsto B_R=B\otimes_{\Fq}R$ and takes values in the category of $\Fq$-algebras.  In fact, $\bbB(R)$ even has the structure of an \etale{} $R$-algebra which is free of rank $\deg(\Q)$.  In particular, for each $\Fq$-algebra $R$, there is a norm map $\bbB(R)\to R$ which is part of a transformation
$$
	\norm_{B/\Fq}\colon\bbB\to\id_{\Fq\mathrm{-algebras}}
$$
between $\bbB$ and the identity functor on the category of $\Fq$-algebras.

Let $\bbBt$ be the functor on variable $\Fq$-algebras $R$ defined by
$$
	\bbBt(R)
	= (R[t]/\Q R[t])^\times.
$$
It is the composition of $\bbB$ with the functor $A\mapsto A^\times$ of $\Fq$-algebras and takes values in the category of groups.  Moreover, the restriction of the norm map $\bbB(R)\to R$ to the group of units yields a homomorphism
$$
	\nu_R\colon\bbBt(R)\to R^\times,
$$
and in particular, $\nu_{\Fq}$ is the map $\nu$ of \S\ref{sec:one-parameter-families}.

For each finite extension $\EFq/\Fq$, let $\bbB_{\EFq}$, $\bbBt_{\EFq}$ be the functors on variable $\Fq$-algebras $R$ defined by
$$
	\bbB_{\EFq}(R)
	= B_{\EFq}\otimes_{\Fq}R,\quad
	\bbBt_{\EFq}(R)
	= (B_{\EFq}\otimes_{\Fq}R)^\times
$$
respectively.

On one hand, $\bbB_{\EFq}$ takes values in the category of $\Fq$-algebras.  However, $\bbB_{\EFq}(R)$ also has the structure of an \etale{} $B_R$-algebra which is free of rank $d$ as a $B_R$-module since
$$
	B_{\EFq}\otimes_{\Fq}R
	= B\otimes_{\Fq}\EFq\otimes_{\Fq}R
	= B_R\otimes_{\Fq}\EFq
$$
and since $B_{\EFq}$ is an \etale{} $B$-algebra which is free of rank $d$ as a $B$-module.  In particular, there is a transformation
$$
	\norm_{\EFq/\Fq}\colon
	\bbB_E\to\bbB
$$
between the functors $\bbB_E$ and $\bbB$.

On the other hand, $\bbBt_{\EFq}$ takes values in the category of groups and is even a smooth commutative group scheme.  More precisely, $\bbBt$ is a group scheme over $\Fq$ of multiplicative type (i.e., a torus), and $\bbBt_{\EFq}$ is the torus $\Res_{\EFq/\Fq}(\bbBt)$ over $\Fq$ given by extending scalars to $\EFq$ and then taking the Weil restriction of scalars of $\bbBt$ back down to $\Fq$ (cf.~\cite[\S 7.6]{BLR}).  Moreover, the transformation $\norm_{\EFq/\Fq}$ induces a transformation
$$
	\norm_{\EFq/\Fq}\colon\bbBt_E\to\bbBt
$$
which is even an \etale{} surjective homomorphism of tori.  
In particular, since
$$
	\bbBt_{\EFq}(\Fq)
	=
	\bbBt(\EFq)
	=
	(\EFq[t]/\Q\EFq[t])^\times
$$
one obtains a second norm map
$$
	\nu_{\EFq}^{\,\prime}
	\colon
	(\EFq[t]/\Q\EFq[t])^\times
	\to
	(\Fq[t]/\Q\Fq[t])^\times
$$
which is a surjective homomorphism by Lang's theorem.


\subsection{Characters of a twisted torus}

Let $\EFq/\Fq$ be a finite extension and $\PhiEOf\EFq\Q$ be the dual group $\Hom(\bbBt(\EFq),\bbC^\times)$ so that $\PhiEOf\Fq\Q=\PhiQ$.  Suppose that $\Q$ splits completely over $\EFq$, and let $a_1,\ldots,a_n\in\EFq$ be the zeros of $\Q$ so that $\Q=\prod_{i=1}^n(t-a_i)$ in $\EFq[t]$.

For each $\EFq$-algebra $R$, the Chinese Remainder Theorem implies that there is a unique algebra isomorphism
\begin{equation}\label{eqn:crt-isomorphism}
	R[t]/\Q R[t]\to\prod_{i=1}^n R[t]/(t-a_i)R[t]
\end{equation}
which sends the residue class of $t$ to the tuple $(a_1,\ldots,a_n)$ of residue class representatives.  Writing it as an isomorphism $\bbB(R)\to R^n$ and restricting to units yields a group isomorphism $\bbBt(R)\to (R^\times)^n$.  As $R$ varies over $\EFq$-algebras, the latter isomorphisms in turn yield an isomorphism of tori $\sigma\colon\bbBt\to\Gm^n$ over $\EFq$.  In particular, applying Weil restriction of scalars from $\EFq$ to $\Fq$ yields an isomorphism
$$
	\Res_{\EFq/\Fq}(\sigma)\colon\bbBt_{\EFq}\to\bbG_{m,\EFq}^n
$$
of tori over $\Fq$ where $\bbG_{m,\EFq}=\Res_{\EFq/\Fq}(\Gm)$.

There is a unique permutation $\phi\in\Sym([n])$ satisfying $a_{\phi^{-1}(i)}=a_i^q$ since $\Q$ is square free and has coefficients in $\Fq$.  While $\sigma$ does not descend to a morphism $\bbBt\to\Gm^n$ in general, we can use $\phi$ to construct a twisted form $\bbT$ of $\Gm^n$ over $\Fq$ such that $\sigma$ is the pullback of a morphism $\bbBt\to\bbT$ over $\Fq$.  More precisely, we define the twisted Frobenius $\tau$ on $\bbT=\Gm^n$ as the composition
$$
	(b_1,\ldots,b_n)
	\mapsto (b_1^q,\ldots,b_n^q)
	\mapsto (b_{\phi(1)}^q,\ldots,b_{\phi(n)}^q)
$$
of the usual Frobenius automorphism and a permutation of the coordinates of $\Gm^n$.  One can easily verify that $\tau^d$ is the $d$th power of the usual Frobenius and thus $\bbT$ is indeed a twist of $\Gm^n$.  Moreover, one can also show that $(a_1,\ldots,a_n)$ is fixed by $\tau$ and even that
$$
	\bbT(\Fq)=\bbT^{\tau=1}=\bbBt(\Fq).
$$
In particular, by precomposing with $\tau$ we obtain the automorphism $\tau_\EFq^\vee$ on
$$
	\Hom(\bbT(\EFq),\bbC^\times)
	=
	\Hom(\Gm^n(\EFq),\bbC^\times)
	=
	\Hom(\EFq^\times,\bbC^\times)^n
$$
given by
\begin{equation}\label{eqn:def-tau}
	\tau_\EFq^\vee\colon
	(\dc_1,\ldots,\dc_n)
	\mapsto
	(\dc_{\phi^{-1}(1)}^q,\ldots,\dc_{\phi^{-1}(n)}^q).
\end{equation}

Composition of $\Res_{\EFq/\Fq}(\sigma)$ with the projection $\bbG_{m,\EFq}^n\to\bbG_{m,\EFq}$ onto the $i$th factor yields a surjective homomorphism
$$
	\pi_i\colon\bbBt_{\EFq}\to\bbG_{m,\EFq}
$$
of tori over $\Fq$.  In particular, taking duals of the respective groups of $\EFq$-rational points and using the bijections $\bbG_{m,\EFq}(\Fq)=\Gm(\EFq)=\EFq^\times$ yields an isomorphism
$$
	\sigma_{\EFq}^\vee
	\colon
	\prod_{i=1}^n\Hom(\EFq^\times,\bbC^\times)
	\ni(\dc_1,\ldots,\dc_n)\mapsto\prod_{i=1}^n\dc_i\pi_i\in
	\PhiEOf\EFq\Q.
$$
We observe that since $\nu_\EFq^\prime$ is surjective its dual $\nu_{\EFq}^{\,\prime\,\vee}$ is a monomorphism $\PhiQ\to\PhiEOf\EFq\Q$ and thus we can identify $\PhiQ$ with a subset of $\Hom(\EFq^\times,\bbC^\times)^n$.  More precisely, it is the subgroup of characters fixed by $\tau_\EFq^\vee$ and thus
\begin{equation}\label{eqn:identifying-PhiQ}
	(\sigma_\EFq^\vee)^{-1}(
	\nu_{\EFq}^{\,\prime\,\vee}(\PhiQ)
	)
	=
	\{\,
		(\dc_1,\ldots,\dc_n)\in\Hom(\EFq^\times,\bbC^\times)^n
		:
		\dc_{\phi(i)}=\dc_i^q\mbox{ for }i\in[n]
	\,\}.
\end{equation}


\subsection{Characters with distinct components}\label{sec:distinct-component-characters}

We say that a character $\dc\in\PhiEOf\EFq\Q$ \defi{has distinct components} iff it lies in the subset
$$
	\PhiEOfDistinct\EFq\Q
	=
	\left\{\,
		\sigma_{\EFq}^\vee(\dc_1,\ldots,\dc_n)\in\PhiEOf\EFq\Q
		:
		\dc_i\neq\dc_j
		\mbox{ for }
		1\leq i<j\leq n
	\,\right\},
$$
and we define the corresponding subset of $\PhiQ$ as the intersection
$$
	\PhiQDistinct
	=
	\PhiEOfDistinct\EFq\Q\cap\nu_{\EFq}^{\,\prime\,\vee}(\PhiQ)
$$
where $\nu_{\EFq}^{\,\prime\,\vee}\colon\PhiQ\to\PhiEOf\EFq\Q$ is the dual of $\nu_\EFq^{\,\prime}$.

\begin{lemma}
$\PhiQDistinct$ is well defined, that is, it does not depend upon our choice of $\EFq$.
\end{lemma}

\begin{proof}
Let $\EFq'/\EFq$ be a finite extension and observe that the norm map $\EFq^{\prime\times}\to\EFq^\times$ is surjective so induces a monomorphism
$$
	\Hom(\EFq^\times,\bbC^\times)
	\to
	\Hom(\EFq^{\prime\times},\bbC^\times),
$$
and thus
$$
	\PhiEOfDistinct{\EFq}\Q
	=
	\PhiEOfDistinct{\EFq'}\Q
	\cap
	\PhiEOf\EFq\Q.
$$
In particular, if $\EFq''/\Fq$ is a second finite extension over which $\Q$ splits completely and if $\EFq'$ contains the compositum $\EFq\EFq''$, then
$$
	\PhiEOfDistinct{\EFq}\Q
	\cap
	\nu_{\EFq}^{\,\prime\,\vee}(\PhiQ)
	=
	\PhiEOfDistinct{\EFq'}\Q
	\cap
	\nu_{\EFq'}^{\,\prime\,\vee}(\PhiQ)
	=
	\PhiEOfDistinct{\EFq''}\Q
	\cap
	\nu_{\EFq''}^{\,\prime\,\vee}(\PhiQ)
$$
and $\PhiQDistinct$ is indeed well defined.
\end{proof}

Let $\Q=\prod_{j=1}^r\pi_i\in\Fq[t]$ be a factorization into monic irreducibles.  The quotient $\EFq_j=\Fq[t]/\pi_j\Fq[t]$ is a finite extension of $\Fq$ of degree and $n_j=\deg(\pi_j)$.  It is also the splitting field of $\pi_j$ and thus may be embedded in $\EFq$.  Moreover, there are bijections
\begin{equation}\label{eqn:phiQ-bijections}
	\PhiQ
	= \prod_{j=1}^r\PhiOf{\pi_j}
	= \prod_{j=1}^r\Hom(\EFq_j^\times,\bbC^\times),\ \ 
	\PhiEOf\EFq\Q
	= \prod_{j=1}^r\PhiEOf\EFq{\pi_j}
	= \prod_{j=1}^r\Hom(\EFq^\times,\bbC^\times)^{n_j}
\end{equation}
given by applying the Chinese Remainder Theorem.

For each monic factor $\Qp$ of $\Q$ in $\Fq[t]$, let $\PhiQpDistinct$ be the subset of $\PhiQp$ defined similarly as above but with $\Qp$ in lieu of $\Q$.  One can easily verify that it does not depend upon the polynomial $\Q$ of which $\Qp$ is a factor.

\begin{lemma}\label{lem:counting-distinct-characters}
$|\PhiOfDistinct{\pi_j}|\sim|\PhiOf{\pi_j}|$, for each $j\in[r]$, as $q\to\infty$.
\end{lemma}

\begin{proof}
Let $j\in[r]$, and suppose without loss of generality that $a_1,\ldots,a_{n_j}$ are the zeros of $\pi_j$ and $\phi(i)\equiv i+1\bmod{n_j}$ for $i\in[n_j]$.  Then by \eqref{eqn:identifying-PhiQ} and \eqref{eqn:phiQ-bijections} there is an identification
\begin{eqnarray*}
	\PhiOf{\pi_j}
	& = &
	\{\,
		(\dc_1,\ldots,\dc_{n_j})\in\Hom(\EFq_j^\times,\bbC^\times)^{n_j}
		:
		\dc_{i+1}=\dc_i^q\mbox{ for }i\in[n_j-1]
	\,\}.
\end{eqnarray*}
since any $\dc\in\Hom(\EFq^\times,\bbC^\times)$ factors through an inclusion $\EFq_j^\times\to\EFq^\times$ if $\dc^{q^{n_j}}=\dc$.

The groups $\EFq_j^\times$ and $\Hom(\EFq_j^\times,\bbC^\times)$ are cyclic and non-canonically isomorphic, so let $g$ and $\chi$ be respective generators.  Then we have a further identifications
\begin{eqnarray*}
	\PhiOf{\pi_j}
	& = &
	\{\,
		(\chi^{e_1},\ldots,\chi^{e_{n_j}})\in\Hom(\EFq_j^\times,\bbC^\times)^{n_j}
		:
		e_{i+1}\equiv qe_i\bmod{q^{n_j}-1}
		\mbox{ for }
		i\in[n_j-1]
	\,\} \\
	& = &
	\{\,
		(g^{e_1},\ldots,g^{e_{n_j}})\in(\EFq_j^\times)^{n_j}
		:
		e_{i+1}\equiv qe_i\bmod{q^{n_j}-1}
		\mbox{ for }
		i\in[n_j-1]
	\,\}.
\end{eqnarray*}
From this last identification one easily deduces an identification between $\PhiOfDistinct{\pi_j}$ and the set
$$
	\{\,
		(g^{e_1},\ldots,g^{e_{n_j}})\in(\EFq_j^\times)^{n_j}
		:
		e_{i+1}\equiv qe_i\bmod{q^{n_j}-1}
		\mbox{ for }
		i\in[n_j-1]
		\mbox{ and }
		\Fq(g^{e_1})=\EFq_j
	\,\},
$$
and thus
$$
	|\PhiOfDistinct{\pi_j}|
	=
	|\{\,
		g^e\in\EFq_j^\times
		:
		e\in[q^{n_j}-1]
		\mbox{ and }
		\EFq_j=\Fq(g^e)
	\,\}|.
$$
Finally, it is well known that the cardinality of the righthand set is asymptotic to $q^{n_j}-1$ as $q\to\infty$ (cf.~\cite[2.2]{Rosen}), and thus
$$
	|\PhiOf{\pi_j}|
	=
	|\Hom(\EFq_j^\times,\bbC^\times)|
	=
	|\EFq_j^\times|
	=
	q^{n_j}-1
	\sim
	|\PhiOfDistinct{\pi_j}|
	\mbox{ for }
	q\to\infty
$$
as claimed.
\end{proof}

\begin{cor}\label{cor:size-of-Qp-distinct}
If $\Qp$ is a monic factor of $\Q$ in $\Fq[t]$, then $|\PhiQpDistinct|\sim|\PhiQp|$ as $q\to\infty$.
\end{cor}

\begin{proof}
Suppose without loss of generality that $\Q=\pi_1\cdots\pi_s$ with $s\in[r]$ so that there is a bijection
$$
	\PhiOf\Qp
	=
	\prod_{j=1}^s
	\PhiOf{\pi_j}.
$$
This bijection in turn induces an inclusion
$$
	\PhiOfDistinct\Qp
	\to
	\prod_{j=1}^s
	\PhiOfDistinct{\pi_j}
$$
whose coimage is bounded above by
$
	\prod_{j=1}^s
	(\deg(\Qp)-n_j)
$
since an element of the codomain lies in the image if (and only if) the components are pairwise distinct.  In particular, 
$$
	|\PhiOfDistinct\Qp|
	\sim
	\prod_{j=1}^s
	|\PhiOfDistinct{\pi_j}|
	\overset{\mathrm{Lemma~\ref{lem:counting-distinct-characters}}}\sim
	\prod_{j=1}^s
	|\PhiOf{\pi_j}|
	\mbox{ for }
	q\to\infty
$$
as claimed.
\end{proof}


\subsection{Properties of $H^2_c$}

Let $X$ be a smooth geometrically connected curve over $\Fq$, let
 $T\seq X$ be a dense Zariski open subset, and let $\FF$ be a sheaf on $X$.

\begin{lemma}\label{lem:birational-invariance-of-H^2_c}
There is a bijection $H^2_c(\Tbar,\FF)\to H^2_c(\bar{X},\FF)$.
\end{lemma}

\begin{proof}
Let $j\colon T\to X$ be the corresponding inclusion.  Then the adjunction map $j_!j^*\FF\to\FF$ is part of an exact sequence of sheaves on $X$
$$
	0\to j_!j^*\FF\to \FF\to \QQ\to 0
$$
where $\QQ$ is a skyscraper sheaf supported on $X\ssm T$.  The bijection in question is part of the corresponding long exact sequence of cohomology
$$
	\cdots
	\to H^1_c(\bar{X},\QQ)
	\to H^2_c(\Tbar,\FF)
	\to H^2_c(\bar{X},\FF)
	\to H^2_c(\bar{X},\QQ)
	\to \cdots
$$
where $H^i_c(\bar{X},\QQ)$ vanishes for $i\neq 0$ since $\QQ$ is a skyscraper sheaf.
\end{proof}

Let $\GG$ be a sheaf on $X$ and $\GG^\vee$ be its dual.  Suppose $\FF$ and $\GG$ are lisse on $T$, and thus so is $\GG^\vee$.  Let $\rho\colon\piOneT\to\GL(V)$, $\omega\colon\piOneT\to\GL(W)$, and $\omega^\vee\colon\piOneT\to\GL(W^\vee)$ be the respective corresponding representations.

\begin{lemma}\label{lem:detecting-FG-isomorphisms}
Suppose $\FF$ and $\GG$ are lisse and  geometrically simple on $T$.

\medskip
\begin{enum}
\item $\dim(H^2_c(\Tbar,\FF\otimes\GG^\vee))=\dim(\Hom_{\piOneT}(W,V))\leq 1$.
\item $\dim(H^2_c(\Tbar,\FF\otimes\GG^\vee))=1$ if and only if $\FF$ and $\GG$ are geometrically isomorphic on $T$.
\end{enum}
\end{lemma}

\begin{proof}
Let $G=\piOneTbar$ so that $\rho$ and $\omega^\vee$ are absolutely simple representations of $G$ and $\rho\otimes\omega^\vee$ is the representation on $V\otimes W^\vee$ corresponding to $\FF\otimes\GG^\vee$.  Therefore
$$
	\dim(H^2_c(\Tbar,\FF\otimes\GG^\vee))
	\overset{\eqref{eqn:invariants-and-coinvariants}}=
	\dim\left((V\otimes W^\vee)_G\right)
	=
	\dim\left((V\otimes W^\vee)^G\right)
	=
	\dim\left(\Hom_G(W,V)\right)
$$
(cf.~\cite[43.14]{CurtisReiner}).  Moreover, the sheaves $\FF,\GG$ are geometrically isomorphic on $T$ if and only if $V$ and $W$ are isomorphic as representations of $G$.  If these equivalent conditions hold, then Schur's lemma implies $\dim(\Hom_G(W,V))=1$, and otherwise $\dim(\Hom_G(W,V))=0$ (see \cite[27.3]{CurtisReiner}).
\end{proof}


\subsection{Invariant scalars}

Let $\l\in\k^\times$.  If we identify $\Gm$ with $\Poneu\ssm\{0,\infty\}$ and regard $\l$ as an element of $\Gm(\k)$, then multiplication by it (i.e., translation) induces an automorphism of $\Poneu$ over $\k$ which we also denote $\l\colon\Poneu\to\Poneu$.  We say $\l$ is an \defi{invariant scalar} of $\GG$ iff the direct image $\l_*\GG$ is geometrically isomorphic to $\GG$.  For example, $1$ is an invariant scalar for every $\GG$, and every $\l$ is an invariant scalar of the constant sheaf $\Qellbar$.

Let $\alpha\colon\piOne{\Gm}\to\Qellbar^\times$ be a tame character.  The corresponding sheaf $\LL_\alpha=\ME{\alpha}$ is a so-called Kummer sheaf.

\begin{lemma}\label{lem:invariant-scalars-for-Kummer}
Every $\l\in\k^\times$ is an invariant scalar of $\LL_\alpha$.
\end{lemma}

\begin{proof}
The tame fundamental group of $\Gm$ is a quotient and completely generated by the images of the inertia groups $I(0)$ and $I(\infty)$.  The character $\alpha$ is completely determined by these images, and translation by $\lambda$ does not change how $I(0)$ and $I(\infty)$ act since it fixes both $0$ and $\infty$.  Therefore $\l_*\LL_\alpha$ and $\LL_\alpha$ are lisse and geometrically isomorphic on $\Gm$, and $\l$ is an invariant scalar of $\LL_\alpha$.
\end{proof}

\begin{cor}
$\l$ is an invariant scalar of $\GG$ if and only if it is an invariant scalar of $\GG\otimes\LL_\alpha$
\end{cor}

\noindent
In particular, the answer to the question of whether or not $\l$ is an invariant scalar of $\Q_*\ME{\rhochi}$ depends only on the coset $\dc\PhiUNu$.

\begin{proof}
The sheaves $\l_*\LL_\alpha$ and $\LL_\alpha$ are lisse and geometrically isomorphic on $\Gm$ by Lemma~\ref{lem:invariant-scalars-for-Kummer}.  Moreover,
$$
	\l_*(\GG\otimes\LL_\alpha)\otimes(\GG\otimes\LL_\alpha)^\vee
	=
	\l_*\GG\otimes(\l_*\LL_\alpha\otimes\LL_\alpha^\vee)\otimes\GG^\vee,
$$
so $\l_*\GG\otimes\GG^\vee$ and $\l_*(\GG\otimes\LL_\alpha)\otimes(\GG\otimes\LL_\alpha)^\vee$ are lisse and geometrically isomorphic on $U\ssm\{0,\infty\}$.  Thus $\l$ is an invariant scalar of $\GG$ if and only if it is an invariant scalar of $\GG\otimes\LL_\alpha$.
\end{proof}

  The following lemma gives a cohomological criterion for detecting invariant scalars.

\begin{lemma}\label{lem:detecting-invariant-scalars}
Let $\l\in\Fqbar^\times$.  Suppose $\l_*\GG$ and $\GG$ are lisse and geometrically simple on $U$.  Then the following are equivalent:

\smallskip
\begin{enum}
\item $\l$ is an invariant scalar of $\GG$;
\item $H^2_c(\Ubar,\l_*\GG\otimes\GG^\vee)\neq\ZeroSpace$;
\item $H^2(\PoneuBar,\l_*\GG\otimes\GG^\vee)\neq\ZeroSpace$.
\end{enum}
\end{lemma}

\begin{proof}
Lemma~\ref{lem:detecting-FG-isomorphisms} implies the equivalence of (1) and (2), and Lemma~\ref{lem:birational-invariance-of-H^2_c} implies the equivalence of (2) and (3).
\end{proof}


\subsection{Avoiding invariant scalars}

Consider the affine plane curve
$$
	X_\l
	:
	\l\Q(x_1)=\Q(x_2),
$$
and let $\pi_i\colon X_\l\to\Aonet$ be the map $(x_1,x_2)\mapsto x_i$.  They are part of a commutative diagram
$$
	\xymatrix{
		X_\l\ar[r]^{\pi_2}\ar[d]_{\pi_1}\ar@{.>}[dr]^\pi
		& \Aonet\ar[d]^{\Q} \\
		\Aonet\ar[r]_{\l\Q}
		& \Aoneu
	}
$$
where $\pi=\Q\pi_2=\l\Q\pi_1$.  Moreover, the maps $\Q$ and $\l\Q$ are generically \'etale of degree $n=\deg(\Q)$, thus their fiber product $\pi$ is generically \'etale of degree $n^2$.

Let $\EFq/\Fq$ be a finite extension over which $\Q$ splits and $Z=\{a_1,\ldots,a_n\}\seq\EFq$ be the zeros of $\Q$.

\begin{lemma}\label{lem:smoothness-of-X_l}
$X_\l$ is smooth over the $n^2$ points of $Z\times_{\Aoneu}Z=Z\times Z$.
\end{lemma}

\begin{proof}
The subset $Z\sub\Aonet$ is the vanishing locus of $\Q$ and $\l\Q$, hence $Z\times_{\Aoneu}Z=Z\times Z$.  Moreover,
$$
	\frac{\partial}{\partial x_2}(\l\Q(x_1)-\Q(x_2))
	=
	\Q'(x_2)
	=
	\sum_{i=1}^n\prod_{j\neq i}(x-a_j)
$$
does not vanish at any $a_i\in Z$ since $\Q$ is square free, so $X_\l$ is smooth at every $(a_i,a_j)\in Z\times Z$.
\end{proof}

Consider the external tensor product sheaf
$$
	\EE_{\rhochi,\l} := \ME{\rhochi}\boxtimes\ME{\rhochi}^\vee
$$
on $\Aonet\times\Aonet$ and the tensor product sheaf
$$
	\TT_{\rhochi,\l}
	:=
	\l\Q_*\ME{\rhochi}\otimes\Q_*\ME{\rhochi}^\vee
$$
one $\Poneu$.  They have respective generic ranks $r$ and $r^2$ since both $\ME{\rhochi}$ and its dual have generic rank $r$.

Let $T_\l\seq X_\l$ be a smooth dense Zariski open subset and $U_\l=\pi(T_\l)$.  Up to shrinking $T_\l$, we suppose that $\EE_{\rhochi,\l}$ is lisse on $T_\l$ and that $\pi$ is \'etale over $U_\l$.

\begin{lemma}\label{lem:focusing-on-E_lambda}
The sheaves $\pi_*(\EE_{\rhochi,\l})$ and $\TT_{\rhochi,\l}$ are lisse and isomorphic on $U_\l$.
\end{lemma}

\begin{proof}
Let $w$ be a geometric point of $U_\l$, and let $W_1=(\l\Q)^{-1}(w)$ and $W_2=\Q^{-1}(w)$.  Then $|W_1|=|W_2|=\deg(\Q)$ and $\pi^{-1}(w)=W_1\times W_2$ since $\pi$ is unramified over $w$, and
$$
	\pi_*(\EE_{\rhochi,\l})_w
	\ =
	\bigoplus_{(w_1,w_2)\in W_1\times W_2}
	\EE_{\rhochi,\l,(w_1,w_2)}
	\ =
	\bigoplus_{(w_1,w_2)\in W_1\times W_2}
	\left(\ME{\rhochi}_{w_1}\otimes\ME{\rhochi}^\vee_{w_2}\right)
$$
whereas
$$
	\TT_{\rhochi,\l,w}
	\ =
	\left(\bigoplus_{w_1\in W_1}\ME{\rhochi}_{w_1}\right)
	\otimes
	\left(\bigoplus_{w_2\in W_2}\ME{\rhochi}^\vee_{w_2}\right).
$$
Therefore both sheaves have the same geometric fibers, and hence they are isomorphic.  It remains to show they are lisse on $U_\l$.

On one hand, $\EE_{\rhochi,\l}$ is lisse on $T_\l$, so its geometric fibers all have the same rank $r^2$.  Moreover, $\Q$ is \etale{} over $U_\l$ by hypothesis, so the geometric fibers of $\pi_*(\EE_{\rhochi,\l})$ also all have the same rank $\dim(\Q)r^2$ and hence $\pi_*(\EE_{\rhochi,\l})$ is lisse on $U_\l$ (see \cite[Prop.~11]{Katz:SC}).  On the other hand, $\pi_*(\EE_{\rhochi,\l})$ is isomorphic to $T_{\rhochi,\l}$ on $U_\l$ which implies the latter is also lisse on $U_\l$.
\end{proof}

The contrapositive of the following corollary gives us a way to show some $\l$ is \emph{not} an invariant scalar.

\begin{cor}\label{cor:detecting-invariant-scalars}
Suppose $\rho$ is geometrically simple and $\dc\in\PhiQ$.  Then the following are equivalent:

\smallskip
\begin{enum}
\item $\l$ is an invariant scalar of $\Q_*\ME{\rhochi}$;
\item $H^2_c(\bar U_\lambda,\TT_{\rhochi,\l})\neq\ZeroSpace$.
\end{enum}

\smallskip\noindent
They imply

\smallskip
\begin{enum}
\setcounter{enumi}{2}
\item $H^2_c(\bar T_\l,\EE_{\rhochi,\l})\neq\ZeroSpace$.
\end{enum}

\end{cor}

\begin{proof}
Lemmas~\ref{lem:detecting-invariant-scalars} and \ref{lem:focusing-on-E_lambda} imply the equivalence of (1) and (2).  If $\piOne{U_\lambda}\to\GL(V)$ is the representation corresponding to $\TT_\l$, then
$
	V^{\piOne{U_\l}}
	\seq
	V^{\piOne{T_\l}}
$
so \eqref{eqn:invariants-and-coinvariants} and (2) imply (3).
\end{proof}

The following proposition was inspired by \cite[Proof of Th.~5.1.3]{Katz:TLFM}.

\begin{prop}\label{prop:no-nontrivial-invariant-scalars}
Suppose $\deg(\Q)\geq 2+\deg(\gcd(\Q,s))$ and $\dc\in\PhiQDistinct$.

\begin{enum}
\item If $\rho$ is geometrically irreducible, then so is $\ME{\rhochi}$.
\item $\l=1$ is the only invariant scalar of $\Q_*\ME{\rhochi}$.
\end{enum}
\end{prop}

\begin{proof}
Let $\EFq/\Fq$ be a splitting field of $\Q$ and $a_1,a_2\in\EFq$ be zeros of $\Q$ which are distinct from each other and the zeros of $s$.  Let $\dc_1,\dc_2\in\Hom(\EFq^\times,\bbC^\times)$ be the corresponding components of $(\sigma_\EFq^\vee)^{-1}(\nu_{\EFq}^{\,\prime\,\vee}(\dc))$ as an element of $(\sigma_\EFq^\vee)^{-1}(\PhiEOf\EFq\Q)$ (compare \eqref{eqn:identifying-PhiQ} and \eqref{eqn:phiQ-bijections}).  Then $\dc_1,\dc_2$ are distinct characters, so $\alpha=\dc_1/\dc_2$ is a non-trivial character.

Let $\l\in\k^\times$ be an arbitrary scalar.  If $\l\neq 1$, then for each component $T'_\l\seq T_\l$ over $\Fqbar$, there is a smooth point $t'=(t'_1,t'_2)\in T'_\l(\Fqbar)$ satisfying $\{t'_1,t'_2\}=\{a_1,a_2\}$.  The map $\pi$ is \'etale over $0$ since $\Q$ is square free, hence we can use $\pi$ to identify $I(t')$ with $I(0)$.  We can also identify $I(t'_1)$ and $I(t'_2)$ with $I(0)$.

On one hand, the fiber of $\ME{\rhochi}$ at $t=t'_i$ and the fiber at $t=0$ of $\Qellbar^{r}\otimes\LL_{\dc_i}$ are isomorphic as $I(0)$-modules since $s(a_i)\neq 0$.  Moreover, the fiber of $\EE_{\rhochi,\l}$ at $t'$ and the fiber at $u=0$ of $\Qellbar^{r^2}\otimes\LL_{\dc}$ are isomorphic as $I(0)$-modules.  On the other hand, the latter fibers have no $I(0)$-invariants since $\dc$ is non-trivial, so a fortiori, the geometric generic fiber of $\EE_{\rhochi,\l}$ has no $\piOne{\Tbar_\l}$-invariants.  Therefore \eqref{eqn:invariants-and-coinvariants} implies $H^2_c(\bar T_\l,\EE_{\rhochi,\l})$ vanishes for $\l\neq 1$, and hence the contrapositive of Corollary~\ref{cor:detecting-invariant-scalars} implies $\l=1$ is the only invariant scalar of $\Q_*\ME{\rhochi}$.
\end{proof}


\subsection{Baby theorem}

In this subsection we prove a simplified version of Theorem~\ref{thm:is-equidistributed}.  

Let $U$ be a dense Zariski open subset of $\Gm=\Poneu\ssm\{0,\infty\}$ and $\theta\colon\piOne{U}\to\GL(W)$ be a continuous representation to a finite-dimensional $\Qellbar$-vector space $W$.  Let $\PhiU$ be the dual of $\Bu=(\Fq[u]/u\Fq[u])^\times$ (cf.~\S\ref{sec:one-parameter-families}).  For $u=0,\infty$, let $W(u)$ denote $W$ regarded as an $I(u)$-module and $W(u)^\unip$ be its maximal submodule where $I(u)$ acts unipotently.  If $\theta$ is geometrically simple and punctually pure of weight $w$ and if $\dim(W)>1$, then we can associate to $\theta$ a pair of Tannakian monodromy groups
$$
	\GG_\geom(\theta,\PhiU)
	\seq
	\GG_\arith(\theta,\PhiU)
	\seq
	\GL_{\R,\Qellbar}
$$
for $\R=\chi(\GmBar,\ME{\theta})$ (see \S\ref{sec:representation-monodromy-groups} and Theorem~\ref{thm:P-is-neutral-Tannakian}).

\begin{theorem}\label{thm:baby-monodromy-theorem}
Suppose that $\theta$ is geometrically simple and punctually pure of weight $w$, that $\dim(W)>1$ or that $\theta$ does not factor through the composed quotient $\piOne{U}\onto\piOne{\Gm}\onto\piOneTame{\Gm}$, and that $\l=1$ is the only invariant scalar of $\ME{\theta}$.  Suppose moreover that $W(0)^\unip$ has dimension at most $r$ and a unique unipotent block of exact multiplicity one and that $\R>72(r^2+1)^2$.  Finally, suppose $W(\infty)^\unip=\ZeroSpace$.  Then $\GG_\geom(\theta,\PhiU)$ equals $\GL_{\R,\Qellbar}$.
\end{theorem}

\noindent
The proof consists of a few steps and will occupy the remainder of this section.

Let $G=\GG_\arith(\theta,\PhiU)$ and $H=\GG_\geom(\theta,\PhiU)$.

\begin{lemma}\label{lem:G/H-is-abelian}
$G$ and $H$ are reductive and there is an exact sequence
$$
	1\to H\to G\to T\to 1
$$
for some torus $T$ over $\Qellbar$.
\end{lemma}

\begin{proof}
Observe that $\ME{\theta}$ is geometrically simple yet is not a Kummer sheaf since otherwise one would have $\dim(W)=1$ and $\theta$ would factor through $\piOne{u}\onto\piOneTame{\Gm}$.  Moreover, $\theta$ is geometrically simple and punctually pure of weight $w$ by hypothesis.  Therefore the lemma follows from Proposition~\ref{prop:middle-extension-monodromy}.\ref{item:prop:mem-reductive-and-normal}.
\end{proof}

A priori $G$ or $H$ could be disconnected, so let $G^0$ and $H^0$ be the respective identity components.

\begin{lemma}\label{lem:lie-irreducible}
$G^0$ and $H^0$ are (Lie-)irreducible subgroups of $\GL_{\R,\Qellbar}$.
\end{lemma}

\begin{proof}
This follows from \cite[Th.~8.2 and Cor.~8.3]{Katz:CE} since $\l=1$ is the only invariant scalar of $\ME{\theta}$.
\end{proof}

Let $\wm{m}\colon(\Qbar^\times)^m\to\bbZ^m$ be the $m$th weight multiplicity map for $m=\R$ given in Definition~\ref{def:weight-multiplicity-map}.

\begin{lemma}\label{lem:nice-element}
There exist an element $g\in G^0$ and an eigenvalue tuple $\gamma\in(\Qellbar^\times)^R$ of $g$ satisfying the following:

\begin{enum}
\item $\gamma=(\gamma_1,\ldots,\gamma_\R)$ lies in $(\Qbar^\times)^\R$ and thus $\det(g)=\gamma_1\cdots\gamma_\R$ lies in $\Qbar^\times$;
\item\label{lem:item:ne-det} $|\iota(\det(g))|^2=(1/q)^w$ for some $w\neq 0$ and every field embedding $\iota\colon\Qbar\to\bbC$;
\item $c=\wm{\R}(\gamma)$ satisfies $\len(c)\leq r+1$ and $1=c_{\len(c)}<c_{\len(c)-1}$ and $c_2\leq r$.
\end{enum}
\end{lemma}

\begin{proof}
This follows from Proposition~\ref{prop:middle-extension-monodromy}.\ref{item:prop:mem-nonzero-weights} with $g=f^c$ for any element $f\in\Frob_{\Fq,\one} $ and for $c={[G:G^0]}$.  More precisely, if $\alpha=(\alpha_1,\ldots,\alpha_\R)$ is an eigenvalue tuple of $f$, then all the $\alpha_i$ lie in $\Qbar$, all the non-zero weights $w_1,\ldots,w_n$ of the $\alpha_i$ are negative since $W(\infty)^\unip$ vanishes, one has $1\leq n\leq r$ since $1\leq \dim(W(0)^\unip)\leq r$, there is a unique non-zero weight of multiplicity one since $W(0)^\unip$ has a unique unipotent block of exact multiplicity one, and the weight zero has multiplicity $\R-n\geq \R-r>1$.  Hence it suffices to take $\gamma\in(\Qbar^\times)^\R$ to be the eigenvalue tuple with $\gamma_i=\alpha_i^c$ for $1\leq i\leq \R$ and $w$ to be $(w_1+\cdots+w_n)c$.
\end{proof}

\begin{cor}\label{cor:det-H}
$\det(H)$ equals $\Qellbar^\times$.
\end{cor}

\begin{proof}
Follows from Lemma~\ref{lem:nice-element}.\ref{lem:item:ne-det} and the argument in  \cite[Proof of Th.~17.1]{Katz:CE} using the element $g$ in Lemma~\ref{lem:nice-element}.
\end{proof}

Let $[G^0,G^0]$ be the derived subgroup of $G^0$.

\begin{lemma}\label{lem:derived-of-G^0}
$[G^0,G^0]$ equals $\SL_{\R,\Qellbar}$.
\end{lemma}

\begin{proof}
Combine Lemmas~\ref{lem:lie-irreducible} and \ref{lem:nice-element} to deduce that the hypotheses of Theorem~\ref{thm:big-monodromy} hold, and thus $G^0$ equals one of $\SL_\R(\Qellbar)$ or $\GL_\R(\Qellbar)$.  The derived subgroup of both of these groups equals $\SL_\R(\Qellbar)$.
\end{proof}

We may now complete the proof of the theorem.  First, we have inclusions
$$
	[G^0,G^0]
	\seq
	[G,G]
	\seq
	[\GL_{\R,\Qellbar},\GL_{\R,\Qellbar}]
	=
	\SL_{\R,\Qellbar},
$$
and Lemma~\ref{lem:derived-of-G^0} implies the outer terms are equal, so the inclusions are equalities.  Moreover, Lemma~\ref{lem:G/H-is-abelian} implies $H$ is normal in $G$ and $G/H$ is abelian, so $H$ contains $[G,G]=\SL_{\R,\Qellbar}$, and hence, by Corollary~\ref{cor:det-H}, $H=\GL_{\R,\Qellbar}$ as claimed.


\subsection{Frobenius reciprocity}

Let $\Q\colon T\to U$ be a finite \'etale map of smooth geometrically connected curves over $\Fq$.  Let $\FF$ (resp.~$\GG$) be a lisse sheaf on $T$ (resp.~$U$) and $\piOneT\to\GL(V)$ (resp.~$\piOneU\to\GL(W)$) be the corresponding representation.  Let $\FF^\vee$ be the dual of $\FF$ and $\piOneT\to\GL(V^\vee)$ be the corresponding representation.

\begin{lemma}\label{lem:dual-is-unambiguous}
$\Q_*(\FF^\vee)$ is isomorphic to the dual of $\Q_*\FF$.
\end{lemma}

\begin{proof}
See \cite[Lem.~3.1.3]{Katz:TLFM}.
\end{proof}

\noindent
Therefore we may unambiguously write $\Q_*\FF^\vee$.

\begin{prop}\label{prop:frobenius-reciprocity}
$
	\dim(H^2_c(\Tbar,\Q^*\GG\otimes\FF^\vee))
	=
	\dim(H^2_c(\Ubar,\GG\otimes\Q_*\FF^\vee)).
$
\end{prop}

\begin{proof}
Let $H=\piOneTbar$ and $G=\piOneUbar$.  We suppose that $V$ (resp.~$W$) is a left $H$-module (resp.~$G$-module), and define $\Ind_H^G(V)$ to be the (Mackey) induced module $\Hom_G(\Qellbar[H],V)$ and $\Res_H^G(W)$ to be the restricted module $W$ regarded as a left $H$-module.  Then Frobenius reciprocity implies that there is a bijection of vector spaces
$$
	\Hom_H(\Res_H^G(W),V)
	\to
	\Hom_G(W,\Ind_H^G(V))
$$
given by $\psi\mapsto (w\mapsto (r\mapsto \psi(rv)))$ (cf.~\cite[\S 3.0]{Katz:TLFM}).  Moreover, Lemma~\ref{lem:detecting-FG-isomorphisms} implies that
$$
	\dim(H^2_c(\Tbar,\Q^*\GG\otimes \FF^\vee))
	=
	\dim(\Hom_H(\Res_H^G(W),V))
$$
and that
$$
	\dim(H^2_c(\Ubar,\GG\otimes\Q_*\FF^\vee))
	=
	\dim(\Hom_G(W,\Ind_H^G(V))),
$$
so the proposition follows immediately.
\end{proof}


\subsection{Begetting simplicity}

In this section we give a criterion for $\Ind(\rhochi)$ to be geometrically simple.  Our argument was inspired by \cite[Proof of Th.~5.1.1]{Katz:QKR}.

\begin{prop}\label{prop:induced-simplicity}
Let $\dc\in\PhiQDistinct$.  Suppose that $\gcd(\Q,s)=t$, that $\deg(\Q)\geq 2$, and that $\dc(\Gamma(t))=\OneSpace$.  If $\rho$ is geometrically simple, then so are $\rhochi$ and $\Ind(\rhochi)$.
\end{prop}

\begin{proof}
Let $T\seq\Ponet$ be a dense Zariski open subset and $U=\Q(T)$.  Up to shrinking $T$, we suppose that $\FF=\ME{\rhochi}$ is lisse over $T$ and that $\Q$ is \'etale over $U$.

Suppose that $\rho$ is geometrically simple and thus so is $\rhochi$.  Let $\GG=\Q_*\FF^\vee$ (cf.~Lemma~\ref{lem:dual-is-unambiguous}), and observe that Lemma~\ref{lem:direct-image-of-middle-extension}.\ref{lem:ind-is-me} implies that $\GG$ and $\ME{\Ind(\rhochi)}^\vee$
are isomorphic over $U$.  We wish to show that $\dim(H^2(\Ubar,\GG\otimes\GG^\vee))=1$ so that Lemma~\ref{lem:detecting-FG-isomorphisms} implies that $\ME{\Ind(\rhochi)}$ is geometrically simple over $U$, that is, that $\Ind(\rhochi)$ is geometrically simple.  In fact, Lemma~\ref{lem:birational-invariance-of-H^2_c} and Proposition~\ref{prop:frobenius-reciprocity} imply that
$$
	\dim(H^2_c(\PoneuBar,\GG\otimes\GG^\vee))
	=
	\dim(H^2_c(\Ubar,\Q_*\FF\otimes\Q_*\FF^\vee))
	=
	\dim(H^2_c(\Ubar,\Q^*\Q_*\FF\otimes\FF^\vee)),
$$
so it suffices to show the last term equals 1.

The functor $\Q^*$ is left adjoint to the functor $\Q_*$ since $\Q$ is finite (cf.~\cite[II.3.14]{Milne}), so the identify map $\Q_*\FF\to\Q_*\FF$ induces an adjoint $\Q^*\Q_*\FF\to\Q$.  Generically it is the trace map $\Ind(V_\dc)\to V_\dc$ and thus is surjective (cf.~\cite[V.1.12]{Milne}).  Let $\KK$ be the kernel so that we have an exact sequence of sheaves
\begin{equation}\label{eqn:untensored-sequence}
	0\to \KK\to \Q^*\Q_*\FF\to\FF\to  0.
\end{equation}
These sheaves and $\FF^\vee$ are all lisse over $T$ and free, so the sequence
\begin{equation}\label{eqn:tensor-sequence}
	0\to \KK\otimes\FF^\vee\to \Q^*\Q_*\FF\otimes\FF^\vee\to\FF\otimes\FF^\vee\to  0
\end{equation}
is exact on $T$.  In particular, we have a corresponding exact sequence of cohomology
$$
	H^2_c(\Ubar,\KK\otimes\FF^\vee)
	\to H^2_c(\Tbar,\Q^*\Q_*\FF\otimes\FF^\vee)
	\to H^2_c(\Tbar,\FF\otimes\FF^\vee)
	\to H^3_c(\Tbar,\KK\otimes\FF^\vee)
$$
the last term of which vanishes.  The hypothesis that $\FF$ is geometrically simple implies the penultimate term has dimension 1 by Lemma~\ref{lem:detecting-FG-isomorphisms}, so it suffices to show that the first term vanishes.

Let $\EFq/\Fq$ be a splitting field of $\Q$, let $a_1,\ldots,a_n\in\EFq$ be the zeros of $\Q$, and let
$$
	(\dc_1,\ldots,\dc_n)
	=
	(\sigma_\EFq^\vee)^{-1}(\nu_{\EFq}^{\,\prime\,\vee}(\dc))
	\in\Hom(\EFq^\times,\bbC^\times)^n
$$
as in \eqref{eqn:identifying-PhiQ}.  We suppose without loss of generality that $a_1=0$ and thus $s(a_2)\cdots s(a_n)\neq 0$ since $\gcd(\Q,\s)=1$.

Let $G=\piOneTbar$ and $H=\piOneUbar$, and let $G\to\GL(V_\dc)$ and $H\to\GL(\Ind_H^G(V_\dc))$ be the representations corresponding to $\FF$ and $\Q_*\FF$ respectively.  The exact sequences \eqref{eqn:untensored-sequence} and \eqref{eqn:tensor-sequence} correspond to exact sequences of $G$-modules
\begin{equation}\label{eqn:untensored-sequence:bis}
	0\to K\to R\to V_\dc\to 0
\end{equation}
and
$$
	0\to K\otimes V_\dc^\vee\to R\otimes V_\dc^\vee\to V_\dc\otimes V_\dc^\vee\to 0
$$
where $R=\Res_H^G(\Ind_H^G(V_\dc))$.  We claim the first term of the latter sequence has no $I(0)$-convariants so a fortiori has no $\piOneTbar$-convariants, and hence $H^2(\Tbar,\KK\otimes\FF^\vee)$ vanishes as claimed.

The translation map $t\mapsto t+a_i$ induces an isomorphism $I(0)\simeq I(a_i)$ for each $i\in[n]$, so we can regard $V_\dc(a_i)$ as an $I(0)$-module.  In fact, we have isomorphisms of $I(0)$-modules
$$
	R(0)
	\simeq
	\bigoplus_{i=1}^n V_\dc(a_i)
	,\quad
	K(0)
	\simeq
	\bigoplus_{i=2}^n V_\dc(a_i)
	,\quad
	(K\otimes V_\dc^\vee)(0)
	\simeq
	\bigoplus_{i=2}^n(\Qellbar^{r-1}\otimes \dc_i^{-1}).
$$
More precisely, the first isomorphism corresponds to the fact that the geometric fibers of $\Q^*\Q_*\FF$ and $\FF$ satisfy $(\Q^*\Q_*\FF)_0=\oplus_{\Q(a)=0}\FF_a$ since $\Q$ is \etale{} over $u=0$ (cf.~\cite[II.3.5]{Milne}); the second isomorphism uses \eqref{eqn:untensored-sequence:bis} and the assumption that $a_1=0$ to identify $K(0)$ with $R(0)/V_\dc(0)$; and the last isomorphism uses that $s(a_2)\cdots s(a_n)\neq 0$, that is, $\CC\ssm\{a_1\}$ lies in the locus of lisse reduction of $\ME{\rhochi}^\vee$.

The hypothesis that $\Gamma(t)$ is in the kernel of $\dc$ implies that $V_\dc(0)\simeq V(0)$ as $I(0)$-modules.  Moreover, $\dc_2,\ldots,\dc_n$ are all non-trivial since they are distinct from the trivial character $\dc_1$ by hypothesis, so each of the summands $(\Qellbar^{r-1}\otimes \dc_i^{-1})$ has \emph{trivial} $I(0)$-coinvariants.  Therefore $K\otimes V_\dc^\vee$ has trivial $\piOneTbar$-coinvariants as claimed.
\end{proof}


\subsection{Preserving unipotent blocks}

For each monic divisor $\Qp$ of $\Q$ in $\Fq[t]$, consider the subset
$$
	\PhiQpGood\rho
	=
	\{\,
		\dc\in\PhiQp
		:
		\ME{\rhochi}
		\mbox{ is supported on }\Aonet[1/\Qp]
	\,\}.
$$
If $\rho$ is the trivial representation, then it consists of the odd primitive characters of conductor $\Qp$.

For $t=0,\infty$, let $V_\dc(t)$ denote $V_\dc$ regarded as an $I(t)$-module.  Similarly, for $u=0,\infty$, let $\Ind(V_\dc)(u)$ denote $\Ind(V_\dc)$ regarded as an $I(u)$-module, and let $\Ind(V_\dc)(u)^\unip$ be the maximal submodule of $\Ind(V_\dc)(u)$ where $I(u)$ acts unipotently.  We say that $\Ind(V_\dc)(0)$ (resp.~$V_\dc(0)$) has a \defi{unipotent block of dimension $e$ and exact multiplicity $m$} iff it has an $I(0)$-submodule isomorphic to $U(e)^{\oplus m}$ but no $I(0)$-submodule isomorphic to $U(e)^{\oplus m+1}$.

\begin{lemma}\label{lem:induced-unipotent}
Suppose $\gcd(\Q,s)=t$, and let $\Qp=\Q/t$ and $\dc\in\PhiQDistinct\cap\PhiQpGood\rho$.  Then

\smallskip
\begin{enum}
\item\label{item:unip-0} $\Ind(V_\dc)(0)$ has a unipotent block of dimension $e$ and exact multiplicity $m$ if and only if $V(0)$ does;

\item\label{item:unip-infty} $\Ind(V_\dc)(\infty)^\unip=\ZeroSpace$.

\end{enum}
\end{lemma}

\begin{proof}
On one hand, $V_\dc(z)^\unip=\ZeroSpace$ for every $z\in\CC\ssm\{0\}$ since $\dc$ is in $\PhiQpGood\rho$ and $\gcd(\Qp,\s)=1$.  Moreover, $V_\dc(0)$ and $V(0)$ are isomorphic as $I(0)$-modules since $\dc(\Bt)=\OneSpace$.  Therefore the only unipotent blocks of $\Ind(V_\dc)(0)$ are those coming from $V_\dc(0)$, and all such blocks contribute identical blocks to $V_\dc(0)$, so \eqref{item:unip-0} holds.  On the other hand, every unipotent block of $\Ind(V_\dc)(\infty)$ contributes to $V_\dc(\infty)^\unip$, and the latter vanishes since $\dc$ is $\rho$-primitive, so \eqref{item:unip-infty} holds.
\end{proof}


\subsection{Proof of Theorem~\ref{thm:is-equidistributed}}\label{subsec:proof-of-equidistribution-theorem}

%
%

Recall that $\R$ is given by
\begin{equation}\label{eqn:R-recall}
	\R
	:=
	\rC(\rho)
	=
	(\deg(\Q)+1)r + \deg(L(T,\rho)) - \dropCee{\rho}
\end{equation}
and it equals $\deg(\LC(T,\rhochi))$ for all $\dc\in\PhiQ$ (see Theorem~\ref{thmB}).

\begin{lemma}\label{lem:R-lower-bound}
$R>72(r^2+1)^2$
\end{lemma}

\begin{proof}
Follows from \eqref{eqn:R-recall} and the hypothesis on $\deg(\Q)$ in the statement of the theorem.
\end{proof}

Let $\Qp=\Q/t$.

\begin{lemma}\label{lem:induced-representation}
Suppose $\dc\in\PhiQDistinct\cap\PhiQpGood\rho$.  Then the following hold:

\smallskip
\begin{enum}
\item\label{lem:item:ir-simple} $\Ind(\rhochi)$ is geometrically simple;
\item\label{lem:item:ir-0-unipotent} $\dim(\Ind(V_\dc)(0)^\unip)=\dim(V_\dc(0)^\unip)$ and $\Ind(V_\dc)(0)$ has a unique unipotent block of exact multiplicity one;
\item\label{lem:item:ir-infty-unipotent} $\Ind(V_\dc)(\infty)^\unip=\ZeroSpace$.
\end{enum}
\end{lemma}

\begin{proof}
Part~\eqref{lem:item:ir-simple} follows from Proposition~\ref{prop:induced-simplicity} since $\dc$ is in $\PhiQDistinct\cap\PhiQp$, since $\rho$ is geometrically simple, and since $\deg(\Q)\geq 2$.  Parts~\eqref{lem:item:ir-0-unipotent} and \eqref{lem:item:ir-infty-unipotent} follow from Lemma~\ref{lem:induced-unipotent} since $\dc$ is also in $\PhiQpGood\rho$ and since $V(0)$ has a unique unipotent block of exact multiplicity one.
\end{proof}

\begin{cor}\label{cor:big-subset}
$(\PhiQDistinct\cap\PhiQpGood\rho)\seq\PhiQBig\rho$.
\end{cor}

\begin{proof}
Let $\dc\in\PhiQDistinct\cap\PhiQpGood\rho$, and let $\theta=\Ind(\rhochi)$ and $W=\Ind(V_\dc)$.  Then Lemmas~\ref{lem:induced-representation} and \ref{lem:Ind-rhochi-pure} imply that $\theta=\Ind(\rhochi)$ is geometrically simple and punctually pure of weight $w$ since $\dc\in\PhiQDistinct$.  Moreover, $\dim(W)=\deg(\Q)\cdot\dim(V)>2$ since $\deg(\Q)\geq 2$, and Proposition~\ref{prop:no-nontrivial-invariant-scalars} implies that $\lambda=1$ is the only invariant scalar of $\ME{\theta}\simeq\Q_*\ME{\rhochi}$ since $\deg(\Q)\geq 3$ and $\dc\in\PhiQDistinct$.  Lemma~\ref{lem:induced-representation} also implies that $W(0)$ has a unique unipotent block of exact multiplicity one, that $\dim(W(0)^\unip)=\dim(V(0)^\unip)\leq\dim(V)=r$, and that $W(\infty)^\unip=\ZeroSpace$.  Finally, Lemma~\ref{lem:R-lower-bound} implies $R>72(r^2+1)^2$.  Therefore the hypotheses of Theorem~\ref{thm:baby-monodromy-theorem} hold, and hence $\dc\in\PhiQBig\rho$.
\end{proof}

\begin{cor}\label{cor:big-subset-of-big}
$(\PhiQDistinct\cap\PhiQpGood\rho)\PhiUNu\seq\PhiQBig\rho$.
\end{cor}

\begin{proof}
Follows from Corollary~\ref{cor:big-subset} since $\PhiQBig\rho$ is a union of cosets $\dc\PhiUNu$.
\end{proof}

Let $\dc\in\PhiQ$ and $\dc\PhiUNu$ be the corresponding coset.

\begin{lemma}\label{lem:unique-alpha}
$|\dc\PhiUNu\cap\PhiQp|=1$.
\end{lemma}

\begin{proof}
We must show that there is a unique element $\alpha\in\PhiU$ satisfying $\dc\alpha^\nu(\Gamma(t))=\OneSpace$.  Since $\gcd(\s,\Q)=t$, we can speak of the component of $\dc$ at $t=0$: it is the character given by restricting $\chi$ to the subgroup $\Gamma(t)\seq\BQ$.  There is a unique element of $\PhiUNu$ with the same component at $t=0$, call it $\beta^\nu$.  Then $\alpha=1/\beta$ is the desired character.
\end{proof}

We need one more estimate to complete the proof of the theorem.

\begin{lemma}\label{lem:counting-big}
$
	|\PhiQDistinct\cap\PhiQpGood\rho|
	\sim
	|\PhiQpDistinct|
	\sim
	|\PhiQp|
	\mbox{ as }
	q\to\infty.
$
\end{lemma}

\begin{proof}
We observe that there are natural inclusions
$$
	\left(\PhiQpDistinct\ssm\cup_{\pi\mid\Qp}\PhiOf{\Qp/\pi}\right)
	\seq
	(\PhiQDistinct\cap\PhiQp)
	\seq
	\PhiQpDistinct
$$
since an element of $\PhiQpDistinct$ will fail to lie in $\PhiQDistinct$ only if one of its $\deg(\Qp)$ components is trivial, that is, if it lies in $\PhiOf{\Qp/\pi}$ for some prime factor $\pi\mid\Qp$.  Intersecting with $\PhiQpGood\rho$ gives further inclusions
$$
	\left((\PhiQpGood\rho\cap\PhiQpDistinct)\ssm\cup_{\pi\mid\Qp}\PhiOf{\Qp/\pi}\right)
	\seq
	(\PhiQDistinct\cap\PhiQpGood\rho)
	\seq
	\PhiQpDistinct.
$$
Finally, we know that
$$
	|\PhiQpGood\rho|
	\overset{\mathrm{Lem.~}\ref{lem:Qp-bn-bound}}\sim
	|\PhiQp|
	\overset{\mathrm{Cor.~}\ref{cor:size-of-Qp-distinct}}\sim
	|\PhiQpDistinct|
	,\quad
	|\cup_{\pi\mid\Qp}\PhiOf{\Qp/\pi}|/|\PhiQ| \ll 1/q = o(1)
$$
and hence
$$
	\left|(\PhiQpGood\rho\cap\PhiQpDistinct)\ssm\cup_{\pi\mid\Qp}\PhiOf{\Qp/\pi}\right|
	\sim
	|\PhiQp|
$$
as $q\to\infty$.
\end{proof}

\begin{cor}\label{cor:counting-big}
$
	|(\PhiQDistinct\cap\PhiQpGood\rho)\PhiUNu|
	\sim
	|\PhiQ|
$
for $q\to\infty$.
\end{cor}

\begin{proof}
Combine Lemma~\ref{lem:unique-alpha} and Lemma~\ref{lem:counting-big}.
\end{proof}

\noindent
The theorem now follows by observing that
$$
	|\PhiQ|
	\overset{\mathrm{Cor.~}\ref{cor:counting-big}}
	\sim
	|(\PhiQDistinct\cap\PhiQpGood\rho)\PhiUNu|
	\overset{\mathrm{Cor.~}\ref{cor:big-subset-of-big}}
	\leq
	|\PhiQBig\rho|
	\leq
	|\PhiQ|
$$
and thus
$$
	|\PhiQBig\rho|
	\sim
	|\PhiQ|
$$
for $q\to\infty$.

\bs\noindent
$\therefore$\ \ The Mellin transform of $\rho$ has big monodromy as claimed and Theorem~\ref{thm:is-equidistributed} holds.


\section{Application to Explicit Abelian Varieties}\label{sec:explicit-abelian-varieties}

In this section we apply the theory developed in the previous sections to  representations coming from (the Tate modules of) a general class of abelian varieties.  More precisely, we give an explicit family of abelian varieties for which we can show the corresponding representations satisfy the hypotheses of Theorem~\ref{thm:is-equidistributed}.  Our principal application, of which Theorem~\ref{thm:intro-theorem} is a special case, is Theorem~\ref{thm:application}.

Throughout this section we suppose that $q$ is an odd prime power so that we can speak of hyperelliptic curves.  One who is interested in even characteristic or in $L$-functions whose Euler factors have odd degree is encouraged to consider Kloosterman sheaves (e.g., see \cite[7.3.2]{Katz:GKM}).


\subsection{Some hyperelliptic curves and their Jacobians}\label{Hyperelliptic curves and their Jacobians}

Let $g$ be a positive integer.  In this section we construct an explicit family of abelian varieties which give rise to Galois representations we can easily show satisfy the hypotheses Theorem~\ref{thm:variance-estimate}.  One member of this family is an elliptic curve, the Legendre curve, and it has affine model
$$
	\Xleg : y^2 = x(x-1)(x-t).
$$
It is isomorphic to its own Jacobian, and the general abelian varieties in our family will be Jacobians of curves.  More precisely, we fix a monic square free $f\in\Fq[x]$ of degree $2g$ and consider the projective plane curve $X/K$ with affine model
\begin{equation}\label{eq:X}
	X \colon y^2 = f(x)(x-t).
\end{equation}
For technical reasons we will eventually suppose that $f$ has a zero $a$ in $\Fq$, and up to the change of variables $x\mapsto x+a$, we will suppose that $a=0$.  We do not need this hypothesis yet since the discussion in this section does not use it.

The curve $X$ has genus $g$.  If $g>1$, it is a so-called hyperelliptic curve, and otherwise it is an elliptic curve.  Either way its Jacobian $J$ is a $g$-dimensional principally polarized abelian variety over $K$.  See \cite{handbook:ecc} for more information about hyperelliptic curves and their Jacobians.

For each finite place $v=\pi$, one can define a reduction $X/\Fpi$ starting with the reduction of \eqref{eq:X} modulo $\pi$.

\begin{lemma}\label{lem:hyper-reduction}
The monic polynomial $\s=f(t)\in\Fq[t]$ satisfies the following:
\begin{enum}
\item if $\pi\nmid\s$, then $X/\Fpi$ is a smooth projective curve of genus $g$;
\item if $\pi\mid\s$, then $X/\Fpi$ is smooth away from a single node and has genus $g-1$.
\end{enum}
\end{lemma}

\begin{proof}
The essential point is that, for any monic polynomial $h(x)$ with coefficients in a field $F$ of characteristic not two, the affine curve $y^2=h(x)$ is smooth iff $h$ is a square free polynomial.  More generally, if $h=h_1h_2^2$ where $h_1,h_2\in F[x]$ are square free and relatively prime, then the following hold:
\begin{enum}
\item the map $(x,y)\mapsto (x,y/h_2(x))$ induces a birational map from $y^2=h_1(x)$ to $y^2=h(x)$;
\item the $\deg(h_2)$ points $(x,y)$ satisfying $h_2(x)=y=0$ are so-called nodes of $y^2=h(x)$;
\item the map in (1) corresponds to blowing up the nodes in (2);
\item the curve $y^2=h_1(x)$ is smooth of genus $\lfloor(\deg(h_1)-1)/2\rfloor$ since $h_1$ is square free;
\item both curves have one (resp.~two) points at infinity if $\deg(h)$ is odd (resp.~even).
\end{enum}

\noindent
(Compare \cite[Ex.~I.5.6]{Hartshorne}.)  The proof of the lemma will consist of showing that we are in this general situation.

Let $t_0\in\Fpi$ satisfy $t\equiv t_0\bmod\pi$, and let $h_0(x):=f(x)(x-t_0)\in\Fpi[x]$.  The polynomial $f(x)$ is square free by hypothesis, so $h_0(x)$ is square free iff $f(t_0)=0$, or equivalently, $\pi\mid\s$.  In particular, if $\pi\nmid\s$, then $h_0$ is square free and $y^2=h_0(x)$ is smooth of genus $g$.  Otherwise, $h_0=h_1h_2^2$ where $h_1=f/(x-t_0)$ and $h_2=x-t_0$ are coprime (since $f$ is square free), and thus $y^2=h_0(x)$ is smooth away from the node $(t_0,0)$ and birational to the curve $y^2=h_1(x)$ which is smooth of genus $g-1$.
\end{proof}

\begin{remark}\label{rmk:claim-about-reduction-at-infinity}
One can also define a reduction $X/\Finf$ by writing $t=1/u$ and clearing denominators, and one eventually finds that $X/\Finf$ has genus zero.  However, the arguments are subtler and beyond the scope of this article, so we omit them.
\end{remark}

For example, $\Xleg$ has smooth reduction away from $t=0,1,\infty$, over $t=0,1$ its reduction is a so-called node, and over $t=\infty$ it is a so-called cusp.  Since it is isomorphic to its Jacobian, these are sometimes refers to these as good, multiplicative, and additive reduction respectively.  However, in general, one needs to construct separately reductions $J/\Fpi$, for every $\pi$, and also a reduction $J/\Finf$.

\begin{lemma}\label{lem:J_pi}\ 
\begin{enum}
\item If $\pi\nmid\s$, then $J/\Fpi$ is the Jacobian of $X/\Fpi$ so is a $g$-dimensional abelian variety;
\item If $\pi\mid\s$, then $J/\Fpi$ is an extension of an abelian variety by a one-dimensional torus.
\end{enum}
\end{lemma}

\begin{proof}
Both statements are easy consequences of Lemma~\ref{lem:hyper-reduction}.  More precisely, if $X/\Fpi$ is projective and smooth away from $n$ nodes, then $J/\Fpi$ is an extension of a $(g-n)$-dimensional abelian variety by an $n$-dimensional torus.  See \cite[9.2.8]{BLR} and keep in mind Lemma~\ref{lem:hyper-reduction}.
\end{proof}

\begin{remark}\label{rmk:reduction-of-J-at-infty}
One can also show that $J/\Finf$ is a $g$-dimensional additive linear algebraic group, but demonstrating it directly is harder and requires a finer statement than the claim in Remark~\ref{rmk:claim-about-reduction-at-infinity}.
\end{remark}

One can regard the various reductions of $J$ as the special fibers of the (identity component of the) N\'eron model of $J/K$ over $\Ponet$.  However, for our purposes, Lemma~\ref{lem:J_pi} contains all the information we need about the model.  More precisely, we only need to know the respective dimensions $g_\pi$, $m_\pi$, and $a_\pi$ of the good, multiplicative, and additive parts of $J/\Fpi$.  Thus
\begin{equation}\label{eqn:gma-values}
	(g_\pi,m_\pi,a_\pi)
	=
	\begin{cases}
		(g,0,0)   & \mbox{if }\pi\nmid\s \\
		(g-1,1,0) & \mbox{if }\pi\mid\s
	\end{cases}
\end{equation}
by Lemma~\ref{lem:J_pi}.  In \S\ref{sec:tate-modules} we will show that
$$
	(g_\infty,m_\infty,a_\infty) = (0,0,g)
$$
as claimed in Remark~\ref{rmk:reduction-of-J-at-infty}.


\subsection{Tate modules}\label{sec:tate-modules}

Let $\ell$ be a prime distinct from the characteristic $p$ of $\Fq$.  For each $m\geq 0$, let $J[\ell^m]\seq J(\bar{K})$ be the subgroup of $\ell^m$-torsion; it is isomorphic to $(\bbZ/\ell^m)^{2g}$ and hence is a finite Galois module.  Multiplication by $\ell$ induces an epimorphism $J[\ell^{m+1}]\onto J[\ell^m]$, for each $m$, and the $\Zell$-Tate module of $J$ is the projective limit
$$
	T_\ell(J) := \varprojlim J[\ell^m].
$$
Concretely one can regard $T_\ell(J)$ as the set
$$
	\{\,
		(P_0,P_1,\ldots)
		:
		P_m\in J[\ell^m]\mbox{ and }\ell P_{m+1}=P_m\mbox{ for }m\geq 0
	\,\}.
$$
It is even a Galois $\bbZ_\ell$-module (since the action of $G_K$ and multiplication by $\ell$ commute), and it is isomorphic to $\Zell^{2g}$ as a $\Zell$-module (cf.~\cite[\S1]{SerreTate}).

Let $V$ be the vector space $T_\ell(J)\otimes_\Zell\Qellbar$ and $\GK\to\GL(V)$ be the corresponding Galois representation.  For each $v\in\PP$, let $V(v)$ denote $V$ as an $I(v)$-module and let $V(v)^\unip$ be the maximal submodule where $I(v)$ acts unipotently.

\begin{prop}\label{prop:tate-module-invariant-dimensions}
Let $v\in\PP$, and let $g_z$ and $m_z$ be the respective dimensions of the abelian and multiplicative part of $J/\bbF_v$  Then
$$
	V(v)^\unip\simeq U(1)^{\oplus 2g_v}\oplus U(2)^{\oplus m_v}.
$$
\end{prop}

\begin{proof}
This is a general fact about Tate modules of abelian varieties.  See \cite[Exp.~IX, \S2.1]{SGA7}.
\end{proof}

Let $\SS=\{\pi\in\PP:\pi\mid\s\}\cup\{\infty\}$ where $\s=f(t)$ as in Lemma~\ref{lem:hyper-reduction}.  Then by Proposition~\ref{prop:tate-module-invariant-dimensions}, the action of $G_K$ on $V$ induces a representation
$$
	\rho\colon\GKS\to\GL(V)
$$
since
$$
	\dim(V^{I(v)})=\dim(V)=2g
	\mbox{ for }
	v\in\PP\ssm\SS
$$
by \eqref{eqn:gma-values}.

\begin{lemma}\label{lem:numerical-invariants-for-J}
$\rho$ is geometrically simple and punctually pure of weight one, and it satisfies
$$
	\dr_v(\rho)
	=
	\begin{cases}
    	 0 & v\in\PP\ssm\SS \\
    	 1 & v\in\SS\ssm\{\infty\} \\
    	2g & v=\infty
	\end{cases},\quad
	\swan(\rho) = 0.
$$
\end{lemma}

\begin{proof}
The values $\dr_v(\rho)$ for $v\neq\infty$ follow directly from \eqref{eqn:gma-values} since
$$
	\dr_v(\rho) = \dim(V) - \dim(V^{I(v)}) = 2g - 2g_v - m_v
$$
by Proposition~\ref{prop:tate-module-invariant-dimensions}.  For the assertions about geometric simplicity and weight and about $\dr_\infty(\rho)$ and $\swan(\rho)$ we refer to \cite[10.1.9 and 10.1.17]{KS} (cf.~\cite[\S 5]{Hall:BM} for a related discussion about $J[\ell]$).
\end{proof}

\begin{cor}
$L(T,J/K)=1$, that is, it is a polynomial and $\deg(L(T,J/K)) = 0$.
\end{cor}

\begin{proof}
The representation $\rho$ is geometrically simple and $\dim(V)=2g>0$, so $\rho$ has trivial geometric invariants.  Moreover, it is punctually pure of weight $w=1$, so Theorem~\ref{thm:archimedean-bound} implies $L(T,\rho)$ is a polynomial of degree
$$
	\degL(\rho)
	=
	\dr(\rho)
	+
	\swan(\rho)
	-
	2\cdot\dim(\Vl)
	\overset{\mathrm{Lem.~}\ref{lem:numerical-invariants-for-J}}{=}
	(\deg(f)\cdot 1+1\cdot 2g)
	+
	0
	-
	2\cdot 2g
	=
	0
$$
as claimed.
\end{proof}

Let $\Q\in\Fq[t]$ be monic and square free and $\CC\sub\PP$ be the finite subset consisting of $\pi$ and $v(\pi)$ for every prime factor $\pi$ of $\Q$ (cf.~\S\ref{sec:twisted-l-functions}).

\begin{lemma}\label{lem:rhochi-properties}
For every $\dc\in\PhiQ$, the representation $\rhochi$ is geometrically simple and punctually pure of weight one, and $\dc$ is not heavy.
\end{lemma}

\begin{proof}
Lemma~\ref{lem:rhochi-pure}.\ref{lem:item:rhochi-pure--simplicity} implies that $\rhochi$ is geometrically simple since $\rho$ is.  Moreover, it has trivial geometric invariants since $\dim(V)=2g>1$, so $\dc$ is not heavy.  Finally, Lemma~\ref{lem:rhochi-pure}.\ref{lem:item:rhochi-pure--purity} implies that it is punctually pure of weight $w=1$ since $\rho$ is.
\end{proof}

\begin{cor}
If $\dc\in\PhiQ$, then $\LC(T,\rhochi)$ is a polynomial and
$$
	\deg(\LC(T,\rhochi))
	=
	2g\cdot \deg(\Q) - \deg(\gcd(\Q,\s)).
$$
\end{cor}

\begin{proof}
By Lemma~\ref{lem:rhochi-properties} the hypotheses of Theorem~\ref{thmB} hold, and hence $\LC(T,\rhochi)$ is a polynomial of degree
$$
	\rC(\rho)
	=
	\deg(L(T,\rho)) + (\deg(\Q)+1)\dim(V) - \dropCee{\rho}
	=
	2g\cdot (\deg(\Q)+1) - \dr_{\CC\cap\SS}(\rho).
$$
The corollary follows by observing that
$$
	\dr_{\CC\cap\SS}(\rho)
	=
	\sum_{v\in\CC\cap\SS} d_v\cdot\dr_v(\rho)
	=
	\deg(\gcd(\Q,\s))\cdot 1
	+
	\dr_\infty(\rho)
$$
and that $\dr_\infty(\rho)=2g$.
\end{proof}


\subsection{Arithmetic application}\label{sec:arithmetic-application}

In this section we show how to apply our main theorem to the example given above.

The Euler factor at $v=\infty$ of the $L$-function of $J$ is trivial since $\dr_\infty(\rho)=\dim(V)$, and thus the complete $L$-function satisfies
$$
	L(T,J/K)
	=
	\prod_{\pi\in\AA}
	L(T^{\deg(\pi)},J/\Fpi)^{-1}
	=
	\prod_{v\in\PP}
	L(T^{d_v},\rho_v)^{-1}
	=
	L_{\{\infty\}}(T,\rho).
$$
Similarly, for the partial $L$-function of $\rho$, we have
$$
	\LC(T,\rho)
	=
	\prod_{v\in\PP\ssm\CC}
	L(T^{d_v},\rho_v)^{-1}
	=
	\prod_{\substack{\pi\in\AA\\\pi\nmid\Q}}
	L(T^{\deg(\pi)},J/\Fpi)^{-1}.
$$

For each $\pi\in\AA$, the Euler factor $L(T,J/\Fpi)^{-1}$ is the reciprocal of a polynomial with coefficients in $\bbZ$ so satisfies
$$
	T\frac{d}{dT}\log(L(T,J/\Fpi))
	=
	\sum_{n=1}^\infty a_{\pi,n}T^n
$$
for integers $a_{\pi,n}\in\bbZ$.

The complete $L$-function is also a polynomial with coefficients in $\bbZ$, and it satisfies
$$
	T\frac{d}{dT}\log(L(T,J/K))
	=
	T\frac{d}{dT}\log(L_{\{\infty\}}(T,\rho))
	=
	\sum_{n=1}^\infty
	\left(
	\sum_{f\in\MM_n}
	\VM(f)
	\right)
	T^n
$$
where $\VM(f)\colon\MM\to\bbZ$ is the von Mangoldt function of $\rho$ defined in \eqref{eqn:von-mangoldt} by
$$
	\VM(f)
	=
	\begin{cases}
	d\cdot a_{\pi,n} & f=\pi^m\mbox{ and }\pi\in\AA_d \\
	0 & \mbox{otherwise}.
	\end{cases}
$$
Similarly, the partial $L$-function of $\rho$ is a polynomial with coefficients in $\bbZ$ and satisfies
$$
	T\frac{d}{dT}\LC(T,\rho)
	=
	\sum_{n=1}^\infty
	\left(
	\sum_{\substack{f\in\MM_n\\\gcd(f,\Q)=1}}
	\VM(f)
	\right)
	T^n.
$$

For $A$ in $\BQ=(\Fq[t]/\Q\Fq[t])^\times$ and positive integer $n$, we defined the sum $\SnAQ$ in \eqref{eqn:SnAQ} by
$$
	\SnAQ
	=
	\sum_{\substack{f\in\MM_n\\f\equiv A\bmod\Q}}
	\VM(f).
$$
We then defined the expected value and variance of this sum as $A$ varies uniformly over $\BQ$ by
$$
	\bbE_A[\SnAQ]
	=
	\frac{1}{\phi(\Q)}\sum_{A\in\BQ}\SnAQ,
	\ 
	\Var_A[\SnAQ]
	=
	\frac{1}{\phi(\Q)}\sum_{A\in\BQ} \left|\SnAQ - \bbE_A[\SnAQ]\right|^2
$$
respectively where $\phi(\Q)=|\BQ|$ (see \eqref{eqn:E-and-V}).

\begin{theorem}\label{thm:application}
Suppose that $\gcd(\Q,\s)=t$ and that $\deg(\Q)>\frac{1}{2g}(72(4g^2+1)^2+1)$.  Then
$$
	\phi(\Q)\cdot\bbE_A[\SnAQ]
	\ =
	\sum_{\substack{f\in\MM_n\\\gcd(f,\Q)=1}}
	\VM(f)
	\mbox{ and }
	\lim_{q\to\infty}
	\frac{\phi(\Q)}{q^{2n}}
	\cdot
	\Var_A[\SnAQ]
	= \min\{n,2g\cdot\deg(\Q)-1\}.
$$
\end{theorem}

\begin{proof}
This will follow from Theorem~\ref{thm:variance-estimate} once we show that all the hypotheses of that theorem are met.  Lemma~\ref{lem:rhochi-properties} implies that $\rho$ is punctually pure of weight $w=1$ and that $\PhiQAwful\rho$ is empty\footnote{There \emph{are} mixed characters, but as shown the proof of Proposition~\ref{prop:var-estimate}, they do not contribute to the main term of the variance estimate.}.  Moreover, Proposition~\ref{prop:tate-module-invariant-dimensions} implies that $V(0)$ has a unique unipotent block of dimension two and no other unipotent block of multiplicity one (since $2g-2\neq 1$), hence Theorem~\ref{thm:is-equidistributed} implies that the Mellin transform of $\rho$ has big monodromy since $\gcd(\Q,\s)=t$ and since
$$
	\deg(\Q)
	>
	\frac{1}{2g}(72((2g)^2+1)^2 - 2g - 0 + (1+2g))
	=
	\frac{1}{2g}(72(4g^2+1)^2+1).
$$
Therefore the hypotheses of Theorem~\ref{thm:variance-estimate} hold as claimed.
\end{proof}

Taking $g=1$ and $f=x(x-1)$ yields Theorem~\ref{thm:intro-theorem} from \S\ref{sec:introduction}.


\appendix


\section{Detecting a big subgroup of $\GL_R$}


\subsection{Weight multiplicity map}

Let $\iota\colon \bbQbar\to\bbC$ be a field embedding, $m$ be a positive integer, and $m=\{1,\ldots,m\}$.

\begin{defn}
A \defi{weight partition map} of an element $\alpha=(\alpha_1,\ldots,\alpha_m)$ in $(\bbQbar^\times)^m$ is a map $w_\alpha\colon[m]\to[m]$ satisfying the following for every $i,j\in[m]$:
$$
	w_\alpha(i) = w_\alpha(j)
	\mbox{ iff }
	|\iota(\alpha_i)| = |\iota(\alpha_j)|;
	\ \ 
	|w_\alpha^{-1}(i)|\geq |w_\alpha^{-1}(j)|
	\mbox{ if }
	i\leq j.
$$
\end{defn}

\smallskip\noindent
In general, $\alpha$ may have multiple weight partition maps, but all will have the same range and yield the same map $[m]\to\bbZ$ given by $i\mapsto |w_\alpha^{-1}(i)|$.  In particular, if $w_\alpha$ is a weight partition map of $\alpha$ and if $\sigma\in\Sym(m)$, then the composed map $w_\alpha\sigma$ is also a weight partition map of $\alpha$.

\begin{defn}\label{def:weight-multiplicity-map}
The \defi{$m$th weight multiplicity map} is the map
$$
	\wm{m}\colon (\bbQbar^\times)^m\to\bbZ^m
$$
which sends an element $\alpha$ to the tuple $\l=(\l_1,\ldots,\l_m)$ satisfying $\l_i=|w_\alpha^{-1}(i)|$ for some weight partition map $w_\alpha$ and every $i\in[m]$.
\end{defn}

\begin{lemma}
Let $\alpha,\beta\in(\bbQbar^\times)^m$, and let $s\in\bbQbar^\times$ and $\sigma\in\Sym(m)$.  Suppose $\beta_i=s\alpha_{\sigma(i)}$ for every $i\in[m]$.  Then $\wm{m}(\alpha)=\wm{m}(\beta)$.
\end{lemma}

\begin{proof}
Let $w_\alpha,w_\beta$ be respective weight partition maps of $\alpha,\beta$.  Then for every $i,j\in[m]$, one has
$$
	w_\beta(i) = w_\beta(j)
	\iff
	|\iota(\beta_i)| = |\iota(\beta_j)|
	\iff
	|\iota(\alpha_{\sigma(i)})| = |\iota(\alpha_{\sigma(j)})|
	\iff
	w_\alpha\sigma(i) = w_\alpha\sigma(j).
$$
In particular, the weight partition maps $\sigma w_\alpha,w_\beta$ of $\alpha,\beta$ respectively coincide, so $\wm{m}(\alpha)=\wm{m}(\beta)$ as claimed.
\end{proof}

\begin{defn}
For any $\l=\wm{m}(\alpha)$, let $\len(\l)=\max\{1\leq i\leq m:\l_i\neq 0\}$.
\end{defn}

\smallskip\noindent
Observe that $[\len(\l)]$ is the range of any weight partition map $w_\alpha$ of $\alpha$ and $(\l_1,\ldots,\l_{\len(\l)})$ is a partition of $m$.


\subsection{Tensor indecomposability}

Let $m,n\geq 2$ be integers, let $\alpha\in (\bbQbar^\times)^m$, $\beta\in (\bbQbar^\times)^n$, and $\gamma\in(\bbQbar^\times)^{mn}$ be elements, and let $a=\wm{m}(\alpha)$, $b=\wm{n}(\beta)$, $c=\wm{mn}(\gamma)$.

Suppose $\tau\colon[m]\times[n]\to [mn]$ is a bijection satisfying
$$
	\gamma_{\tau(i,j)} = \alpha_i\beta_j
	\mbox{ for }
	(i,j)\in[m]\times[n],
$$
and let $w_\alpha,w_\beta,w_\gamma$ be weight partition maps of $\alpha,\beta,\gamma$ respectively.

\begin{lemma}\label{lem:def-kappa}
There exists a unique map $[\len(a)]\times[\len(b)]\to[\len(c)]$ which makes the following diagram commute:
$$
	\xymatrix{
		[m]\times [n]\ar[r]^\tau\ar[d]_{w_\alpha\times w_\beta}
			& [mn]\ar[d]^{w_\gamma} \\
		[\len(a)]\times[\len(b)]\ar[r]
			& [\len(c)].
	}
$$
\end{lemma}

\begin{proof}
To see that such a map exists observe that $w_\gamma\tau$ factors through $w_\alpha\times w_\beta$ since
\begin{eqnarray*}
	(w_\alpha\times w_\beta)(i_1,j_1)
	=
	(w_\alpha\times w_\beta)(i_2,j_2)
	& \iff &
		|\alpha_{i_1}|=|\alpha_{i_2}|\mbox{ and }|\beta_{j_1}|=|\beta_{j_2}| \\
	& \Longrightarrow &
		|\alpha_{i_1}\beta_{j_1}| = |\alpha_{i_2}\beta_{j_2}| \\
	& \iff &
		|\gamma_{\tau(i_1,j_1)}| = |\gamma_{\tau(i_2,j_2)}| \\
	& \iff &
		w_\gamma\tau(i_1,j_1) = w_\gamma\tau(i_2,j_2)
\end{eqnarray*}
for every $i_1,i_2\in[m]$ and $j_1,j_2\in[n]$.  To see that the map is unique, observe that the left vertical map of the diagram is surjective and that the map must satisfy $l\mapsto w_\gamma\tau(i,j)$ for any $(i,j)$ in $(w_\alpha\times w_\beta)^{-1}(l)$.
\end{proof}

Let $\kappa\colon [\len(a)]\times[\len(b)]\to[\len(c)]$ be the map of Lemma~\ref{lem:def-kappa}.

\begin{lemma}\label{lem:kappa-is-injective}
For each $l\in[\len(a)]$, the restriction of $\kappa$ to $\{l\}\times[\len(b)]$ is injective.
\end{lemma}

\begin{proof}
Recall that $[\len(a)]$ and $[\len(b)]$ are the respective ranges of $w_\alpha$ and $w_\beta$, so suppose $i\in[m]$ and $j_1,j_2\in[n]$.  Moreover, one has
\begin{eqnarray*}
	\kappa(w_\alpha(i),w_\beta(j_1))=\kappa(w_\alpha(i),w_\beta(j_2))
	& \iff &
		w_\gamma\tau(i,j_1) = w_\gamma\tau(i,j_2) \\
	& \iff &
		|\gamma_{\tau(i,j_1)}| = |\gamma_{\tau(i,j_2)}| \\
	& \iff &
		|\alpha_i\beta_{j_1}| = |\alpha_i\beta_{j_2}| \\
	& \iff &
		w_\beta(j_1)=w_\beta(j_2),
\end{eqnarray*}
and thus the restriction of $\kappa$ to $\{w_\alpha(i)\}\times[\len(b)]$ is injective as claimed.
\end{proof}

Let $r$ be a positive integer.

\begin{lemma}\label{lem:parts-bounded-by-r}\ 

\smallskip
\begin{enum}
\item\label{item:c-len-to-ab-lens} If $c_{\len(c)}\leq r$, then $a_{\len(a)}\leq r$ and $b_{\len(b)}\leq r$.

\item\label{item:a_1-to-c_len(b)}
If $a_1>r$ (resp.~$b_1>r$), then $c_{\len(b)}>r$ (resp.~$c_{\len(a)}>r$).

\end{enum}
\end{lemma}

\begin{proof}
For part \eqref{item:c-len-to-ab-lens}, we prove the contrapositive.  More precisely, if $k\in[\len(c)]$, then one has
$$
	c_k
	=
	\sum_{\kappa(i,j)=k}
	a_i b_j
	\geq
	a_{\len(a)}b_{\len(b)}
	\geq
	\max\{a_{\len(a)},b_{\len(b)}\},
$$
and thus $c_{\len(c)}>r$ if $a_{\len(a)}>r$ or $b_{\len(b)}>r$.  Thus \eqref{item:c-len-to-ab-lens} holds.

For part \eqref{item:a_1-to-c_len(b)}, we suppose, without loss of generality, that $a_1>r$ and show that $c_{\len(b)}>r$.  
We first observe that Lemma~\ref{lem:kappa-is-injective} implies the integers $\kappa(1,1),\ldots,\kappa(1,\len(b))$ are distinct.  Moreover, for each $l\in[\len(b)]$, one has
$$
	c_{\kappa(1,l)}
	\geq
	a_1b_l
	>
	r\cdot 1
	=
	r.
$$
Therefore at least $\len(b)$ integers in the monotone decreasing sequence $c_1,\ldots,c_{\len(b)}$ exceed $r$, and thus \eqref{item:a_1-to-c_len(b)} holds.
\end{proof}

The following proposition is the main result of this subsection.  We will use its contrapositive to deduce that a certain representation is tensor indecomposable whenever $mn\gg r$.

\begin{prop}\label{prop:tensor-indecomposable}
Suppose $c_{\len(c)}=1$ and $c_2\leq r$.  If $\len(c)\leq r+1$, then $m,n\leq r^2+1$ and thus $mn\leq (r^2+1)^2$.
\end{prop}

\begin{proof}
Lemma~\ref{lem:parts-bounded-by-r}.\ref{item:c-len-to-ab-lens} implies that $a_{\len(a)}=b_{\len(b)}=1$ since $c_{\len(c)}=1$.  Therefore $\len(a)\geq 2$ and $\len(b)\geq 2$ since $m\geq 2$ and $n\geq 2$ respectively, and moreover, $c_2\geq c_{\len(a)}$ or $c_2\geq c_{\len(b)}$.  Hence the contrapositive of Lemma~\ref{lem:parts-bounded-by-r}.\ref{item:a_1-to-c_len(b)} implies $a_1\leq r$ and $b_1\leq r$ since $c_2\leq r$.  In particular, if $\len(c)\leq r+1$, then Lemma~\ref{lem:kappa-is-injective} implies $\len(a),\len(b)\leq r+1$, and thus
$$
	m = \sum_{i=1}^{\len(a)} a_i\leq ra_1 + a_{\len(a)}\leq r^2+1,\quad
	n = \sum_{j=1}^{\len(b)} b_j\leq rb_1 + b_{\len(b)}\leq r^2+1
$$
as claimed.
\end{proof}


\subsection{Pairing avoidance}

Let $n$ be a positive integer and $I$ be the $n\times n$ identity matrix.  We define the orthogonal and symplectic groups of matrices by
$$
	\O_n(\bbQbar)
	=
	\left\{\,M\in\GL_n(\bbQbar) : MM^t = I\,\right\}
$$
and
$$
	\Sp_{2n}(\bbQbar)
	=
	\left\{\,M\in\GL_{2n}(\bbQbar) : MPM^t = P
		\mbox{ for }P=\left(\begin{array}{rr}0&I\\-I&0\end{array}\right)
	\,\right\}
$$
respectively.

\begin{lemma}\label{lem:eigenvalue-involution}
Suppose $m=n$ (resp.~$m=2n$) and $g\in\O_n(\bbQbar)$ (resp.~$g\in\Sp_{2n}(\bbQbar)$).  Let $\alpha\in(\bbQbar^\times)^m$ be a tuple of the eigenvalues of $g$ and $a=\wm{m}(\alpha)$.  Then some involution $\pi\in\Sym(\len(a))$ satisfies the following:

\begin{enum}
\item $a_i=a_{\pi(i)}$ for every $i\in[\len(a)]$;
\item $\pi$ has at most one fixed point.
\end{enum}
\end{lemma}

\begin{proof}
The involution $s\mapsto 1/s$ of $\bbQbar^\times$ induces a permutation of the eigenvalues of elements of $\O_n(\bbQbar)$ and $\Sp_{2n}(\bbQbar)$.  The latter is an involution $\sigma\in\Sym(m)$ with the property that, for any weight partition map $w_\alpha$ of $\alpha$ and every $i\in[m]$, one has
$$
	w_\alpha(i) = w_\alpha\sigma(i)
	\iff
	|\alpha_i| = |\alpha_{\sigma(i)}|
	\iff
	|\alpha_i| = |1/\alpha_i|
	\iff
	|\alpha_i| = 1.
$$
The involution in question is given by $w_\alpha(i)\mapsto w_\alpha\sigma(i)$ for every $i\in[m]$; recall $w_\alpha$ maps onto $[\len(a)]$.
\end{proof}

The following is the main result of this subsection.  We will use its contrapositive to show that some subgroup of $\GL_m(\bbQbar)$ fails to preserve non-degenerate pairings which are either symmetric or alternating.

\begin{prop}\label{prop:pairing-avoidance}
Suppose $m=n$ (resp.~$m=2n$) and $g\in\GL_n(\bbQbar)$.  Let $\alpha\in(\bbQbar^\times)^m$ be a tuple of the eigenvalues of $g$ and $a=\wm{m}(\alpha)$.  If there exist $i,j$ such that $a_i,a_j$ are distinct from each other and from all $a_k$ for $k\neq i,j$, then $g\not\in\O_n(\bbQbar)$ (resp.~$g\not\in\Sp_{2n}(\bbQbar)$).
\end{prop}

\begin{proof}
We prove the contrapositive.  More precisely, if $g\in\O_n(\bbQbar)$ (resp.~$g\in\Sp_{2n}(\bbQbar)$) and if $\pi\in\Sym(\len(a))$ is an involution satisfying the properties of Lemma~\ref{lem:eigenvalue-involution}, then $\pi(i)=i$ for at most one $i$.  Therefore, for all but at most one $i$ and for $j=\pi(i)$, one has $i\neq j$ and $a_i=a_j$.  In particular, there is at most one $i$ such that $a_i\neq a_j$ for $j\neq i$.
\end{proof}


\subsection{Main theorem}

In this section we state and prove the main result of this appendix.

\begin{theorem}\label{thm:big-monodromy}
Let $r,R$ be positive integers and $G$ be a connected reductive subgroup of $\GL_R(\Qellbar)$.  Let $g\in G$ be an element and $\gamma\in(\Qellbar^\times)^R$ be an eigenvector tuple of $g$.  Suppose that $G$ is irreducible, that $\gamma$ lies in $(\bbQbar^\times)^R$, and that $c=\wm{R}(\gamma)$ satisfies $\len(c)\leq r+1$ and $1=c_{\len(c)}<c_{\len(c)-1}$ and $c_2\leq r$.  If $R>72(r^2+1)^2$, then either $G=\SL_R(\Qellbar)$ or $G=\GL_R(\Qellbar)$.
\end{theorem}

\noindent
The proof will occupy the remainder of this subsection.

Since $G$ is algebraic, it contains the semisimplification of $g$, an element for which $\gamma$ is also an eigenvector.  Hence we replace $g$ by its semisimplification and suppose without loss of generality that $g$ is semisimple.  We also replace $G$ and $g$ by the conjugates $h^{-1}Gh$ and $h^{-1}gh$ by a suitable element $h\in\GL_R(\Qellbar)$ so that we may suppose without loss of generality that $g$ is the diagonal matrix $\mathrm{diag}(\gamma_1,\ldots,\gamma_R)$.

Let $V=\Qellbar^R$ and $f$ be the diagonal matrix
$$
	f = \mathrm{diag}(|\iota(\gamma_1)|,\ldots,|\iota(\gamma_m)|).
$$
We claim we may regard $f$ as an element of $\GL_R(\Qellbar)$.  More precisely, it is an element of $\GL_R(\iota(\bbQbar))\sub\GL_R(\bbC)$ since $|\iota(\gamma_i)|^2=\iota(\gamma_i)\overline{\iota(\gamma_i)}$ lies in the algebraically closed subfield $\iota(\bbQbar)\sub\bbC$ and thus so does $|\iota(\gamma_i)|$.  Replacing $G$, $g$, $f$ by conjugates by a suitable common permutation matrix, we suppose without loss of generality that $|\iota(\gamma_1)|$ is an eigenvalue of $f$ of multiplicity $c_1$.

\begin{lemma}\label{lem:rank-bound}
$f$ is a semisimple element of $G$ such that $f-|\iota(\gamma_1)|\in\End(V)$ has rank at most $r^2$.
\end{lemma}

\begin{proof}
For some sequence $e_1,\ldots,e_n$ of tuples $e_i=(e_{i,1},\ldots,e_{i,m})\in\bbZ^m$, the intersection of $G$ with the subgroup of diagonal matrices in $\GL_R(\Qellbar)$ consists of all matrices $\diag(\alpha_1,\ldots,\alpha_m)$ satisfying
$$
	\prod_{i=1}^m\alpha_i^{e_{1,i}}
	=
	\prod_{i=1}^m\alpha_i^{e_{2,i}}
	=
	\cdots
	=
	\prod_{i=1}^m\alpha_i^{e_{n,i}}
	=
	1.
$$
By hypothesis, $g$ lies in this intersection, and thus
$$
	|\iota(\prod_{i=1}^m\gamma_i^{e_{1,i}})|
	=
	|\iota(\prod_{i=1}^m\gamma_i^{e_{2,i}})|
	=
	\cdots
	=
	|\iota(\prod_{i=1}^m\gamma_i^{e_{n,i}})|
	=
	|\iota(1)|
$$
or equivalently
$$
	\prod_{i=1}^m|\iota(\gamma_i)|^{e_{1,i}}
	=
	\prod_{i=1}^m|\iota(\gamma_i)|^{e_{2,i}}
	=
	\cdots
	=
	\prod_{i=1}^m|\iota(\gamma_i)|^{e_{n,i}}
	=
	1.
$$
Therefore $f$ is a diagonal (hence semisimple) element of $G$ as claimed.  It remains to show $f-|\iota(\gamma_1)|\in\End(V)$ has rank at most $r^2$.  Indeed, exactly $c_1$ of its eigenvalues equal $|\iota(\gamma_1)|$, hence the rank of $f-|\iota(\gamma_1)|$ is
$$
	R - c_1 \leq \sum_{i=2}^{\len(c)} c_i \leq r\cdot r = r^2
$$
by our hypotheses on $c$.
\end{proof}

Let $[G,G]$ be the derived (i.e., commutator) subgroup of $G$.  Observe that $G$ acts irreducibly on $V=\Qellbar^R$ by hypothesis, so its center $Z(G)$ consists entirely of scalars and $G$ is an almost product of $[G,G]$ and $Z(G)$.  In particular, $[G,G]$ is a connected semisimple group which also acts irreducibly on $V$, and for some $a\in\Qellbar^\times$, the scalar multiple $af$ lies in $[G,G]$.

Let $\mathfrak{g}\seq\gl_R=\End(V)$ be the Lie algebra of $[G,G]$.  It is a semisimple irreducible Lie subalgebra of $\gl_R$ since $[G,G]$ is semisimple and acts irreducibly on $V$.  It also contains $af$, and Lemma~\ref{lem:rank-bound} implies that $\dim((af-a|\iota(\gamma_1)|)V)\leq r^2$.  Finally, the contrapositive of Proposition~\ref{prop:tensor-indecomposable} implies that $\mathfrak{g}$ is simple since otherwise $V$ would be tensor decomposable as a representation of $G$.  Therefore, a result of Zarhin \cite[Th.~6]{Zarhin} implies that $\mathfrak{g}$ is one of $\mathfrak{sl}(V)$, $\mathfrak{so}(V)$, or $\mathfrak{sp}(V)$ since
$$
	R = \dim(V)
	> 72(r^2)^2
	\geq 72\dim((f-|\iota(\gamma_1)|)V)^2
	= 72\dim((af-a|\iota(\gamma_1)|)V)^2
$$
by our hypotheses on $R$.

To complete the proof of the theorem it suffices to rule out $\mathfrak{g}=\mathfrak{so}(V)$ and $\mathfrak{g}=\mathfrak{sp}(V)$ or equivalently to show that $G$ preserves neither an orthogonal nor a symplectic pairing.  However, our hypotheses on $c$ together with the contrapositive of Proposition~\ref{prop:pairing-avoidance} implies that $G$ preserves neither such type of pairing, so $\mathfrak{g}=\mathfrak{sl}(V)$ as claimed.  That is, $[G,G]$ is $\SL(V)$ and $G$ is equal to one of $\SL(V)$ or $\GL(V)$.


\section{Perverse Sheaves and the Tannakian Monodromy Group}\label{sec:tannakian-appendix}


\subsection{Category of perverse sheaves}

Given a smooth curve $X$ over a perfect field $\bbF$, we can speak of the so-called derived category $\Dbc{X}$.  Its objects $M$ are complexes of constructible $\Qellbar$-sheaves on $X$ over $\bbF$ whose cohomology complex
$$
	\cdots
	\longto \H{-1}M
	\longto \H{0}M
	\longto \H{1}M
	\longto \cdots
$$
is bounded and whose cohomology sheaves $\H{i}M$ are all constructible.  There is a well-defined dual object $\Dual{M}$, the Verdier dual of $M$.  Moreover, for each $n\in\bbZ$, there is a well-defined shifted complex $M[n]$ which satisfies $\H{i}{M[n]}=\H{i+n}M$.

We say that $M$ is \defi{semi-perverse} iff $\H{0}{M}$ is punctual and $\H{i}{M}$ vanishes for $i>0$ and that $M$ is \defi{perverse} iff $M$ and $\Dual{M}$ are semi-perverse.  We write $\Perv{X}$ for the full subcategory of perverse objects in $\Dbc{X}$.  It is an abelian category thus one can speak of subquotients of its objects as well as kernels and cokernels of its morphisms.  It is common to call its objects perverse sheaves despite the fact that they are \emph{complexes} of sheaves.

There is a natural functor from the category of constructible $\El$-sheaves on $X$ over $k$ to $\Dbc{X}$: it sends a sheaf $\FF$ to a complex concentrated at $i=0$ and takes a morphism to the unique extension to a morphism of complexes.  The image of this functor is not stable under duality though: if $\FF^\vee$ is the dual of $\FF$, then $\Dual{\FF}$ is isomorphic to $\FF^\vee(1)[2]$.  If instead one sends sends each $\FF$ to $\FF(1/2)[1]$, then self-dual objects are taken to self-dual objects and middle-extension sheaves are taken to perverse sheaves.


\subsection{Purity}

Let $X$ be a smooth curve over $\Fq$.  We say an object $M$ in $\Dbc{X}$ is \defi{$\iota$-mixed of weights $\leq w$} iff $\H{i}M$ is punctually $\iota$-mixed of weights $\leq w+i$ for every $i$, and then $M[n]$ is $\iota$-mixed of weights $w+n$.  We also say $M$ is \defi{$\iota$-pure of weight $w$} iff $M$ is $\iota$-mixed of weights $\leq w$ and $\Dual{M}$ is $\iota$-mixed of weights $\leq -w$, and then $M[n]$ is $\iota$-pure of weight $w+n$.  Finally, we say $M$ is \defi{pure of weight $w$} iff it is $\iota$-pure of weight $w$ for every field embedding $\iota\colon\Qbar\to\bbC$.


\subsection{Subobjects and subquotients}

Let $(\CC,\oplus)$ be an abelian category, let $\zero$ be its zero object, and let $M,N$ be a pair of objects in $\CC$.

We say that $N$ is a \defi{subobject} of $M$ and write $N\seq M$ iff there is a monomorphism $N\into M$ in $\CC$.  More generally, we say $N$ of $M$ is a \defi{subquotient} of $M$ iff there exist an object $S$, a monomorphism $S\into M$, and an epimorphism $S\onto N$ all in $\CC$.  Equivalently, $N$ is a subquotient of $M$ iff there exist an object $Q$, an epimorphism $M\onto Q$, and a monomorphism $N\into Q$ all in $\CC$.

\begin{prop}\label{prop:subquotient-of-pure-is-pure}
If $M\in\Perv{\Gm}$ is $\iota$-pure of weight $w$, then so is every subquotient $N$.
\end{prop}

\begin{proof}
See \cite[5.3.1]{BBD}.
\end{proof}

Given a pair $N_1,N_2\seq M$ of subobjects, we write $N_1\seq N_2\seq M$ iff $N_1\seq N_2$ and, for the corresponding monomorphisms, $N_1\into M$ equals the composition $N_1\into N_2\into M$.  We also write $N_1=N_2\seq M$ iff $N_1\seq N_2\seq M$ and $N_2\seq N_1\seq M$.  For example, if $M$ is an object in $\Perv{\Gm}$ and if $\phi$ is the Frobenius automorphism of $\bar{M}$, then the subobjects $N\seq M$ give rise to precisely those subobjects $\bar{N}\seq\bar{M}$  satisfying $\bar{N}=\phi(\bar{N})\seq\bar{M}$.


\subsection{Kummer sheaves}

Let $\Gm=\Poneu\ssm\{0,\infty\}$ over $\Fq$, and let $\piOneTame{\Gm}$ be the tame \'etale fundamental group, that is, the maximal quotient of $\piOne{\Gm}$ whose kernel contains the $p$-Sylow subgroups of $I(0)$ and $I(\infty)$.  It lies in an exact sequence
$$
	1\to \piOneTame{\GmBar}\to \piOneTame{\Gm}\to \Gal(\Fqbar/\Fq)\to 1
$$
where $\piOneTame{\GmBar}$ is the image of $\piOne{\GmBar}$ via the tame quotient $\piOne{\Gm}\onto\piOneTame{\Gm}$.

We say a constructible sheaf on $\PoneBar$ is a \defi{Kummer sheaf} iff it is a middle-extension sheaf which is lisse of rank one on $\GmBar$ and for which the corresponding representation factors through the  quotient $\piOne{\GmBar}\onto\piOneTame{\GmBar}$.  Equivalently, the Kummer sheaves are the middle-extension sheaves $\LL_\rho$ on $\PoneBar$ associated to a continuous character $\rho\colon\piOneTame{\GmBar}\to\Qellbar^\times$.


\subsection{Middle convolution on $\PP$}\label{sec:cateogory-P}

Let $\pi\colon\Gm\times\Gm\to\Gm$ be the multiplication map on $\Gm$ over $\Fq$.  Using it one can define two additive bifunctors on $\Dbc{\GmBar}$ corresponding to two flavors of multiplicative convolution:
$$
	M\star_! N := R\pi_!(M\boxtimes N),\quad
	M\star_* N := R\pi_*(M\boxtimes N).
$$
There is a canonical map $M\star_! N\to M\star_* N$, but it need not be an isomorphism in general.  However, if both convolution objects lie in $\Perv{\GmBar}$, then one can speak of the image of the map and define
$$
	M\midstar N := \mathrm{Image}(M\star_! N\to M\star_* N).
$$
This observation led Katz to define the full subcategory $\PP$ of $\Perv{\GmBar}$ whose objects are all $M$ for which $N\mapsto M\star_!N$ and $N\mapsto M\star_*N$ take perverse sheaves to perverse sheaves (see \cite[\S2.6]{Katz:RLS} and \cite[Ch.~2]{Katz:CE}).  Among other things, it includes perverse sheaves $\FF[1]$ for $\FF$ a simple middle-extension sheaf on $\GmBar$ of generic rank at least two.  Moreover, it is an additive category with respect to the usual direct sum of sheaves.  Katz called the resulting additive bifunctor on $\PP$ middle convolution.  


\subsection{The category $\PP_\arith$}\label{sec:category-P_arith}

Let $\Dbc{\Gm}\to\Dbc{\GmBar}$ be the ``extension of scalars'' functor which sends an object of $M$ over $\Fq$ to the object $\bar{M}=M\times_{\Fq}\Fqbar$.  It maps objects of $\Perv{\Gm}$ to objects of $\Perv{\GmBar}$, and we define $\PP_\arith$ to be the full subcategory of $\Perv{\Gm}$ whose objects $M$ are those for which $\bar{M}$ lies in $\PP$.  Among other things, $\PP_\arith$ contains perverse sheaves $\FF[1]$ for $\FF$ a geometrically simple middle-extension sheaf on $\Gm$ over $\Fq$ which is of generic rank at least two.

Once again we have the two flavors of multiplicative convolution
$$
	M\star_! N := R\pi_!(M\boxtimes N),\quad
	M\star_* N := R\pi_*(M\boxtimes N).
$$
for any pair of objects $M,N$ in $\Perv{\Gm}$.  We can also define middle convolution on $\PP_\arith$ as before
$$
	M\midstar N := \mathrm{Image}(M\star_! N\to M\star_* N).
$$
for any pair of objects $M,N$ in $\PP_\arith$.

\begin{prop}\label{prop:convolution-of-pures-is-pure}
If $M$ and $N$ are $\iota$-pure of weights $m$ and $n$ respectively, then $M\midstar N$ is $\iota$-pure of weight $m+n$.
\end{prop}

\begin{proof}
Our argument is essentially that of \cite[Ch.~4]{Katz:CE}.  On one hand, $M\boxtimes N$ is $\iota$-pure of weight $m+n$ on $\Gm\times\Gm$, hence \cite[3.3.1]{Deligne:WeilII} and Proposition~\ref{prop:subquotient-of-pure-is-pure} imply $M\star_! N$ and its perverse quotient $M\midstar N$ are $\iota$-mixed of weight $m+n$.  On the other hand, $DM$ and $DN$ are $\iota$-pure of weights $m$ and $n$ respectively, and
\begin{eqnarray*}
	D(M\midstar N)
	& = &
	\mathrm{Image}(D(M\star_*N)\to D(M\star_! N)) \\
	& = &
	\mathrm{Image}(DM\star_! DN\to DM\star_* DN)
	\ \ =\ \ 
	DM\midstar DN
\end{eqnarray*}
hence $D(M\midstar N)$ is $\iota$-mixed weights $\leq m+n$ (cf.~\cite[6.2]{Deligne:WeilII}).  Thus $M\midstar N$ is $\iota$-pure of weight $m+n$ as claimed.
\end{proof}


\subsection{The category $\Tann{\GmBar}$}\label{sec:tann-gmbar}

Gabber and Loeser defined an object $M$ in $\Perv{\GmBar}$ to be \defi{negligible} iff its Euler characteristic $\chi(\GmBar,M)$ vanishes (see \cite[pg.~529]{GL}), or equivalently, it is isomorphic to a successive extension of shifted Kummer sheaves $\LL_\rho[1]$ (cf.~\cite[3.5.3]{GL}).  They showed that the full subcategory $\Negl{\GmBar}$ of $\Perv{\GmBar}$ whose objects are the negligible sheaves is a thick subcategory of the abelian category (see \cite[3.5.2]{GL}), and thus one can speak of the quotient category
$$
	\Tann{\GmBar}:=\Perv{\GmBar}/\Negl{\GmBar}.
$$
They then proceeded to show that $\Tann{\GmBar}$ is a neutral Tannakian category (see \cite[3.7.5]{GL} and \cite[II.2.19]{DM}).

\begin{theorem}\label{thm:P-is-neutral-Tannakian}
The composite map $\PP\to\Perv{\GmBar}\to\Tann{\GmBar}$ induces an equivalence of categories such that:
\begin{enum}
\item middle convolution on $\PP$ induces a tensor product $\otimes$ on $\Tann{\GmBar}$;
\item the unit object $\one$ corresponds to the skyscraper sheaf $i_*\Qellbar$ for $i\colon\{1\}\to\GmBar$ the inclusion;

\item the dual $M^\vee$ of an object $M$ is the object $[x\mapsto 1/x]^*DM$;

\item the dimension $\dim(M)$ of an object $M$ is $\chi(\GmBar,M)$;

\item\label{thm:item:fiber-functor} a fiber functor is $M\mapsto H^0(\AoneuBar,j_{0!}M)$ for $j_0\colon\Gm\to\Aoneu$ the inclusion.

\end{enum}
\end{theorem}

\medskip\noindent
See \cite[3.7.2]{GL} and \cite[Ch.~2 and Ch.~3]{Katz:CE}.


\subsection{The category $\Tann{\Gm}$}

Let $\Negl{\Gm}$ be the full subcategory of $\Perv{\Gm}$ whose objects $M$ are those for which $\bar{M}$ lies in $\Negl{\GmBar}$, and let
$$
	\Tann{\Gm} := \Perv{\Gm}/\Negl{\Gm}.
$$
Like $\Tann{\GmBar}$, the quotient category is an abelian category and even a neutral Tannakian category with tensor product $\otimes$ given by middle convolution.  Moreover, the ``extension of scalars'' functor induces a functor
$$
	\Tann{\Gm}\to\Tann{\GmBar}
$$
which also call the ``extension of scalars'' functor.

\begin{prop}
Suppose $M,N\in\Tann{\Gm}$ are $\iota$-pure of weights $m$ and $n$ respectively.  Then $M^\vee$, $N^\vee$, and $M\otimes N$ are $\iota$-pure of weights $m$, $n$, and $m+n$ respectively.
\end{prop}

\begin{proof}
The Verdier duals $DM$ and $DN$ are $\iota$-pure of weights $m$ and $n$ respectively, hence so are the Tannakian duals $M^\vee=[x\mapsto 1/x]^*DM$ and $N^\vee=[x\mapsto 1/x]^*DN$.  Moreover, Proposition~\ref{prop:convolution-of-pures-is-pure} implies that $M\otimes N=M\midstar N$ is $\iota$-pure of weight $m+n$.
\end{proof}


\subsection{Semisimple abelian categories}

We say that $M$ is \defi{simple} iff the only subobjects $N\seq M$ in $\CC$ are isomorphic to $\zero$ or $M$.  More generally, we say that $M$ is \defi{semisimple} iff it is isomorphic to a finite direct sum $N_1\oplus\cdots\oplus N_m$ of simple subobjects $N_1,\ldots,N_m\seq M$.  We say that $\CC$ is \defi{semisimple} iff each of its objects is semisimple.

\begin{prop}\label{prop:weight-zero-implies-semisimple}
If $M\in\Tann{\Gm}$ is $\iota$-pure of weight zero, then $\gp{\bar{M}}$ is semisimple.
\end{prop}

\begin{proof}
If $N_1,N_2\in\Tann{\Gm}$ are $\iota$-pure of weight zero, then so is $N_1\oplus N_2$. Therefore Proposition~\ref{prop:convolution-of-pures-is-pure}
implies that $T^{a,b}(M)$ is pure of weight zero, for every $a,b\geq 0$, and \cite[5.3.8]{BBD} implies that $T^{a,b}(\bar{M})$ is semisimple.
\end{proof}


\subsection{Tannakian monodromy group}

Let $k$ be an algebraically closed field of characteristic zero and $\Vec_k$ be the category of finite-dimensional vector spaces over $k$.  It is well known that the latter yields a rigid abelian tensor category $(\Vec_k,\otimes)$ with respect to the usual operators $\oplus$ and $\otimes$ of vector spaces and with unit object $\one=k$.

Let $(\CC,\otimes)$ be a neutral Tannakian category over $k$.  Thus $(\CC,\otimes)$ is a rigid abelian tensor category whose unit object $\one$ satisfies  $k=\End(\one)$ and for which there exists a fiber functor $\omega$, that is, an exact faithful $k$-linear tensor functor $\omega\colon\CC\to\Vec_k$.  For example, $\Vec_k$ is a neutral Tannakian category and the identity functor $\Vec_k\to\Vec_k$ is a fiber functor.  More generall, given an affine group scheme $G$ over $k$, the category $\Rep_k(G)$ of linear representations of $G$ on finite-dimensional $k$-vector spaces yields a neutral Tannakian category $(\Rep_k(G),\otimes)$, and the forgetful functor $\Rep_k(G)\to\Vec_k$ is a fiber functor.

Given an object $M$ of $\CC$, its dual $M^\vee$, and non-negative integers $a,b$, let
$$
	T^{a,b}(M) := M^{\otimes a}\oplus (M^\vee)^{\otimes b}
$$
and let $\gp{M}$ be the full tensor subcategory of $\CC$ whose objects consist of all subobjects of $T^{a,b}(M)$ for all $a,b\geq 0$.  For each automorphism $\gamma\in\Aut_\CC(M)$, let $\gamma^\vee\in\Aut_\CC(M^\vee)$ be the corresponding dual automorphism and $T^{a,b}(\gamma)\in\Aut_\CC(T^{a,b}(M))$ be the induced automorphism.

Let $\Alg_k$ be the category of $k$-algebras and $\Set$ be the category of sets.  Given a pair $\omega_1,\omega_2$ of fiber functors $\CC\to\Vec_k$ and an object $M$ in $\CC$, one can define a functor
$$
	\ulIsom^\otimes(\omega_1|M,\omega_2|M)\colon\Alg_k\to\Set
$$
by sending a $k$-algebra $R$ to the set
$$
	\{\,
		\gamma\in\Isom_R(\omega_1(M)_R,\omega_2(M)_R)
		:
		T^{a,b}(\gamma)(\omega_1(N))\seq\omega_2(N)
		\mbox{ for all }
		a,b\geq 0
		\mbox{ and }
		N\seq T^{a,b}(M)
	\,\}
$$
where $\omega_i(M)_R=\omega_i(M)\otimes_k R$ and
$$
	\Isom_R(\omega_1(M)_R,\omega_2(M)_R)
	=
	\{\,
		\gamma\in\Hom_R(\omega_1(M)_R,\omega_2(M)_R)
		:
		\gamma\mbox{ is invertible }
	\,\}.
$$
Similarly, given a single fiber functor $\omega\colon\CC\to\Vec_k$ and object $M$ in $\CC$, one can define a functor
$$
	\ulAut^\otimes(\omega|M)\colon\Alg_k\to\Set
$$
as the functor $\ulIsom^\otimes(\omega|M,\omega|M)$.

\begin{theorem}\label{thm:tannakian-monodromy-group}
Let $\omega_1,\omega_2$ be fiber functors $\CC\to\Vec_k$ and $M$ be an object of $\CC$.

\smallskip
\begin{enum}
\item $\ulAut^\otimes(\omega_i|M)$ is representable by an algebraic group scheme $G_{\omega_i|M}$ over $k$;
\item\label{thm:item:tmg-reductive} if $\gp{M}$ is semisimple, then $G_{\omega_i|M}$ is reductive;
\item\label{thm:item:tmg-torsor} $\ulIsom^\otimes(\omega_1|M,\omega_2|M)$ is represented by an affine scheme over $k$ which is a $G_{\omega_1|M}$-torsor;
\end{enum}
\end{theorem}

\noindent
See \cite[II.2.11, II.2.20, II.2.28, and II.3.2]{DM}.

We call the group scheme $G_{\omega_i|M}$ in the theorem the \defi{Tannakian monodromy group} of $\gp{M}$ with respect to $\omega_i$.

\begin{theorem}\label{thm:pure-of-weight-zero-begets-reductive}
Let $\omega\colon\Perv{\GmBar}\to\Vec_k$ be a fiber functor over $\Fqbar$ and $M\in\Perv{\Gm}$.  If $M$ is pure of weight zero, then $G_{\omega|\bar{M}}$ is reductive.
\end{theorem}

\begin{proof}
This follows from Proposition~\ref{prop:weight-zero-implies-semisimple} and Theorem~\ref{thm:tannakian-monodromy-group}.\ref{thm:item:tmg-reductive}.
\end{proof}


\subsection{Geometric versus arithmetic monodromy}

For every object $M$ in $\Tann{\Gm}$ and all integers $a,b\geq 0$, the ``extension of scalars'' functor sends a subobject $N\seq T^{a,b}(M)$ to a subobject $\bar{N}\seq T^{a,b}(\bar{M})$.  Moreover, composing the functor with a fiber functor $\omega$ on $\Tann{\GmBar}$ yields a fiber a fiber functor on $\Tann{\Gm}$ which we also denote $\omega$.  Thus there is a natural transformation
$$
	\ulAut^\otimes(\omega|\bar{M})
	\to
	\ulAut^\otimes(\omega|M)
$$
and a corresponding monomorphism of Tannakian monodromy groups
$$
	G_{\omega|\bar{M}}
	\to
	G_{\omega|M}.
$$
We call $G_{\omega|\bar{M}}$ and $G_{\omega|M}$ 
the \defi{geometric} and \defi{arithmetic Tannakian monodromy groups} of $M$ with respect to $\omega$ respectively.

\begin{prop}\label{prop:arith-is-reductive}
Suppose $M$ is in $\Tann{\Gm/\Fq}$ and is pure of weight zero.
\begin{enum}
\item $G_{\omega|\bar{M}}$ is a normal subgroup of $G_{\omega|M}$
\item If $M$ is arithmetically semisimple, then $G_{\omega|M}/G_{\omega|\bar{M}}$ is a torus, and thus $G_{\omega|M}$ is reductive.
\end{enum}
\end{prop}

\begin{proof}
Proposition~\ref{prop:weight-zero-implies-semisimple} implies that $\bar{M}$ is semisimple, so part (1) follows from \cite[Th.~6.1]{Katz:CE}.  Therefore we can speak of the quotient $G_{\omega|M}/G_{\omega|\bar{M}}$, and \cite[Lem.~7.1]{Katz:CE} implies it is a quotient of $M$ is arithmetically semisimple.  Moreover, Proposition~\ref{thm:pure-of-weight-zero-begets-reductive} implies that $G_{\omega|\bar{M}}$ is reductive, so part (2) follows by observing that the extension of a torus by a reductive group is reductive.
\end{proof}


\subsection{Frobenius element}

Let $\omega$ be a fiber functor $\Tann{\GmBar}\to\Vec_k$, let $\EFq/\Fq$ be a finite extension, and let $M$ be in $\Tann{\Gm/\EFq}$.  The geometric Frobenius element of $\Gal(\Fqbar/\EFq)$ induces a well-defined automorphism $\phi_E$ of $\bar{M}$.  By applying $\omega$, one obtains a well-defined $k$-linear automorphism of $\omega(\bar{M})$, that is, an element of $\GL(\omega(\bar{M}))=\GL(\omega(M))$.  It is even an element of $G_{\omega|M}$ since, for every $N\seq T^{a,b}(M)$ and $a,b\geq 0$, one has
$$
	\bar{N}=T^{a,b}(\phi_\EFq)(\bar{N})\seq T^{a,b}(\bar{M})
$$
and thus
$$
	\omega(\bar{N})=T^{a,b}(\phi_\EFq)(\omega(\bar{N}))\seq \omega(T^{a,b}(\bar{M}))=T^{a,b}(\omega(M)).
$$
We call $\omega(\phi_\EFq)$ the \defi{geometric Frobenius element} of $G_{\omega|M}$.  


\subsection{Frobenius conjugacy classes}

Let $\omega_1,\omega_2$ be fiber functors $\Tann{\GmBar}\to\Vec_k$, let $M$ be an element of $\Tann{\Gm}$, and let $\pi$ be an element of $\ulIsom^\otimes(\omega_1|M,\omega_2|M)(k)$.  Then Theorem~\ref{thm:tannakian-monodromy-group}.\ref{thm:item:tmg-torsor} implies that the map $g\mapsto\pi g$ induces a bijection
$$
	G_{\omega_1|M}\to \ulIsom^\otimes(\omega_1|M,\omega_2|M).
$$
Moreover, the map $g_2\mapsto g_2^\pi=\pi^{-1}g_2\pi$ induces an isomorphism $G_{\omega_2|M}\to G_{\omega_1|M}$.  While the map is not canonical (since $\pi$ is not), the conjugacy class
$$
	\Frob_{\omega_2|M}
	=
	\{\,
		\omega_2(\phi)^{\pi g_1}
		:
		g_1\in G_{\omega_1|M}(k)
	\,\}
	\sub
	G_{\omega_1|M}(k)
$$
is well defined.  We call it the \defi{geometric Frobenius conjugacy class} of $\omega_2|M$ in $G_{\omega_1|M}$.

For each finite extension $\EFq/\Fq$ and each character $\rho\in\PhiEOf\EFq{u}$, let $\LL_\rho$ be the corresponding Kummer sheaf on $\Gm$ over $E$ and $\omega_\rho\colon\Tann{\GmBar}\to\Vec_k$ be the functor given by
$$
	M\mapsto H^0(\AoneuBar,j_{0!}(M\otimes\LL_\rho)).
$$
It is a fiber functor by \cite[3.2]{Katz:CE}, and $\omega_\one$ is the fiber functor of Theorem~\ref{thm:P-is-neutral-Tannakian}.\ref{thm:item:fiber-functor}.  We write
$$
	\Frob_{\EFq,\rho}\sub G_{\omega_\one|M}
$$
for the corresponding geometric Frobenius conjugacy class of $\omega_\rho|M_{\EFq}$ where $M_{\EFq}=M\times_{\Fq}\EFq$.

Let $m=\dim(\omega_\rho(M))$ and $n\in\{0,1,\ldots,m\}$.  We say that $\omega_\rho(M)$ is \defi{mixed of weights $w_1,\ldots,w_m$} iff there exists an eigenvector tuple $\alpha=(\alpha_1,\ldots,\alpha_m)\in(\Qellbar^\times)^m$ of any element of $\Frob_{\EFq,\rho}$ such that $\alpha\in(\Qbar^\times)^m$ and such that
$$
	|\iota(\alpha_i)|^2 = (1/|\EFq|)^{w_i}\mbox{ for }1\leq i\leq m
$$
for every field embedding $\iota\colon\Qbar\to\bbC$.  We also say that $\omega_\rho(M)$ is \defi{mixed of non-zero weights $w_1,\ldots,w_n$} iff it is mixed of weights $w_1,\ldots,w_m$ with $w_{n+1}=\cdots=w_m=0$.


\subsection{Monodromy for pure middle-extension sheaves}\label{sec:representation-monodromy-groups}

Let $U\seq\Gm$ be a dense Zariski open subset over $\Fq$.  Let $\theta\colon\piOne{U}\to\GL(W)$ be a continuous representation to a finite-dimensional $\Qellbar$-vector space $W$ and $\FF=\ME{\theta}$ be the associated middle-extension sheaf on $\Gm$.  Suppose that $\theta$ is punctually pure of weight $w$ so that $M=\FF((1+w)/2)[1]$ is pure of weight zero.  Suppose moreover that $\theta$ is geometrically simple and that it does not factor through the composed quotient $\piOne{U}\onto\piOne{\Gm}\onto\piOneTame{\Gm}$ so that $M$ lies in $\PP_\arith$.

Let $\PhiU$ be the dual of $\Bu=(\Fq[u]/u\Fq[u])^\times$ (cf.~\S\ref{sec:one-parameter-families}).  We define the \defi{geometric} and \defi{arithmetic Tannakian monodromy groups} of (the Mellin transformation of) $\theta$ to be
$$
	\GG_\geom(\theta,\PhiU):=G_{\omega_\one|\bar{M}},\quad
	\GG_\arith(\theta,\PhiU):=G_{\omega_\one|M}.
$$
For $u=0,\infty$, let $W(u)$ denote $W$ regarded as an $I(u)$-module, and let $W(u)^\unip$ be the maximal submodule of $W(u)$ where $I(u)$ acts unipotently.  Moreover, let $e_{u,1},\ldots,e_{u,d_u}$ be positive integers integers satisfying
$$
	W(u)^\unip\simeq U(e_{u,1})\oplus\cdots\oplus U(e_{u,d_u})
$$
as $I(u)$-modules where $U(e)$ denotes the irreducible $e$-dimensional $I(u)$-module on which $I(u)$ acts unipotently.

\begin{prop}\label{prop:middle-extension-monodromy}\ 

\medskip
\begin{enum}
\item\label{item:prop:mem-reductive-and-normal} The groups $\GG_\geom(\theta,\PhiU)$ and $\GG_\arith(\theta,\PhiU)$ are reductive, and there is an exact sequence
$$
	1
	\to \GG_\geom(\theta,\PhiU)
	\to \GG_\arith(\theta,\PhiU)
	\to T
	\to 1
$$
for some torus $T$ over $\Qellbar$.

\item\label{item:prop:mem-nonzero-weights} For each finite extension $\EFq/\Fq$ and each $\alpha\in\PhiEOf\EFq{u}$, the fiber $\omega_\rho(M)$ is mixed of non-zero weights $-e_{0,1},\ldots,-e_{0,d_0},e_{\infty,1},\ldots,e_{\infty,d_\infty}$.

\end{enum}
\end{prop}

\begin{proof}
Part (1) follows from Proposition~\ref{prop:arith-is-reductive}, and part (2) follows from \cite[Th.~16.1]{Katz:CE}.
\end{proof}


\bibliographystyle{amsplain.bst}
\bibliography{main}

\providecommand{\bysame}{\leavevmode\hbox to3em{\hrulefill}\thinspace}
\providecommand{\MR}{\relax\ifhmode\unskip\space\fi MR }
\providecommand{\MRhref}[2]{%
  \href{http://www.ams.org/mathscinet-getitem?mr=#1}{#2}
}
\providecommand{\href}[2]{#2}
\begin{thebibliography}{10}

\bibitem{BBD}
A.~A. Be{\u\i}linson, J.~Bernstein, and P.~Deligne, \emph{Faisceaux pervers},
  Analysis and topology on singular spaces, {I} ({L}uminy, 1981), Ast\'erisque,
  vol. 100, Soc. Math. France, Paris, 1982, pp.~5--171. \MR{751966}

\bibitem{BLR}
Siegfried Bosch, Werner L\"utkebohmert, and Michel Raynaud, \emph{N\'eron
  models}, Ergebnisse der Mathematik und ihrer Grenzgebiete (3) [Results in
  Mathematics and Related Areas (3)], vol.~21, Springer-Verlag, Berlin, 1990.
  \MR{1045822}

\bibitem{BKS}
H.~M. Bui, J.~P. Keating, and D.~J. Smith, \emph{On the variance of sums of
  arithmetic functions over primes in short intervals and pair correlation for
  {$L$}-functions in the {S}elberg class}, J. Lond. Math. Soc. (2) \textbf{94}
  (2016), no.~1, 161--185. \MR{3532168}

\bibitem{C}
Tsz~Ho Chan, \emph{More precise pair correlation of zeros and primes in short
  intervals}, J. London Math. Soc. (2) \textbf{68} (2003), no.~3, 579--598.
  \MR{2009438}

\bibitem{handbook:ecc}
Henri Cohen, Gerhard Frey, Roberto Avanzi, Christophe Doche, Tanja Lange, Kim
  Nguyen, and Frederik Vercauteren (eds.), \emph{Handbook of elliptic and
  hyperelliptic curve cryptography}, Discrete Mathematics and its Applications
  (Boca Raton), Chapman \& Hall/CRC, Boca Raton, FL, 2006. \MR{2162716}

\bibitem{CFZ}
Brian Conrey, David~W. Farmer, and Martin~R. Zirnbauer, \emph{Autocorrelation
  of ratios of {$L$}-functions}, Commun. Number Theory Phys. \textbf{2} (2008),
  no.~3, 593--636. \MR{2482944}

\bibitem{CS}
J.~B. Conrey and N.~C. Snaith, \emph{Applications of the {$L$}-functions ratios
  conjectures}, Proc. Lond. Math. Soc. (3) \textbf{94} (2007), no.~3, 594--646.
  \MR{2325314}

\bibitem{CurtisReiner}
Charles~W. Curtis and Irving Reiner, \emph{Representation theory of finite
  groups and associative algebras}, AMS Chelsea Publishing, Providence, RI,
  2006, Reprint of the 1962 original. \MR{2215618}

\bibitem{SGA4.5}
P.~Deligne, \emph{Cohomologie \'etale}, Lecture Notes in Mathematics, Vol. 569,
  Springer-Verlag, Berlin-New York, 1977, S\'eminaire de G\'eom\'etrie
  Alg\'ebrique du Bois-Marie SGA 4$\frac{1}{2}$, Avec la collaboration de J. F.
  Boutot, A. Grothendieck, L. Illusie et J. L. Verdier. \MR{0463174}

\bibitem{Deligne:WeilII}
Pierre Deligne, \emph{La conjecture de {W}eil. {II}}, Inst. Hautes \'Etudes
  Sci. Publ. Math. (1980), no.~52, 137--252. \MR{601520}

\bibitem{DM}
Pierre Deligne, James~S. Milne, Arthur Ogus, and Kuang-yen Shih, \emph{Hodge
  cycles, motives, and {S}himura varieties}, Lecture Notes in Mathematics, vol.
  900, Springer-Verlag, Berlin-New York, 1982. \MR{654325}

\bibitem{DE}
Persi Diaconis and Steven~N. Evans, \emph{Linear functionals of eigenvalues of
  random matrices}, Trans. Amer. Math. Soc. \textbf{353} (2001), no.~7,
  2615--2633. \MR{1828463}

\bibitem{DS}
Persi Diaconis and Mehrdad Shahshahani, \emph{On the eigenvalues of random
  matrices}, J. Appl. Probab. \textbf{31A} (1994), 49--62, Studies in applied
  probability. \MR{1274717}

\bibitem{FG}
J.~B. Friedlander and D.~A. Goldston, \emph{Variance of distribution of primes
  in residue classes}, Quart. J. Math. Oxford Ser. (2) \textbf{47} (1996),
  no.~187, 313--336. \MR{1412558}

\bibitem{GL}
Ofer Gabber and Fran\c{c}ois Loeser, \emph{Faisceaux pervers {$l$}-adiques sur
  un tore}, Duke Math. J. \textbf{83} (1996), no.~3, 501--606. \MR{1390656}

\bibitem{GM}
Daniel~A. Goldston and Hugh~L. Montgomery, \emph{Pair correlation of zeros and
  primes in short intervals}, Analytic number theory and {D}iophantine problems
  ({S}tillwater, {OK}, 1984), Progr. Math., vol.~70, Birkh\"auser Boston,
  Boston, MA, 1987, pp.~183--203. \MR{1018376}

\bibitem{SGA7}
Alexander Grothendieck, \emph{Groupes de monodromie en g\'eom\'etrie
  alg\'ebrique. {I}}, Lecture Notes in Mathematics, Vol. 288, Springer-Verlag,
  Berlin-New York, 1972, S{\'e}minaire de G{\'e}om{\'e}trie Alg{\'e}brique du
  Bois-Marie 1967--1969 (SGA 7 I), Dirig{\'e} par A. Grothendieck. Avec la
  collaboration de M. Raynaud et D. S. Rim. \MR{0354656}

\bibitem{Hall:BM}
Chris Hall, \emph{Big symplectic or orthogonal monodromy modulo {$l$}}, Duke
  Math. J. \textbf{141} (2008), no.~1, 179--203. \MR{2372151 (2008m:11112)}

\bibitem{Hartshorne}
Robin Hartshorne, \emph{Algebraic geometry}, Springer-Verlag, New
  York-Heidelberg, 1977, Graduate Texts in Mathematics, No. 52. \MR{0463157}

\bibitem{HooleyII}
C.~Hooley, \emph{On the {B}arban-{D}avenport-{H}alberstam theorem. {II}}, J.
  London Math. Soc. (2) \textbf{9} (1974/75), 625--636. \MR{0382203}

\bibitem{HooleyICM}
\bysame, \emph{The distribution of sequences in arithmetic progressions},
  (1975), 357--364. \MR{0498441}

\bibitem{HooleyI}
Christopher Hooley, \emph{On the {B}arban-{D}avenport-{H}alberstam theorem.
  {I}}, J. Reine Angew. Math. \textbf{274/275} (1975), 206--223, Collection of
  articles dedicated to Helmut Hasse on his seventy-fifth birthday, III.
  \MR{0382202}

\bibitem{Katz:GKM}
Nicholas~M. Katz, \emph{Gauss sums, {K}loosterman sums, and monodromy groups},
  Annals of Mathematics Studies, vol. 116, Princeton University Press,
  Princeton, NJ, 1988. \MR{955052}

\bibitem{Katz:RLS}
\bysame, \emph{Rigid local systems}, Annals of Mathematics Studies, vol. 139,
  Princeton University Press, Princeton, NJ, 1996.

\bibitem{Katz:TLFM}
\bysame, \emph{Twisted {$L$}-functions and monodromy}, Annals of Mathematics
  Studies, vol. 150, Princeton University Press, Princeton, NJ, 2002.

\bibitem{Katz:SC}
\bysame, \emph{A semicontinuity result for monodromy under degeneration}, Forum
  Math. \textbf{15} (2003), no.~2, 191--200. \MR{1956963}

\bibitem{Katz:CE}
\bysame, \emph{Convolution and equidistribution}, Annals of Mathematics
  Studies, vol. 180, Princeton University Press, Princeton, NJ, 2012, Sato-Tate
  theorems for finite-field Mellin transforms. \MR{2850079}

\bibitem{Katz:QKR}
\bysame, \emph{On a question of {K}eating and {R}udnick about primitive
  {D}irichlet characters with squarefree conductor}, Int. Math. Res. Not. IMRN
  (2013), no.~14, 3221--3249. \MR{3085758}

\bibitem{KS}
Nicholas~M. Katz and Peter Sarnak, \emph{Random matrices, {F}robenius
  eigenvalues, and monodromy}, American Mathematical Society Colloquium
  Publications, vol.~45, American Mathematical Society, Providence, RI, 1999.
  \MR{1659828 (2000b:11070)}

\bibitem{KRRR}
Jonathan Keating, Brad Rodgers, Edva Roditty-Gershon, and Zeev Rudnick,
  \emph{Sums of divisor functions in $\mathbb{F}_q[t]$ and matrix integrals},
  arXiv:1504.07804.

\bibitem{KRII}
Jonathan Keating and Zeev Rudnick, \emph{Squarefree polynomials and {M}\"obius
  values in short intervals and arithmetic progressions}, Algebra Number Theory
  \textbf{10} (2016), no.~2, 375--420. \MR{3477745}

\bibitem{KR-G}
Jonathan~P. Keating and Edva Roditty-Gershon, \emph{Arithmetic correlations
  over large finite fields}, Int. Math. Res. Not. IMRN (2016), no.~3, 860--874.
  \MR{3493436}

\bibitem{KR}
Jonathan~P. Keating and Ze{\'e}v Rudnick, \emph{The variance of the number of
  prime polynomials in short intervals and in residue classes}, Int. Math. Res.
  Not. IMRN (2014), no.~1, 259--288. \MR{3158533}

\bibitem{LPZ}
Alessandro Languasco, Alberto Perelli, and Alessandro Zaccagnini,
  \emph{Explicit relations between pair correlation of zeros and primes in
  short intervals}, J. Math. Anal. Appl. \textbf{394} (2012), no.~2, 761--771.
  \MR{2927496}

\bibitem{Milne}
James~S. Milne, \emph{\'{E}tale cohomology}, Princeton Mathematical Series,
  vol.~33, Princeton University Press, Princeton, N.J., 1980. \MR{559531}

\bibitem{Montgomery}
H.~L. Montgomery, \emph{Primes in arithmetic progressions}, Michigan Math. J.
  \textbf{17} (1970), 33--39. \MR{0257005}

\bibitem{M}
\bysame, \emph{The pair correlation of zeros of the zeta function},  (1973),
  181--193. \MR{0337821}

\bibitem{MS}
Hugh~L. Montgomery and K.~Soundararajan, \emph{Primes in short intervals},
  Comm. Math. Phys. \textbf{252} (2004), no.~1-3, 589--617. \MR{2104891}

\bibitem{Raynaud}
Michel Raynaud, \emph{Caract\'eristique d'{E}uler-{P}oincar\'e d'un faisceau et
  cohomologie des vari\'et\'es ab\'eliennes}, S\'eminaire {B}ourbaki, {V}ol.\
  9, Soc. Math. France, Paris, 1995, pp.~Exp.\ No.\ 286, 129--147. \MR{1608794}

\bibitem{Rod}
Brad Rodgers, \emph{Arithmetic functions in short intervals and the symmetric
  group}, arXiv:1609.02967.

\bibitem{R-G}
E.~Roditty-Gershon, \emph{Square-full polynomials in short intervals and in
  arithmetic progressions}, Res. Number Theory \textbf{3} (2017), 3:3.
  \MR{3597214}

\bibitem{Rosen}
Michael Rosen, \emph{Number theory in function fields}, Graduate Texts in
  Mathematics, vol. 210, Springer-Verlag, New York, 2002. \MR{1876657}

\bibitem{Rud}
Zeev Rudnick, \emph{Some problems in analytic number theory for polynomials
  over a finite field},  (2014).

\bibitem{SerreTate}
Jean-Pierre Serre and John Tate, \emph{Good reduction of abelian varieties},
  Ann. of Math. (2) \textbf{88} (1968), 492--517. \MR{0236190 (38 \#4488)}

\bibitem{Zarhin}
Yu.~G. Zarhin, \emph{Linear simple {L}ie algebras and ranks of operators}, The
  {G}rothendieck {F}estschrift, {V}ol.\ {III}, Progr. Math., vol.~88,
  Birkh\"auser Boston, Boston, MA, 1990, pp.~481--495. \MR{1106920}

\end{thebibliography}


\end{document}